\documentclass{amsart}

\usepackage{bbold}
\usepackage{amsmath,amssymb,amsbsy,amscd,amsthm,mathrsfs,ulem,
mathtools,multirow,enumitem,
stmaryrd} 

\usepackage[all]{xy}

\theoremstyle{plain}
\newtheorem{thm}{Theorem}[section]
\newtheorem{lem}[thm]{Lemma}
\newtheorem{prop}[thm]{Proposition}
\newtheorem{cor}[thm]{Corollary}

\newtheorem{notation}{Notation}

\newtheorem{llem}{Lemma}

\newtheorem{claim}{Claim}

\theoremstyle{definition}
\newtheorem{definition}[thm]{Definition}

\newtheorem{example}[thm]{Example}

\newtheorem{rem}[thm]{Remark}

\newtheorem*{klaim}{Claim}

\newcommand{\maru}[1]{{\ooalign{\small\hfil#1\/\hfil\crcr
\raise.167ex\hbox{\mathhexbox20D}}}}

\newcommand{\fm }{{\mathfrak{m}}}

\newcommand{\A}{\mathcal{A}}

\newcommand{\E }{{\mathop{\mathcal{E}}\nolimits}}

\renewcommand{\rightsquigarrow }{\mathbin{\sim}}

\newcommand{\gr }{\mathop{\rm gr}\nolimits}
\newcommand{\Der }{\mathop{\rm Der}\nolimits}

\newcommand{\zs}{\{ 0\} }
\newcommand{\sm}{\setminus}

\newcommand{\Aut }{\mathop{\mathrm{Aut}}\nolimits}
\newcommand{\T }{\mathrm{T}}
\newcommand{\D }{\mathrm{D}}

\newcommand{\id}{{\rm id}}

\newcommand{\nd}{\noindent}
\newcommand{\ol}{\overline}

\newcommand{\F}{{\bf F}}
\newcommand{\R}{{\bf R}}
\newcommand{\Q}{{\bf Q}}
\newcommand{\Z}{{\bf Z}}
\newcommand{\N}{{\bf N}}
\newcommand{\ep}{\epsilon}
\newcommand{\x}{{\boldsymbol x}}
\newcommand{\y}{{\boldsymbol y}}
\newcommand{\w}{{\bf w}}
\newcommand{\bv}{{\bf v}}

\newcommand{\ba}{{\boldsymbol a}}

\newcommand{\bi}{{\boldsymbol i}}
\newcommand{\bj}{{\boldsymbol j}}
\newcommand{\bk}{{\boldsymbol k}}

\newcommand{\bu}{{\boldsymbol u}}

\newcommand{\ff}{{\bf f}}
\newcommand{\bg}{{\bf g}}
\newcommand{\bh}{{\bf h}}

\newcommand{\kx}{k[\x]}
\newcommand{\kxr}{k(\x)}

\newcommand{\Rx}{R[\x]}

\newcommand{\Ryy}{R[[\y]]}

\newcommand{\degw  }{\deg}
\newcommand{\degww  }{\mathop{\rm deg}_{\w }\nolimits}

\newcommand{\supp }{\mathop{\rm supp}\nolimits}

\newcommand{\sym}{\mathfrak{S}}

\begin{document}

\title[Nagata's conjecture in positive characteristic]{
Nagata's conjecture on a polynomial automorphism in positive characteristic}

\author{Shigeru Kuroda}
\address{Department of Mathematical Sciences\\
Tokyo Metropolitan University\\
1-1 Minami-Osawa, Hachioji, Tokyo, 192-0397, Japan}
\email{kuroda@tmu.ac.jp}
\thanks{This work is partly supported by JSPS KAKENHI
Grant Numbers 22K03273, 18K03219, 15K04826 and 24740022.}

\subjclass[2020]{Primary 14R10, Secondary 14R20,14G17,13F20}

\keywords{polynomial automorphism, 
tame automorphism, wild automorphism, 
Nagata's conjecture, 
Shestakov-Umirbaev reduction.}

\maketitle

\begin{abstract}
An automorphism of the polynomial ring $k[x_1,\ldots ,x_n]$ 
over a field $k$ is said to be $\textit{tame}$ if it can be 
obtained by composing 
affine automorphisms and elementary automorphisms, 
and {\it wild} otherwise. 
Jung and van der Kulk showed that 
every automorphism of $k[x_1,x_2]$ is tame. 
In 1972, 
Nagata conjectured that a certain automorphism of $k[x_1,x_2,x_3]$ 
is wild. 
In 2003, 
Shestakov and Umirbaev proved this conjecture for $\mathop{\mathrm{char}}k=0$. 
The purpose of this paper is to 
prove the conjecture for $\mathop{\mathrm{char}}k\ge 7$. 
This is the first time that the existence of 
a wild automorphism has been confirmed 
in positive characteristic. 
\end{abstract}

\section{Introduction}\label{sect:intro}
\setcounter{equation}{0}

Let $k$ be a field of characteristic $p\ge 0$, 
$\kx =k[x_1,\ldots ,x_n]$ the polynomial ring in $n$ variables over $k$, 
and $\Aut _k\kx $ the automorphism group of the $k$-algebra $\kx $. 
Write $F\in \Aut _k\kx $ as $F=(f_1,\ldots ,f_n)$, 
where $f_i:=F(x_i)$ for $i=1,\ldots ,n$. 
Then, 
$F$ is said to be {\it affine} if 
$f_1,\ldots ,f_n$ are linear polynomials, 
{\it elementary} if 
$F=(x_1,\ldots ,x_{l-1},x_l+h,x_{l+1},\ldots ,x_n)$ 
for some $1\le l\le n$ 
and $h\in k[\{ x_1,\ldots ,x_n\} \sm \{ x_l\} ]$, 
{\it tame} if $F\in \T _n(k)$, 
and {\it wild} if $F\not\in \T _n(k)$, 
where 
$\T _n(k)$ is the subgroup of $\Aut _k\kx $ 
generated by all affine automorphisms and elementary automorphisms.

In 1942, 
Jung~\cite{Jung} proved that $\Aut _kk[x_1,x_2]=\T _2(k)$ for $p=0$, 
which was later extended to the case $p\ge 0$ by van der Kulk~\cite{Kulk} 
(see also \cite{Makar-Limanov}). 
In 1972, 
Nagata~\cite{Nagata} conjectured that 
$\Aut _kk[x_1,x_2,x_3]\ne \T _3(k)$, 
and presented the following 
as a potential wild automorphism of $k[x_1,x_2,x_3]$: 
\begin{equation}\label{eq:Nagata}
(x_1-2(x_1x_3+x_2^2)x_2-(x_1x_3+x_2^2)^2x_3,
x_2+(x_1x_3+x_2^2)x_3,
x_3). 
\end{equation}
In 2003, 
Shestakov-Umirbaev \cite{SUpoisson,SU} 
settled this 
conjecture in the affirmative for $p=0$. 
This was the first time that the existence of 
a wild automorphism was confirmed. 
The purpose of this paper is to prove 
Nagata's conjecture for $p\ge 7$ 
and establishes the existence of 
a wild automorphism in positive characteristic. 
The question whether 
$\Aut _k\kx =\T _n(k)$ holds for $n\ge 4$ remains open, 
as $F\not\in \T _n(k)$ does not imply 
$(F,x_{n+1},\ldots ,x_{n+r})\not\in \T _{n+r}(k)$ for $r\ge 1$ 
(see, e.g., \cite{BEW,EH,Smith}).

Shestakov and Umirbaev proved 
the wildness of (\ref{eq:Nagata}) using the criterion in \cite{SU}, 
which relies on 
a polynomial degree inequality from \cite{SUpoisson}. 
Later, we generalized the inequality in \cite{SUineq} 
and applied it in \cite{tame3} to improve the criterion 
(see also \cite[Chapter~1]{JC}). 
The assumption $p=0$ plays a vital role in these results. 
Extending these inequalities to $p>0$ 
has remained an open problem for nearly two decades. 
In this paper, 
we resolve this issue by generalizing the inequality of \cite{SUineq} 
to the case $p>0$ (Theorem~\ref{thm:main}). 
This enables us to extend 
the criterion of \cite{tame3} 
to the case $p\ge 7$ (Theorem~\ref{thm:mainSU}), 
yielding the conclusion that 
(\ref{eq:Nagata}) is wild when $p\ge 7$.

The rest of this section provides 
an overview of the main results of this paper.

\begin{notation}\label{notation:intro}\rm 
(i) 
$\N $ denotes the set of all nonnegative integers, 
and $\N _+:=\N \sm \zs $.

\nd (ii) 
For $f,g\in \kx \sm \zs $, 
we write $f\approx g$ 
(resp.\ $f\not\approx g$) if 
$f/g\in k^*$ 
(resp.\ $f/g\not\in k^*$).

\nd (iii) 
$F=(f_1,\ldots ,f_n)$, $G=(g_1,\ldots ,g_n)$, etc., 
denote elements of $\kx ^n$ 
unless otherwise specified. 
For a permutation $\sigma \in \sym _n$, 
we define 
$F_\sigma :=(f_{\sigma (1)},\ldots ,f_{\sigma (n)})$.

\nd (iv) 
Let $T$ be a variable. 
If $P\in \kx [T]$, 
then $P^{(l)}$ denotes the $l$th order derivative for $l\in \N $, 
and $P(g)$ denotes the result of the substitution 
$T\mapsto g$ for $g\in \kx [T]$.

\nd (v) 
When $p$ is a prime, 
we define $v_p(m):=\max \{ d\in \N \mid m\in p^d\Z\} $ 
for $m\in \N _+$. 
\end{notation}

\subsection{Weighted degree and differential}\label{subsetc:w-degree}

Let $\Gamma $ be a 
{\it totally ordered $\Q $-vector space}, 
i.e., a $\Q $-vector space 
equipped with a total ordering such that 
$\alpha \le \beta $ implies 
$\alpha +\gamma \le \beta +\gamma $ 
for all $\alpha ,\beta ,\gamma \in \Gamma $ 
(e,g., $\R $ with the usual ordering). 
We fix a {\it weight} 
$\w =(w_1,\ldots ,w_n)\in (\Gamma _+)^n$, 
where $\Gamma _+:=\{ \gamma \in \Gamma \mid \gamma >0\} $. 
For each $\bi =(i_1,\ldots ,i_n)\in \N ^n$, 
we define 
$\bi \cdot \w :=\sum _{l=1}^ni_lw_l$ 
and $\x ^{\bi }:=x_1^{i_1}\cdots x_n^{i_n}$.

Now, 
let $f=\sum _{\bi \in \N ^n}c_{\bi }\x ^{\bi }\in \kx $, 
where $c_{\bi }\in k$. 
We define the $\w $-{\it degree} of $f$ by 
$\degww  f:=\max \{ \bi \cdot \w \mid \bi \in \supp f\} $, 
where $\supp f:=\{ \bi \in \N ^n\mid c_{\bi }\ne 0\} $ and 
$\max \emptyset :=-\infty $. 
We define $f^{\w }:=\sum _{\bi \in I}c_{\bi }\x ^{\bi }$, 
where $I:=\{ \bi \in \N ^n\mid \bi \cdot \w =\degww f\} $. 
Note that $\degww fg=\degww f+\degww g$ and 
$(fg)^{\w }=f^{\w }g^{\w }$ hold for all $f,g\in \kx $.

When $\Gamma =\R $, 
the $\mathbb{1}:=(1,\ldots ,1)$-degree is the same as the total degree.

\begin{rem}\label{rem:principle}\rm 
For a finite tuple $(f_i)_{i=1}^s\in \kx ^s\sm \zs $, 
we set 
$$
\delta :=\max \{ \degww  f_i\mid i=1,\ldots ,s\} 
\text{ \ and \ }
K:=\{ i \mid i=1,\ldots ,s,\ \degww  f_i=\delta \} . 
$$
Then, 
we have 
``$h:=\sum _{i\in K}f_i^{\w }\ne 0
\Leftrightarrow \degww  (\sum _{i=1}^sf_i)=\delta 
\Rightarrow 
(\sum _{i=1}^sf_i)^{\w }=h$". 
\end{rem}

We say that $\w $ is {\it independent} 
if $w_1,\ldots ,w_n$ are linearly independent over $\Q $ 
(e.g., $\w =(1,\sqrt{2},\sqrt{3})$ with $\Gamma =\R $). 
In this case, 
$f^{\w }$ is a monomial for all $f\in \kx $, 
and 
\begin{equation}\label{eq:approx}
\degww  f=\degww  g\iff f^{\w }\approx g^{\w } 
\end{equation}
holds for all $f,g\in \kx \sm \zs $.

Next, let $\Omega _{\kx /k}$ be the module of 
differentials of $\kx $ over $k$, 
and $\bigwedge ^s\Omega _{\kx /k}$ 
the $s$th exterior power of the $\kx $-module $\Omega _{\kx /k}$, 
where $s\ge 1$. 
We denote by $df$ the differential of $f\in \kx $, 
which satisfies the relation $df=\sum _{i=1}^n(\partial f/\partial x_i)dx_i$.

We use the following notation:

\nd $\bullet $ 
For 
$\bi =(i_1,\ldots ,i_s)
\in \Lambda (n,s):=\{ (i_1,\ldots ,i_s)\in \N ^s\mid 
1\le i_1<\cdots <i_s\le n\} $, 
we define 
$d\x _{\bi }:=dx_{i_1}\wedge \cdots \wedge dx_{i_s}$ 
and $\w _{\bi }:=(w_{i_1},\ldots ,w_{i_s})$.

\nd $\bullet $ 
For $\bv =(v_1,\ldots ,v_s)\in \Gamma ^s$, 
we define $|\bv |:=\sum _{i=1}^sv_i$.

\nd $\bullet $ 
For $\ff =(f_1,\ldots ,f_s)\in \kx ^s$, 
we define 
$\degww \ff :=\sum _{i=1}^s\degww f_i$ and 
\begin{equation}\label{eq:dff}
d\ff :=df_1\wedge \cdots \wedge df_s
=\sum _{\bi \in \Lambda (n,s)}
\left|\frac{\partial \ff }{\partial \x _{\bi }}\right|
d\x _{\bi }, \text{ where }
\left|\frac{\partial \ff }{\partial \x _{\bi }}\right|
:=\det \left(
\dfrac{\partial f_j}{\partial x_{i_l}}
\right)_{j,l}. 
\end{equation}

Now, 
let $\omega \in \bigwedge ^s\Omega _{\kx /k}$. 
Then, 
we can uniquely express 
$\omega 
=\sum _{\bi \in \Lambda (n,s)}f_{\bi }d\x _{\bi }$, 
where $f_{\bi }\in \kx $. 
We define 
$\degww  f_{\bi }d\x _{\bi }:=\degww  f_{\bi }+|\w _{\bi }|$ 
for each $\bi \in \Lambda (n,s)$, 
and 
\begin{equation}\label{eq:deg omega}
\degww  \omega :=\max \{ \degww  f_{\bi }d\x _{\bi }
\mid \bi \in \Lambda (n,s)\} . 
\end{equation}
Note that $\degww  \omega >0$ 
unless $\omega =0$, 
since $|\w _{\bi }|>0$ for all $\bi $. 
Moreover, we have

\nd 
(1) $\degww f\omega =\degww f+\degww \omega $ 
for all $f\in \kx $ 
and $\omega \in \bigwedge ^s\Omega _{\kx /k}$;

\nd 
(2) $\degww  (\omega +\omega ')\le \max \{ 
\degww  \omega ,\degww  \omega '\} $ 
for all 
$\omega ,\omega '\in \bigwedge ^s\Omega _{\kx /k}$; 

\nd 
(3) 
$\degww  \omega _1\wedge \omega _2
\le \degww  \omega _1+\degww  \omega _2$ 
for all $\omega _1\in \bigwedge ^{s_1}\Omega _{\kx /k}$ and 
$\omega _2\in \bigwedge ^{s_2}\Omega _{\kx /k}$, 
where $s_1,s_2\ge 1$.

\begin{rem}\label{rem:independence}\rm 
Let $\ff =(f_1,\ldots ,f_s)\in \kx ^s$ and $f\in \kx $. 

\nd (i) 
If $d\ff \ne 0$, then 
$f_1,\ldots ,f_s$ are 
algebraically independent over $k$ 
(see \cite[\S 26]{Matsumura}).

\nd (ii) 
By definition, we have 
$\deg df=\max \{ \degww \partial f/\partial x_i+w_i
\mid i=1,\ldots ,n\} \le \deg f$. 
If $p=0$, 
then equality holds unless $f\in k$, 
which is not 
always the case if $p>0$.

\nd (iii) 
By (3) and (ii), 
we have 
$\degww d\ff \le 
\degww f_1+\degww df_2\wedge \cdots \wedge df_s\le 
\degww \ff $.

\nd (iv) $\degww \ff = \degww d\ff $ holds if and only if 
$df_1^{\w }\wedge \cdots \wedge df_s^{\w }\ne 0$. 
Hence, 
if $f_1^{\w },\ldots ,f_s^{\w }$ are algebraically dependent over $k$, 
then we have $\degww \ff >\degww d\ff $ by (i). 
\end{rem}

Let $F\in \Aut _k\kx $. 
Then, 
we have $|\partial F/\partial \x |\in k^*$. 
Since $dF =|\partial F/\partial \x |d\x $ by (\ref{eq:dff}), 
we obtain $\degww  dF=|\w |$. 
Hence, 
it follows from Remark~\ref{rem:independence} (iii) that 
\begin{equation}\label{eq:deg F}
\degww  F:=\degww f_1+\cdots +\degww f_n\ge |\w |. 
\end{equation}

\subsection{Wildness criterion}\label{subsetc:criterion}

The composition of $F,G\in \Aut _k\kx $ is given by 
$$
FG=(g_1(f_1,\ldots ,f_n),\ldots ,g_n(f_1,\ldots ,f_n)),
$$
since $x_i\stackrel{G}{\mapsto }
g_i\stackrel{F}{\mapsto }
F(g_i)=g_i(F(x_1),\ldots ,F(x_n))
=g_i(f_1,\ldots ,f_n)$ 
for each $i$. 
For example, 
if $G=(x_1+x_2^7x_3,x_2,\ldots ,x_n)$, 
then we have $FG=(f_1+f_2^7f_3,f_2,\ldots ,f_n)$.

\begin{definition}\label{defn:ER}\rm
We say that $F\in \Aut _k\kx $ admits an 
{\it elementary reduction} 
if there exists an elementary automorphism $E\in \Aut _k\kx $ 
such that $\degww  FE<\degww  F$. 
\end{definition}

By (\ref{eq:deg F}), 
$\deg _{\mathbb{1}}F\ge |\mathbb{1}|=n$ 
holds for all $F\in \Aut _k\kx $, 
with equality holding if and only if $F$ is affine. 
If $F\in \Aut _kk[x_1,x_2]$ is not affine, 
then $F$ admits an elementary reduction 
with $\w =\mathbb{1}$ 
(see, e.g., \cite[\S ~6.9]{Cohn}). 
This easily yields $\Aut _kk[x_1,x_2]=\T _2(k)$. 
When $p=0$, 
Shestakov-Umirbaev~\cite[Example 1]{SU} discovered $F\in \T _3(k)$ 
that is not affine 
and does not admit an elementary reduction with $\w =\mathbb{1}$. 
This type of $F$, however, 
satisfies $\deg_{\mathbb{1}}FE_1E_2<\deg_{\mathbb{1}}F$ 
for some elementary automorphisms $E_1$ and $E_2$, 
and is said to admit a {\it reduction of type I}. 
They also defined the concepts of 
{\it reductions of types II, III, and IV} 
due to theoretical necessity, 
and proved that if $F\in \T _3(k)$ is not affine, 
then $F$ admits an elementary reduction 
with $\w =\mathbb{1}$ 
or a reduction of type 
I, II, III, or IV 
(see \cite[Theorem 1]{SU}). 
They concluded that (\ref{eq:Nagata}) is wild, 
as it fails to admit any of these reductions.

In the rest of \S ~\ref{subsetc:criterion}, 
we assume that $n=3$. 
In our previous work, 
we defined the following concept, which reformulates 
the reductions of types I, II, and III.

\begin{definition}[{\cite{tame3}}]\label{defn:SUC}\rm
(i) We call $(F,G)\in (\Aut _k\kx )^2$ 
a {\it Shestakov-Umirbaev pair} 
({\it SU pair}) 
if the following conditions hold:

\renewcommand{\theenumi}{\roman{enumi}}

\nd 
\begin{enumerate}[leftmargin=12mm]

\item [(SU1)]
$g_1\in f_1+kf_3^2+kf_3$, 
$g_2\in f_2+kf_3$, and $g_3\in f_3+k[g_1,g_2]$; 

\item [(SU2)] $\deg _{\w }f_1\leq \deg _{\w }g_1$ 
and $\deg _{\w }f_2=\deg _{\w }g_2$; 

\item [(SU3)] $(g_1^{\w })^2\approx (g_2^{\w })^s$ 
for some odd $s\geq 3$; 

\item [(SU4)] $\deg _{\w }f_3\leq \deg _{\w }g_1$ 
and $f_3^{\w }\not\in k[g_1^{\w }, g_2^{\w }]$; 

\item [(SU5)] 
$\deg _{\w }f_3>\deg _{\w }g_3$; 

\item [(SU6)] 
$\deg _{\w }g_3<\deg _{\w }g_1-\deg _{\w }g_2 
+\deg _{\w }dg_1\wedge dg_2$. 

\end{enumerate}

\nd (ii) 
We say that $F\in \Aut _k\kx $ admits a 
{\it Shestakov-Umirbaev reduction} 
({\it SU reduction}) 
if there exist $\sigma \in \sym _3$ and $G\in \Aut _k\kx $ such that 
$(F_{\sigma },G)$ is an SU pair. 
\end{definition}

The following theorem refines and generalizes 
the Shestakov-Umirbaev criterion \cite[Theorem~1]{SU} 
mentioned above.

\begin{thm}[{\cite[Theorem 2.1]{tame3}}]\label{thm:ptame3}
Assume that $p=0$. 
Let $F\in \T _3(k)$ be such that $\deg _{\w }F>|\w |$. 
Then, $F$ admits an elementary reduction or an SU reduction. 
\end{thm}

The main result of this paper is as follows.

\begin{thm}\label{thm:mainSU}
Assume that $p\ge 7$. 
Let $F\in \T _3(k)$ be such that $\deg _{\w }F>|\w |$. 
Then, $F$ admits an elementary reduction or 
an SU reduction. 
\end{thm}

Using Theorem~\ref{thm:ptame3}, 
we proved that various families of 
elements of $\Aut _k\kx $ 
are wild when $p=0$ 
(see \cite{wild3, Sugaku, GABA}). 
Theorem~\ref{thm:mainSU} implies that 
they are also wild when $p\ge 7$ 
if their constructions are independent of the characteristic. 
The wildness of (\ref{eq:Nagata}) for $p\ge 7$ can be verified 
by applying Theorem~\ref{thm:ptame3} with an independent 
weight $\w $, 
exactly as in the case $p=0$ \cite[p.\ 79]{tame3} 
(see also \cite[p.\,11, Exercise 11]{JC}). 
When $\w $ is independent, 
the proof of Theorem~\ref{thm:mainSU} 
simplifies due to (\ref{eq:approx}) 
(see \S\S ~\ref{subsect:kouhan} and~\ref{subsect:pf of lem 3}).

\subsection{Shestakov-Umirbaev type inequality}\label{subsect:SUineq}

Take $g\in \kx \sm k$ and  
\begin{equation}\label{eq:Pick any P}
P=\sum _{i\ge 0}p_iT^i\in \kx [T]\sm \zs , 
\text{ where }p_i\in \kx . 
\end{equation}
Set $\w _g:=(\w ,\degww  g)\in (\Gamma _+)^{n+1}$ 
and define the $\w _g$-degree 
of $P$ by regarding $T$ as the $(n+1)$-st variable. 
Then, 
since $\supp P=\bigsqcup _{i\ge 0}\supp p_iT^i$, 
we have 
\begin{equation}\label{eq:degw_gP}
\deg _{\w _g}P
=\max \{ \deg _{\w _g}p_iT^i\mid i\in \N \} 
=\max \{ \degww  p_ig^i\mid i\in \N \} , 
\end{equation}
which is at least 
the $\w $-degree of $P(g)=\sum _{i\ge 0}p_ig^i$. 
We also have 
\begin{equation}\label{eq:P(T)^w_g1}
P^{\w _g}=\sum _{i\in I}p_i^{\w }T^i, 
\quad \text{where}\quad 
I:=\{ i\in \N \mid \degww  p_ig^i=\deg _{\w _g}P\} . 
\end{equation}
We now define $m_{\w }^g(P)\in \N $ to 
be the multiplicity of 
the roof $g^{\w }$ in $P^{\w _g}$, 
i.e., the number such that 
\begin{equation}\label{eq:P(T)^w_g2}
P^{\w _g}=(T-g^{\w })^{m_{\w }^g(P)}P_0
\text{ \ for some \ $P_0\in \kx [T]$ \ with \ $P_0(g^{\w })\ne 0$}.
\end{equation}

\begin{example}\label{ex:setting}
Let $\w :=\mathbb{1}$, 
$g:=x_1^3+x_2$, 
and $P:=T^4-2(f^3+f^2)T^2+f^4T+f^6$, 
where $f:=x_1^2+x_2$. 
Then, 
we have $\w _g=(\mathbb{1},3)$, 
$\deg _{\w _g}P=12$, $I=\{ 4,2,0\} $, 
and 
$P^{\w _g}=T^4-2x_1^6T^2+x_1^{12}=(T^2-x_1^6)^2$. 
Since $g_1^{\w }=x_1^3$, 
we see that $m_{\w }^g(P)=4$ if $p=2$, 
and $m_{\w }^g(P)=2$ otherwise. 
\end{example}

\begin{rem}\label{rem:m(P)=0}\rm 
(i) 
``$m_{\w }^g(P)=0\Leftrightarrow P^{\w _g}(g^{\w })\ne 0$" 
holds by definition.

\nd (ii) 
``$P^{\w _g}(g^{\w })\ne 0
\Leftrightarrow \degww  P(g)=\deg _{\w _g}P
\Rightarrow P(g)^{\w }=P^{\w _g}(g^{\w })$" 
holds. 
In fact, 
we have $P^{\w _g}(g^{\w })
=\sum _{i\in I}p_i^{\w }(g^{\w })^i
=\sum _{i\in I}(p_ig^i)^{\w }$ by (\ref{eq:P(T)^w_g1}). 
Hence, 
in view of (\ref{eq:degw_gP}), 
the assertion follows from 
Remark~\ref{rem:principle} applied to $(p_ig^i)_{i\ge 0} $. 
\end{rem}

Now, 
let $\ff =(f_1,\ldots ,f_r)\in \kx ^r$ and $g\in \kx $ 
be such that $d\ff \wedge dg\ne 0$. 
Then, 
$g$ is transcendental over $k[\ff ]:=k[f_1,\ldots ,f_r]$ 
by Remark~\ref{rem:independence} (i). 
We define 
\begin{equation}\label{eq:Lambdas}
\begin{aligned}
\Lambda _0&:=
\degww  d\ff +\degww  g-\degww  d\ff \wedge dg,\\ 
\Lambda _1&:=
\degww \ff +\degww g-\degww  d\ff \wedge dg.
\end{aligned}
\end{equation}
By Remark~\ref{rem:independence} (iii), 
we have $\Lambda _1\ge \Lambda _0\ge 0$.

The following theorem generalizes 
the Shestakov-Umirbaev inequality 
\cite[Theorem~3]{SUpoisson} 
and provides the foundation for 
Theorem~\ref{thm:ptame3} 
(see also \cite{EMW,MY,Vene}).

\begin{thm}[{\cite[Theorem 2.1]{SUineq}}]\label{thm:0SUineq}
Assume that $p=0$. 
Let $(\ff ,g)\in \kx ^r\times \kx $ 
be such that $d\ff \wedge dg\ne 0$. 
Then, for every $P\in k[\ff ][T]\sm \zs $, we have 
$$
\degww  P(g)\ge \deg _{\w _g}P-m_{\w }^g(P)\Lambda _0. 
$$
\end{thm}

If $p=0$, 
then 
$m_{\w }^g(P)
=\min \{ l\in \N \mid \deg _{\w _g}P^{(l)}=\degww  P^{(l)}(g)\} $ 
holds (see \cite[Lemma 3.1~(ii)]{SUineq}). 
We note that \cite{SUineq} takes this as a definition of $m_{\w }^g(P)$.

The following theorem is the key result that forms the basis of this paper.

\begin{thm}\label{thm:main}
Assume that $p>0$. 
Let $(\ff ,g)\in \kx ^r\times \kx $ be such that $d\ff \wedge dg\ne 0$. 
Then, for every $P\in k[\ff ][T]\sm \zs $, we have 
\begin{equation}\label{eq:pSUineq general}
\degww  P(g)\ge 
\deg _{\w _g}P-m_{\w }^g(P)\Lambda _0-v_p(m_{\w }^g(P)!)\Lambda _1. 
\end{equation}
\end{thm}

\begin{rem}\label{rem:Legendre}\rm 
Let $p$ be a prime, 
and $m=\sum _{i\ge 0}m_ip^i$ the $p$-adic expansion of $m\in \N $. 
Then, 
Legendre's formula (see, e.g., \cite[Theorem 2.6.4]{NF}) asserts 
\begin{equation}\label{eq:Legendre}
v_p(m!)=\dfrac{m-\ol{m}}{p-1}, 
\text{ \ where \ }
\ol{m}:=\sum _{i\ge 0}m_i. 
\end{equation}
\end{rem}

This paper has two parts. 
Part 1 proves Theorem~\ref{thm:main} and explores its consequences. 
Part 2 builds on these results to prove Theorem~\ref{thm:mainSU} 
by extending the proof of Theorem~\ref{thm:ptame3}. 
While the similarity between Theorems~\ref{thm:0SUineq} and \ref{thm:main} 
is a significant advantage, 
the proof of Theorem~\ref{thm:mainSU} requires various 
new ideas, concepts, and methods 
to overcome fundamental obstacles.

\part{Shestakov-Umirbaev type inequality in positive characteristic}

The first goal of Part 1 is to prove Theorem~\ref{thm:main}, 
with the basic idea outlined in \S~\ref{sect:principle}. 
After preparatory work in \S ~\ref{sect:preliminary}, 
we prove Theorem~\ref{thm:main} in \S ~\ref{sect:Proof of pSU}. 
The second goal is to derive consequences of Theorem~\ref{thm:main}, 
which are presented in \S ~\ref{sect:consequence}.

We remark that the definitions given in 
\S~\ref{subsetc:w-degree} 
still make sense when $k$ 
is replaced by any commutative ring. 
We also use the following notation and definitions. 
\begin{center}\it
In what follows, we write ``$\degww  $" as ``$\deg $" for simplicity. 
\end{center}

\begin{notation}\label{notation:E_l}\rm 
Let $R$ and $R'$ be commutative rings.

\nd (i) 
If $\zeta $ is a system of variables, 
then $R[[\zeta ]]$ denotes 
the formal power series ring in $\zeta $ over~$R$.

\nd (ii) 
Let $E:R\to A$ be a map, 
where $A\in \{ R'[T],R'[[T]]\} $. 
Then, for each $a\in R$, 
we express $E(a)$ as $\sum _{l\ge 0}E_l(a)T^l$ with $E_l(a)\in R'$.

\nd (iii) 
For $\omega ',\omega =\sum _{\bi \in \Lambda (n,s)}f_{\bi }d\x _{\bi }
\in \bigwedge ^s\Omega _{\Rx /R}$ 
with $f_{\bi }\in \Rx $, 
we write $\omega '\prec \omega $ if 
$\omega '=f_{\bi }d\x _{\bi }$ for some $\bi \in \Lambda (n,s)$. 
Clearly, 
$\omega '\prec \omega $ implies  
$\degw  \omega '\le \degw \omega $.

\nd (iv) 
Let $\kxr :=Q(\kx )$ be the rational function field. 
For $f=f_1/f_2\in \kxr $ 
with $f_1\in \kx $ and $f_2\in \kx \sm \zs $, 
we define $\degw  f:=\degw  f_1-\degw  f_2$.

\nd (v) 
$f\in \kx $ evaluated at $\bu \in k^n$ is denoted by $f(\bu )$. 
\end{notation}

\section{Principle}\label{sect:principle}
\setcounter{equation}{0}

Let $k$ be any field, 
$A$ a $k$-subalgebra of $\kx $, 
and $g\in \kx $ transcendental over $A$. 
We define $E:A[g]\ni P(g)\mapsto P(g+T)\in \kx [T]$ 	
and use Notation~\ref{notation:E_l}~(ii).

\begin{definition}\label{def:(A,g)-function}\rm 
We call a map $\nu :\N \to \Gamma $ an 
$(A,g)$-{\it function} if the following condition holds: 
\nd (F) $\degw E_l(h)\le \degw h+\nu (l)$ 
for all $l\in \N $ and $h\in A[g]\sm \zs $. 
\end{definition}

The basic idea of our inequality 
is formulated as the following theorem.

\begin{thm}\label{thm:fundamental}
Let $k$, $A$, and $g$ be as above, 
and $\nu :\N \to \Gamma $ an $(A,g)$-function. 
Then, for every $P\in A[T]\sm \zs $, 
we have 
\begin{equation}\label{eq:principle}
\degw P(g)\ge \deg _{\w _g}P-m_{\w }^g(P)\degw g-\nu (m_{\w }^g(P)).
\end{equation}
\end{thm}
\begin{proof}
By (F) with $l=m_{\w }^g(P)$ and $h=P(g)$, 
we have 
\begin{equation}\label{eq:F2}
\degw P(g)\ge \degw E_{m_{\w }^g(P)}(P(g))-\nu (m_{\w }^g(P)). 
\end{equation}
We can write $E(P(g))=\sum _{j\ge 0}E_j(P(g))T^j$. 
Hence, (\ref{eq:P(T)^w_g1}) gives 
\begin{equation}\label{gather:principle}
E(P(g))^{\w _g}=\sum _{j\in J}E_j(P(g))^{\w }T^j, 
\end{equation}
where $J:=\{ j\in \N \mid \degw E_j(P(g))g^j=\deg _{\w _g}E(P(g))\} $. 
Writing $P$ as in (\ref{eq:Pick any P}) also yields 
$E(P(g))=P(g+T)=\sum _{i\ge 0}p_i\cdot (g+T)^i$. 
To obtain the two required conditions, 
we apply Remark~\ref{rem:principle} to $(p_i\cdot (g+T)^i)_{i\ge 0}$. 
Let $\delta $ be the maximum among 
$\deg _{\w _g}p_i\cdot (g+T)^i=\degw p_ig^i$ for $i\ge 0$. 
Then, we have $\delta =\deg _{\w _g}P$ by (\ref{eq:degw_gP}). 
Hence, 
$K:=\{ i\in \N \mid \deg _{\w _g}p_i\cdot (g+T)^i=\delta \} $ 
is equal to $I$ in (\ref{eq:P(T)^w_g1}). 
This, 
together with (\ref{eq:P(T)^w_g1}) and (\ref{eq:P(T)^w_g2}), 
yields 
$$
h:=\sum _{i\in K}(p_i\cdot (g+T)^i)^{\w _g}
=\sum _{i\in K}p_i^{\w }\cdot (g^{\w }+T)^i
=P^{\w _g}(g^{\w }+T)=T^{m_{\w }^g(P)}P_0(g^{\w }+T). 
$$
Since $P_0\ne 0$, we have $h\ne 0$. 
Thus, 
since $E(P(g))=\sum _{i\ge 0}p_i\cdot (g+T)^i$, 
it follows from 
Remark~\ref{rem:principle} that 
\begin{gather}
\deg _{\w _g}E(P(g))=\delta =\deg _{\w _g}P, \label{eq:pf:P(T+g)0} \\
E(P(g))^{\w _g}=h=T^{m_{\w }^g(P)}P_0(g^{\w }+T). \label{eq:pf:P(T+g)1}
\end{gather}
The coefficient of $T^{m_{\w }^g(P)}$ in 
$T^{m_{\w }^g(P)}P_0(g^{\w }+T)$ is $P_0(g^{\w })\ne 0$. 
Hence, 
comparing (\ref{eq:pf:P(T+g)1}) with (\ref{gather:principle}) gives 
$m_{\w }^g(P)\in J$, 
i.e., 
$\degw E_{m_{\w }^g(P)}(P(g))g^{m_{\w }^g(P)}=\deg _{\w _g}E(P(g))$. 
This equality combined with 
(\ref{eq:F2}) and (\ref{eq:pf:P(T+g)0}) 
completes the proof of (\ref{eq:principle}). 
\end{proof}

Now, 
assume that $p>0$, 
and let $\ff =(f_1,\ldots ,f_r)\in \kx ^r$ and $g\in \kx $ 
be such that $d\ff \wedge dg\ne 0$. 
To prove Theorem~\ref{thm:main}, 
we may replace $k$ with 
any extension field of $k$, or 
any subfield $k'$ of $k$ with 
$f_1,\ldots ,f_r,g\in k'[\x ]$ and $P\in k'[\ff ][T]$. 
Hence, 
we may assume that $k$ is an infinite field 
finitely generated over its prime field. 
In this setting, 
by Theorem~\ref{thm:fundamental}, 
it suffices to verify that 
\begin{equation}\label{eq:nu_{f,g}}
\nu _{\ff ,g}:\N \ni l\mapsto 
l(\degw d\ff -\degw d\ff \wedge dg)+v_p(l!)\Lambda _1
\in \Gamma 
\end{equation}
is a $(k[\ff],g)$-function. 
To achieve this, 
we introduce the following concept.

\begin{definition}\label{def:(f,g)-map}\rm 
We call a homomorphism $E:\kx \to \kxr [[T]]$ of $k$-algebras 
an $(\ff ,g)$-{\it map} if the following conditions hold, 
where (M2) employs Notation~\ref{notation:E_l} (ii):

\nd (M1) $E(f_i)=f_i$ for $i=1,\ldots ,r$, 
and $E(g)=g+T$;

\nd (M2) 
$\degw E_l(x_i)\le w_i+\nu _{\ff ,g}(l)$ 
for all $l\in \N $ and $i=1,\ldots ,n$. 
\end{definition}

The existence of an $(\ff ,g)$-map $E$ 
implies that 
$\nu _{\ff ,g}$ 
is a $(k[\ff],g)$-function. 
Indeed, (M1) implies that $E$ restricts to the map 
$k[\ff ][g]\ni P(g)\mapsto P(g+T)\in \kx [T]$. 
Moreover, 
(M2) implies that 
$\nu _{\ff ,g}$ satisfies (F) by the following lemma.

\begin{lem}\label{lem:E}
\nd{\rm (i)} 
Let $\nu :\N \to \Gamma $ be a map, 
and $E:\kx \to \kxr [[T]]$ a homomorphism of $k$-algebras. 
Assume that the following conditions hold$:$

{\rm ($1$)} 
$\degw E_l(x_i)\le w_i+\nu (l)$ for all 
$l\in \N $ and $i=1,\ldots ,n;$

{\rm ($2$)} 
$\nu (l)+\nu (l')\le \nu (l+l')$ for all $l,l'\in \N $.

\nd 
Then, 
$\degw E_l(h)\le \degw h+\nu (l)$ 
holds for all $l\in \N $ and $h\in \kx \sm \zs $.

\nd{\rm (ii)} 
Let $\lambda ,\lambda '\in \Gamma $. 
If $\lambda '\ge 0$, 
then 
$\nu :\N \ni l\mapsto l\lambda +v_p(l!)\lambda '\in \Gamma $ 
satisfies {\rm (i) (2)}. 
\end{lem}
\begin{proof}
(i) 
Writing 
$h=\sum _{{\boldsymbol i}\in \N ^n}c_{\boldsymbol i}\x ^{\boldsymbol i}$ 
with $c_{\boldsymbol i}\in k$, 
we have $E_l(h)
=\sum _{{\boldsymbol i}\in \N ^n}c_{\boldsymbol i}E_l(\x ^{\boldsymbol i})$. 
This implies that 
$\deg E_l(h)\le \max \{ \deg E_l(\x ^{\bi })\mid \bi \in \supp h\} $. 
The definition of $\deg h$ also gives 
$\degw h+\nu (l)=\max \{ \bi \cdot \w +\nu (l)\mid \bi \in \supp h\} $. 
Thus, 
it suffices to verify that 
$\deg E_l(\x ^{\boldsymbol i})\le {\boldsymbol i}\cdot \w +\nu (l)$
for all ${\boldsymbol i}\in \N ^n$. 
Recall that $E_l(\x ^{\boldsymbol i})$ 
is the coefficient of $T^l$ in $E(\x ^{\boldsymbol i})$. 
Write $\x ^{\boldsymbol i}=x_{j_1}\cdots x_{j_e}$, where 
$j_1,\ldots ,j_e\in \{ 1,\ldots ,n\} $. 
Then, we have 
$$
E(\x ^{\boldsymbol i})=E(x_{j_1})\cdots E(x_{j_e})
=\left(\sum _{l_1=0}^{\infty }E_{l_1}(x_{j_1})T^{l_1}\right)
\cdots 
\left(\sum _{l_e=0}^{\infty }E_{l_e}(x_{j_e})T^{l_e}\right), 
$$
yielding $E_l(\x ^{\boldsymbol i})=\sum _{l_1+\cdots +l_e=l}
E_{l_1}(x_{j_1})\cdots E_{l_e}(x_{j_e})$. 
For all $l_1,\ldots ,l_e\in \N $ with $l_1+\cdots +l_e=l$, 
it follows from ($1$) and ($2$) that 
\begin{align*}
\degw E_{l_1}(x_{j_1})\cdots E_{l_e}(x_{j_e})
\le \sum _{t=1}^e(w_{j_t}+\nu (l_t)) 
\le \sum _{t=1}^ew_{j_t}+\nu (l)
={\boldsymbol i}\cdot \w +\nu (l). 
\end{align*}
This verifies that 
$\deg E_l(\x ^{\boldsymbol i})\le {\boldsymbol i}\cdot \w +\nu (l)$.

(ii) 
This is clear, since 
$v_p(l!)+v_p(l'!)\le v_p(l!\cdot l'!\cdot \binom{l+l'}{l})=v_p((l+l')!)$. 
\end{proof}

Therefore, 
the proof of Theorem~\ref{thm:main} reduces to the following key theorem.

\begin{thm}\label{thm:key}
Assume that $p>0$, 
and $k$ is an infinite field 
finitely generated over its prime field. 
Then, 
for every 
$(\ff ,g)\in \kx ^r\times \kx $ 
with $d\ff \wedge dg\ne 0$, 
there exists an $(\ff ,g)$-map. 
\end{thm}

Finally, 
we discuss 
the {\it translation} 
$\tau _{\bu }:=(x_1+u_1,\ldots ,x_n+u_n)\in \Aut _k\kx $, 
where $\bu =(u_1,\ldots ,u_n)\in k^n$. 
Let $h\in \kx $ and $\bh =(h_1,\ldots ,h_s)\in \kx ^s$, 
and define 
$\tau _{\bu }\bh :=(\tau _{\bu }(h_1),\ldots ,\tau _{\bu }(h_s))$. 
Then, the following assertions (T1)--(T4) hold:

\nd (T1) 
$\tau _{\bu }(h)(0)=h(0+u_1,\ldots ,0+u_n)=h(\bu )$;

\nd (T2) 
$\partial \tau _{\bu }(h)/\partial x_i
=\tau _{\bu }(\partial h/\partial x_i)$ 
for $i=1,\ldots ,n$;

\nd (T3) 
$\tau _{\bu }(\x ^{\bi })^{\w }=\x ^{\bi }$ for all $\bi \in \N ^n$, 
implying 
$\tau _{\bu }(h)^{\w }=h^{\w }$ and $\degw \tau _{\bu }(h)=\degw h$;

\nd (T4) 
For all $\bi \in \Lambda (n,s)$, 
we have $\left| 
\partial \tau _{\bu }\bh /
\partial \x _{\bi }\right|
=\tau _{\bu }\left( 
\left| \partial \bh /\partial \x _{\bi }\right| \right) $ 
by (T2), 
and hence 
$\degw \left| \partial \tau _{\bu }\bh /\partial \x _{\bi }\right|
=\degw 
\left| \partial \bh /\partial \x _{\bi }\right| $ 
by (T3). 
Thus, 
$\degw d\tau _{\bu }\bh =\degw d\bh$ holds by (\ref{eq:dff}).

We still denote by $\tau _{\bu }$ 
the extension of $\tau _{\bu }$ to $\kxr $. 
Then, (T3) implies

\nd (T3$'$) $\deg \tau _{\bu }(h)=\deg h$ for all $h\in \kxr $.

\begin{lem}\label{lem:(f,g)-map reduction}
Let $(\ff ,g)\in \kx ^r\times \kx $ 
be such that $d\ff \wedge dg\ne 0$.

\nd{\rm (i)} 
Take any $\bu \in k^n$. 
Then, 
we have $\tau _{\bu }\ff \wedge \tau _{\bu }(g)\ne 0$ 
by {\rm (T4)}. 
Moreover, 
for every $\tau _{\bu }(\ff ,g):=(\tau _{\bu }\ff ,\tau _{\bu }(g))$-map 
$E':\kx \to \kxr [[T]]$, 
the composite map 
$$
E:\kx \stackrel{\tau _{\bu }}{\longrightarrow }
\kx \stackrel{E'}{\longrightarrow }
\kxr [[T]]
\stackrel{\tilde{\tau }_{-\bu }}{\longrightarrow }
\kxr [[T]]
$$
is an $(\ff ,g)$-map, 
where we define 
$\tilde{\tau }_{-\bu }(\sum _{l=0}^\infty h_lT^l):=
\sum _{l=0}^\infty \tau _{-\bu }(h_l)T^l$.

\nd{\rm (ii)} 
An $(\ff ,g)$-map is an $(\ff +\ba ,g+b)$-map 
for all $(\ba ,b)\in k^r\times k$. 

\end{lem}
\begin{proof}
(i) By definition, 
$E'$ satisfies the following conditions:

\nd (M$1'$) $E'(\tau _{\bu }(f_i))=\tau _{\bu }(f_i)$ for $i=1,\ldots ,r$, 
and $E'(\tau _{\bu }(g))=\tau _{\bu }(g)+T$;

\nd (M$2'$) 
$\degw E_l'(x_i)\le w_i+\nu _{\tau _{\bu }\ff ,\tau _{\bu }(g)}(l)$ 
for all $l\in \N $ and $i=1,\ldots ,n$. 

\nd 
Now, since $\tau _{-\bu }=\tau _{\bu }^{-1}$, 
it follows from (M$1'$) that $E$ satisfies (M1). 
For $i=1,\ldots ,n$, 
we have $E(x_i)=\sum _{l=0}^\infty 
\tau _{-\bu }(E_l'(x_i))T^l+u_i$, 
i.e., 
$E_l(x_i)=\tau _{-\bu }(E_l'(x_i))$ for $l\ge 1$ 
and $E_0(x_i)=\tau _{-\bu }(E_0'(x_i))+u_i$. 
By (T3$'$) and (T4), we also have 
$\deg \tau _{-\bu }(E_l'(x_i))=\deg E_l'(x_i)$ and 
$\nu _{\tau _{\bu }\ff ,\tau _{\bu }(g)}(l)=\nu _{\ff ,g}(l)$ 
for all $l\in \N $. 
Thus, 
(M$2'$) implies that 
$E$ satisfies the inequality in (M2) for $l\ge 1$. 
Since $w_i+\nu _{\ff ,g}(0)=w_i>\deg u_i$, 
the same holds for $l=0$.

(ii) Straightforward. 
\end{proof}

\section{Preliminary}\label{sect:preliminary}
\setcounter{equation}{0}

\subsection{Formal Inverse Function Theorem}\label{sect:FIFT}

Let $R[[\y ]]$ be the formal power series ring in $\y :=(y_1,\ldots ,y_s)$ 
over a commutative ring $R$. 
Then, 
each $s$-tuple 
$\ff \in \mathcal{M}:=(\sum _{i=1}^sy_i\Ryy )^s$ 
defines the {\it substitution map} 
$\Ryy \ni h\mapsto h(\ff )\in \Ryy $, 
which is a homomorphism of $R$-algebras and denoted by $\ff $. 
Every homomorphism $\phi :R[[\y ]]\to R[[\y ]]$ 
of $R$-algebras satisfying 
$(\phi(y_1),\ldots ,\phi (y_s))\in \mathcal{M}$ 
is the substitution map defined by 
$(\phi(y_1),\ldots ,\phi (y_s))$ 
(see, e.g., \cite[Chapter VII, \S 1, pp.~135--136]{ZS}).

Now, for $\ff ,\bg =(g_1,\ldots ,g_s)\in \mathcal{M}$, 
we define $\ff \bg :=(g_1(\ff ),\ldots ,g_s(\ff ))\in \mathcal{M}$. 
Note that the composite 
$R[[\y ]]\stackrel{\bg }{\to }R[[\y ]]\stackrel{\ff }{\to }R[[\y ]]$ 
is a homomorphism of $R$-algebras with 
$y_i\mapsto g_i\mapsto g_i(\ff )$, 
and hence is equal to 
the substitution map defined by~$\ff \bg $. 
Thus, by the associativity of composition, 
$(\ff \bg )\bh =\ff (\bg \bh )$ 
holds for all $\ff ,\bg ,\bh\in \mathcal{M}$.

The following theorem is well-known as the 
{\it Formal Inverse Function Theorem}.

\begin{thm}[see, e.g., {\cite[Theorem 1.1.2]{Essen}}]
\label{thm:FIFT}
Let $\ff \in \mathcal{M}$ be such that 
$|\partial \ff /\partial \y |(0)\in R^*$, 
where $|\partial \ff /\partial \y |(0)$ denotes the 
constant term of the Jacobian 
$|\partial \ff /\partial \y |\in R[[\y ]]$. 
Then, 
there exists 
$\bg \in \mathcal{M}$ 
such that $\ff \bg =\y $. 
\end{thm}

\begin{cor}\label{lem:2023Feb}
Let $R'$ be an extension ring of $R$, 
and $\bh \in (\sum _{i=1}^sy_iR'[[\y ]])^s$. 
If there exists $\ff \in \mathcal{M}$ 
such that $|\partial \ff /\partial \y |(0)\in R^*$ 
and $\bh \ff \in \mathcal{M}$, 
then we have $\bh \in \mathcal{M}$. 
\end{cor}
\begin{proof}
By Theorem~\ref{thm:FIFT}, 
there exists $\bg \in \mathcal{M}$ such that $\ff \bg =\y $. 
Then, since $\bh \ff ,\bg \in \mathcal{M}$, 
it follows that 
$\bh =\bh \y =\bh (\ff \bg )=(\bh \ff )\bg \in \mathcal{M}$. 
\end{proof}

\subsection{Exponential map}\label{sect:EXP}

Let $R\subset A$ be an extension of integral domains, 
and $\Der _R(A)$ the set of all $R$-{\it derivations} of $A$, 
defined as 
$R$-linear maps $D:A\to A$ satisfying 
$D(ab)=D(a)b+aD(b)$ for all $a,b\in A$. 
Recall that the following hold for all $D\in \Der _R(A)$: 
(1) 
If $S$ is a multiplicative set of $A$, 
then $D$ uniquely extends to an $R$-derivation $D'$ 
of the localization $A_S$ by 
$D'(a/s):=(D(a)s - aD(s))/s^2$. 
(2) For each $a\in A$, 
$aD\in \Der _R(A)$ 
is defined in an obvious way.

Now, if $\Q \subset R$, 
then each $D\in \Der _R(A)$ defines the {\it exponential map} 
$$
\exp TD:A[[T]]\ni f\mapsto 
\sum _{l=0}^{\infty }\dfrac{\tilde{D}^l(f)}{l!}T^l
\in A[[T]], 
$$
where 
$\tilde{D}\in \Der _R(A[[T]])$ is defined by 
$\tilde{D}(\sum _{i=0}^\infty a_iT^i)
:=\sum _{i=0}^\infty D(a_i)T^i$. 
This is an automorphism of the $R$-algebra $A[[T]]$ 
(see, e.g., \cite[Proposition 1.2.14]{Essen}). 
By definition, 
we have 
$(\exp TD)(a)=a$ and $(\exp TD)(b)=b+T$ 
for all $a\in \ker D$ and $b\in A$ with $D(b)=1$.

\begin{example}\label{ex:JD}
Let $1\le r<n$, 
$\tilde{\ff }=(\tilde{f}_1,\ldots,\tilde{f}_r)\in \Rx ^r$, 
and $\bk :=(1,2,\ldots ,r+1)$.

\nd (i) 
We can 
define $\delta _{\tilde{\ff }}\in \Der _R(\Rx )$ by 
$\delta _{\tilde{\ff }}(h):=|\partial (\tilde{\ff },h)/\partial \x _{\bk }|$ 
for all $h\in \Rx $. Then, 
\begin{equation}\label{eq:def:D_ff}
\delta _{\tilde{\ff }}(h)d\x _{\bk }\prec d\tilde{\ff }\wedge dh
\qquad \text{ holds for all }h\in \Rx 
\text{ by (\ref{eq:dff})}. 
\end{equation}

\nd (ii) 
Let $\tilde{g}\in \Rx $ be such that 
$\delta _{\tilde{\ff }}(\tilde{g})\ne 0$. 
Then, 
$\delta _{\tilde{\ff }}$ uniquely extends to an $R$-derivation $D'$ of 
$\Rx [\delta _{\tilde{\ff }}(\tilde{g})^{-1}]$. 
Set $\Delta :=\delta _{\tilde{\ff }}(\tilde{g})^{-1}D'\in 
\Der _R(\Rx [\delta _{\tilde{\ff }}(\tilde{g})^{-1}])$. 
Then, since 
$C:=\{ \tilde{f}_1,\ldots ,\tilde{f}_r,x_{r+2},\ldots ,x_n\} 
\subset 
\ker \delta _{\tilde{\ff }}\subset \ker \Delta $ 
and $\Delta (\tilde{g})=1$, 
we have $(\exp T\Delta )(f)=f$ 
for all $f\in C$, 
and $(\exp T\Delta )(\tilde{g})=\tilde{g}+T$. 
\end{example}

Assume further that $A$ is a subring of $R[[\x ]]$ with $\Rx \subset A$. 
Then, we have 
\begin{equation}\label{eq:exp continuous}
(\exp TD)(x_i)=x_i+T\sum _{l=1}^{\infty }\dfrac{D^l(x_i)}{l!}T^{l-1}
\in \sum _{i=1}^nx_iR[[\x ,T]]+TR[[\x ,T]] 
\end{equation}
for $i=1,\ldots ,n$. 
Hence, 
the substitution map $\theta _D:R[[\x ,T]]\to R[[\x ,T]]$ is defined~by 
\begin{equation}\label{eq:exp TD map}
\theta _D:=((\exp TD)(x_1),\ldots ,(\exp TD)(x_n),T). 
\end{equation}
Since $\exp TD$ is an $R$-algebra automorphism, 
the following holds for all $f\in \Rx $: 
\begin{equation}\label{eq:rem:theta_D}
(\exp TD)(f)=f((\exp TD)(x_1),\ldots ,(\exp TD)(x_n))
=\theta _D(f).
\end{equation}

\subsection{Lifting of coefficients}\label{sect:Lift}

Let $p$ be a prime, 
$k=\F _p(\xi _1,\ldots ,\xi _s)$ 
a finitely generated extension field of $\F _p$, 
and $\pi ':\Z [\y ]=\Z [y_1,\ldots ,y_s]
\to k$ the ring homomorphism 
defined by $y_i\mapsto \xi _i$. 
Then, 
$R:=\Z [\y ]_{\ker \pi '}$ 
is a Noetherian local domain with maximal ideal 
$\mathfrak{m}:=(\ker \pi ')R$. 
Since $\pi '(\Z [\y ]\sm \ker \pi ')=\F _p[\xi _1,\ldots ,\xi _s]\sm \zs $, 
we see that 
$\pi '$ uniquely extends to a surjective ring homomorphism $\pi :R\to k$.

\begin{rem}\label{rem:RLR}\rm 
(i) We have $\ker \pi =\mathfrak{m}$ 
and $R/\mathfrak{m}\simeq k$. 

\nd (ii) 
Since $(\ker \pi ')\cap \Z $ 
is a prime ideal of $\Z $ containing $p$, 
we have $(\ker \pi ')\cap \Z =p\Z $. 
Hence, $R_p:=R[1/p]$ contains $\Q $. 
\end{rem}

\begin{lem}\label{lem:RLR}
In the notation above, 
$p^e\not\in \fm ^{e+1}$ holds for all $e\in \N $. 
\end{lem}

For example, 
if $\xi _1,\ldots ,\xi _s$ are algebraically independent over $\F _p$, 
then we have 
$\ker \pi '=p\Z [\y ]$ and $\mathfrak{m}=pR$. 
In this case, Lemma~\ref{lem:RLR} is obvious.

We derive Lemma~\ref{lem:RLR} from 
the following facts in commutative ring theory: 
Let $(A,M)$ be a $d$-dimensional Noetherian local ring, 
where $d\ge 0$.

\nd (R1) 
$A$ is called a {\it regular local ring} 
if $M$ can be generated by $d$ elements. 
A Noetherian ring $B$ is called a {\it regular ring} 
if the localization of $B$ at every prime ideal is a regular local ring. 
If $B$ is a regular ring, 
then so is the polynomial ring $B[T]$ 
by \cite[Theorem 19.5]{Matsumura}. 
Since 
a field is a regular ring, 
it follows that $\F _p[\y ]$ is a regular ring by induction on $s$.

\nd (R2) 
Let $\gr _M(A)=\bigoplus _{l\in \N }M^l/M^{l+1}$ be 
the associated graded ring, 
$a_1,\ldots ,a_t$ generators of $M$, 
and $\ol{a}_1,\ldots ,\ol{a}_t$ their images in $M/M^2$. 
Then, we have $\gr _M(A)=(A/M)[\ol{a}_1,\ldots ,\ol{a}_t]$, 
and $\ol{a}_1^e\in \gr _M(A)$ is equal to 
the image of $a_1^e$ in $M^e/M^{e+1}$ for all $e\in \N $ 
(see \cite[\S~13]{Matsumura}). 
If $t=d$, 
then $A$ is a regular local ring by definition, 
implying that $\gr _M(A)$ is a polynomial ring in $d$ variables over $A/M$ 
by \cite[Theorem~14.4]{Matsumura}. 
This implies that 
$a_1^e\not\in M^{e+1}$ for all $e\in \N $, 
since $\ol{a}_1^e\ne 0$.

\begin{proof}[Proof of Lemma~{\rm \ref{lem:RLR}}]
Factor $\pi '$ as 
$\Z [\y ]\stackrel{\sigma }{\to }
\F _p[\y ]\stackrel{\psi }{\to }k$, 
where $\sigma $ is the natural surjection 
and $\psi $ is the ring homomorphism defined by $y_i\mapsto \xi _i$. 
Then, 
we have $\sigma (\Z [\y ]\sm \ker \pi ')=\F _p[\y ]\sm \ker \psi $, 
so $\sigma $ uniquely extends to a surjective ring homomorphism 
$\tilde{\sigma }:R\to \F _p[\y ]_{\ker \psi }$. 
It induces a ring isomorphism 
$R/pR=R/\ker \tilde{\sigma }\to \F _p[\y ]_{\ker \psi }$. 
Since $\F _p[\y ]_{\ker \psi }$ 
is a regular local ring by (R1), 
it follows that so is $R/pR$. 
Now, set $d:=\dim R$. 
Since $\dim R/pR\le d-1$, 
the maximal ideal $\fm /pR$ of $R/pR$ 
is generated by the images of some 
$a_2,\ldots ,a_d\in \fm $. 
Then, 
$\fm $ is generated by $p$, $a_2,\ldots ,a_d$. 
Therefore, the lemma follows from (R2). 
\end{proof}

\subsection{Reduction modulo $\fm ^{e+1}$}\label{sect:Reduce}

Let $(R,\fm )$ be the local ring defined in \S~\ref{sect:Lift}. 
For each $e\in \N $, 
we define the natural surjection 
\begin{equation}\label{eq:pi_e}
\pi _e:\Rx \to R_e[\x ], 
\text{ \ where \ }R_e:=R/\fm ^{e+1}. 
\end{equation}

\begin{lem}\label{lem:reduce coeff}
Let $e\in \N $, 
and 
$h=\sum _{\bi }c_{\bi }\x ^{\bi },
h'=\sum _{\bj }c_{\bj }'\x ^{\bj }\in \Rx $, 
where $c_{\bi },c_{\bj }'\in R$.

\nd {\rm (i)} 
$\deg \pi _e(ah)=\deg \pi _e(h)$ for all $a\in R^*$.

\nd {\rm (ii)} 
$\degw \pi _d(h)\le \degw \pi _e(h)\le \degw h$ 
for all $0\le d\le e$.

\nd {\rm (iii)} 
$\degw \pi _0(h)=\degw \pi _e(p^eh)$.

\nd {\rm (iv)} 
$\degw \pi _e(hh')
\le \max \{ \degw \pi _{d}(h)+\degw \pi _{e-d}(h')\mid 
0\le d\le e\} $. 
\end{lem}
\begin{proof}
Since 
$\supp \pi _e(h)=\{ \bi \in \N ^n \mid c_{\bi }\not\in \fm ^{e+1}\} $, 
(i) and (ii) are obvious.

(iii) 
$c\in R$ belongs to $R\sm \fm =R^*$ 
if and only if $p^ec\not\in \fm ^{e+1}$, 
where the ``if" part is clear since $p^e\in \fm ^e$, 
and the ``only if" part is due to Lemma~\ref{lem:RLR}. 
This implies that $\supp \pi _0(h)=\supp \pi _e(p^eh)$, 
and hence $\degw \pi _0(h)=\degw \pi _e(p^eh)$.

(iv) 
Pick $\bk \in \supp \pi _e(hh')$ 
with $\degw \pi _e(hh')=\bk \cdot \w $. 
Then, 
the coefficient $c$ of $\x ^{\bk }$ in $hh'$ 
is not in $\fm ^{e+1}$. 
Since $c=\sum _{\bi +\bj =\bk }c_{\bi }c_{\bj }'$, 
it follows that 
$c_{\bi }c_{\bj }'\not\in \fm ^{e+1}$ for some 
$\bi $, $\bj $ 
with $\bi +\bj =\bk $. 
Clearly, 
$c_{\bi }\not\in \fm ^{e+1}$. 
Choose $0\le d\le e$ so that 
$c_{\bi }\not\in \fm ^{d+1}$ and $c_{\bi }\in \fm ^d$. 
Then, 
$c_{\bj }'\not\in \fm ^{e-d+1}$. 
Thus, 
we have 
$\bi \in \supp \pi _d(h)$ and $\bj \in \supp \pi _{e-d}(h')$, 
implying that $\bi \cdot \w \le \degw \pi _d(h)$ and 
$\bj \cdot \w \le \degw \pi _{e-d}(h')$. 
This yields 
$\degw \pi _e(hh')=\bk \cdot \w =\bi \cdot \w +\bj \cdot \w 
\le \degw \pi _d(h)+\degw \pi _{e-d}(h')$. 
\end{proof}

For each $e\in \N $, 
$R$-linear maps 
$\Rx ^s\to R_e[\x ]^s$ and 
$\bigwedge ^s\Omega _{\Rx /R}\to \bigwedge ^s\Omega _{R_e[\x ]/R_e}$ 
are defined by sending 
$\bh =(h_1,\ldots ,h_s)\in \Rx ^s$ and 
$\omega =\sum _{\bi }f_{\bi }d\x _{\bi }
\in \bigwedge ^s\Omega _{\Rx /R}$ 
to 
$\pi _e\bh :=(\pi _e(h_1),\ldots ,\pi _e(h_s))$ 
and 
$\pi _e(\omega ):=\sum _{\bi }\pi _e(f_{\bi })d\x _{\bi }$, 
respectively.

\begin{rem}\label{rem:reduce jacobian}\rm 
Let $e\in \N $, $h\in \Rx $, 
and $\bh \in \Rx ^s$. 
Then, 
$\pi _e(\partial h/\partial x_l)=\partial \pi _e(h)/\partial x_l$ 
holds for $l=1,\ldots ,n$. 
It follows that 
$\pi _e(|\partial \bh/\partial \x _{\bi }|)
=|\partial \pi _e\bh /\partial \x _{\bi }|$ 
for all $\bi \in \Lambda (n,s)$. 
Hence, 
we have $\pi _e(d\bh)=d\pi _e\bh $ by (\ref{eq:dff}). 
\end{rem}

\begin{lem}\label{lem:reduce coeff2}
Let $e\in \N $, 
$\omega ,\omega '\in \bigwedge ^s\Omega _{\Rx /R}$, 
and $h,h'\in \Rx $. 

\nd{\rm (i)} 
If $\omega '\prec \omega $, 
then $\pi _e(\omega ')\prec \pi _e(\omega )$ and 
$\deg \pi _e(\omega ')\le \deg \pi _e(\omega )$.

\nd{\rm (ii)} 
$\deg \pi _e(\omega +\omega ')
=\deg (\pi _e(\omega )+\pi _e(\omega '))
\le \max \{ \degw \pi _e(\omega ),\degw \pi _e(\omega ')\} $.

\nd{\rm (iii)} 
$\deg \pi _e(b\omega )\le \deg \pi _e(\omega )$ 
for all $b\in R$.

\nd{\rm (iv)} 
$\deg \pi _d(\omega )\le \deg \pi _e(\omega )\le \deg \omega $ 
for all $0\le d\le e$.

\nd{\rm (v)} 
$\degw \pi _e(\omega \wedge dh)\le 
\max \{ \degw \pi _{d}(\omega )+\degw \pi _{e-d}(h)\mid 
0\le d\le e\} $.

\nd{\rm (vi)} 
$\degw \pi _e(h\omega )\le 
\max \{ \degw \pi _{d}(h)+\degw \pi _{e-d}(\omega )\mid 
0\le d\le e\} $.

\nd{\rm (vii)} 
$\degw \pi _e(h\omega \wedge dh')\le 
\max \{ \degw \pi _{i}(h)
+\degw \pi _{j}(\omega)
+\degw \pi _{l}(h')
\mid i+j+l=e\} $.

\nd{\rm (viii)} 
$\degw \pi _e(\omega \wedge dx_l)
\le \degw \pi _e(\omega )+w_l$ 
for $l=1,\ldots ,n$. 
\end{lem}
\begin{proof}
(i)--(iv) are obvious. 
For (v), 
write $\omega =\sum _{\bi }f_{\bi }d\x _{\bi }$ with $f_{\bi }\in \Rx $. 
Then, 
we have 
$\omega \wedge dh
=\sum _{\bi ,j}f_{\bi }h_jd\x _{\bi }\wedge dx_j$, 
where 
$h_j:=\partial h/\partial x_j$. 
By (ii), 
there exist $\bi $ and $j$ such that 
$\deg \pi _e(\omega \wedge dh)
\le \deg \pi _e(f_{\bi }h_jd\x _{\bi }\wedge dx_j)
\le \deg \pi _e(f_{\bi }h_j)+|\w _{\bi }|+w_j$. 
Thanks to Lemma~\ref{lem:reduce coeff} (iv), 
this is at most 
$\deg \pi _{d}(f_{\bi })d\x _{\bi }+\deg \pi _{e-d}(h_j)+w_j$ 
for some $0\le d\le e$. 
Since $\pi _{d}(f_{\bi })d\x _{\bi }\prec \pi _{d}(\omega )$ 
and 
$\deg \pi _{e-d}(h_j)
=\deg (\partial \pi _{e-d}(h)/\partial x_j)\le \deg \pi _{e-d}(h)-w_j$, 
it follows that $\degw \pi _e(\omega \wedge dh)
\le \degw \pi _{d}(\omega )+\degw \pi _{e-d}(h)$, 
proving (v). 
Since 
$\degw \pi _e(h\omega )=\degw \pi _e(hf_{\bi })d\x _{\bi }$ 
for some $\bi $ by definition, 
(vi) follows from Lemma~\ref{lem:reduce coeff}~(iv) similarly. 
Combining (v) and (vi) yields (vii). 
(viii) follows from (v) and (iv), 
since $\deg \pi _{e-d}(x_l)=w_l$ for all $0\le d\le e$. 
\end{proof}

\section{Proof of Theorem~\ref{thm:key}}\label{sect:Proof of pSU}
\setcounter{equation}{0}

We now prove Theorem~\ref{thm:key} 
and thereby Theorem~\ref{thm:main}.

\nd 
{\bf Step 1. Translation.} 
By (\ref{eq:dff}) and (\ref{eq:deg omega}), 
there exists $\bk \in \Lambda (n,r+1)$ such that 
\begin{equation}\label{eq:pf1}
\degw d\ff \wedge dg=
\degw 
\left|
\dfrac{\partial (\ff ,g)}{\partial \x _{\bk }}
\right|+|\w _{\bk }|. 
\end{equation}
By simultaneously changing the indices of 
$w_1,\ldots ,w_n$ and $x_1,\ldots ,x_n$, 
we may assume that 
$\bk =(1,2,\ldots ,r+1)$. 
Since $d\ff \wedge dg\ne 0$ by assumption, 
(\ref{eq:pf1}) implies that 
$|\partial (\ff ,g)/\partial \x _{\bk }|\ne 0$. 
Hence, by the assumption that $k$ is an infinite field, 
there exists $\bu \in k^n$ such that 
$|\partial (\ff ,g)/\partial \x _{\bk }|(\bu )\ne 0$. 
Then, 
(T4) and (T1) yield 
$|\partial \tau _{\bu }(\ff ,g)/\partial \x _{\bk }|(0)
=\tau _{\bu }(|\partial (\ff ,g)/\partial \x _{\bk }|)(0)
=|\partial (\ff ,g)/\partial \x _{\bk }|(\bu )\ne 0$. 
Since the existence of a $\tau _{\bu }(\ff ,g)$-map implies 
that of an $(\ff ,g)$-map by Lemma~\ref{lem:(f,g)-map reduction}~(i), 
by replacing $(\ff ,g)$ with $\tau _{\bu }(\ff ,g)$, 
we may assume that 
\begin{equation}\label{eq:pf2}
\left|
\dfrac{\partial (\ff ,g)}{\partial \x _{\bk }}
\right|(0)\ne 0. 
\end{equation}
In view of (T4), 
this change preserves (\ref{eq:pf1}). 
By Lemma~\ref{lem:(f,g)-map reduction} (ii), 
we can also assume that $f_1(0)=\cdots =f_r(0)=g(0)=0$ 
without violating (\ref{eq:pf1}) or (\ref{eq:pf2}).

\nd {\bf Step 2. Lift.} 
Since $k$ is finitely generated over its prime field by assumption, 
we can define the local ring $(R,\fm )$ as in \S~\ref{sect:Lift}. 
We define $\pi _e$ as in (\ref{eq:pi_e}). 
For $f\in \kx $, 
we call $\tilde{f}\in \Rx $ a {\it lift} of $f$ 
if $\pi _0(\tilde{f})=f$ and $\degw \tilde{f}=\degw f$, 
where we identify $R/\fm $ with $k$. 
Each $f\in \kx $ has a 
lift $\tilde{f}\in \Rx $ with $\supp \tilde{f}=\supp f$. 
Hence, 
there exist lifts 
$\tilde{f}_1,\ldots ,\tilde{f}_r,\tilde{g}\in \Rx $ 
of $f_1,\ldots ,f_r$, $g$ 
with $\tilde{f}_1(0)=\cdots =\tilde{f}_r(0)=\tilde{g}(0)=0$, 
respectively. 
We set $\tilde{\ff }:=(\tilde{f}_1,\ldots ,\tilde{f}_r)$.

Now, 
define $D:=\delta _{\tilde{\ff }}\in \Der _{R_p}(R_p[\x ])$ 
as in Example~\ref{ex:JD} (i). 
Since $\tilde{\ff }\in \Rx ^r$, 
we have $D(\Rx )\subset \Rx $. 
Hence, $D(\tilde{g})$ belongs to $\Rx $. 
Note that 
\begin{equation}\label{eq:pf3}
\pi _0(D(\tilde{g}))=\pi _0
\left( 
\left|
\dfrac{\partial (\tilde{\ff },\tilde{g})}{\partial \x _{\bk }}
\right|
\right)
= \left|
\dfrac{\partial (\ff ,g)}{\partial \x _{\bk }}
\right| 
\end{equation}
by the definition of $D$ and Remark~\ref{rem:reduce jacobian}. 
Hence, it follows from (\ref{eq:pf2}) that 
\begin{equation}\label{eq:pf4}
D(\tilde{g})(0)=\left|
\dfrac{\partial (\tilde{\ff },\tilde{g})}{\partial \x _{\bk }}
\right|(0)
\in \{ a\in R\mid \pi _0(a)\ne 0\} =R\sm \fm =R^*. 
\end{equation}
This implies that $D(\tilde{g})\in R[[\x ]]^*$. 
Thus, we have $A':=R_p[\x ][D(\tilde{g})^{-1}]\subset R_p[[\x ]]$.

Set $\Delta :=D(\tilde{g})^{-1}D'\in \Der _{R_p}(A')$, 
where $D'\in \Der _{R_p}(A')$ is the unique extension of $D$. 
Then, 
since $\Q \subset R_p$ by Remark~\ref{rem:RLR} (ii), 
the substitution map 
$\theta _{\Delta }:R_p[[\x ,T]]\to R_p[[\x ,T]]$ 
is defined 
as in (\ref{eq:exp TD map}). 
For $i=1,\ldots ,r+1$, 
we define $\{ q_{i,l}\} _{l=1}^{\infty }$ as follows, 
which is a sequence in $\Rx $ since $D(\Rx )\subset \Rx $: 
\begin{equation}\label{eq:def:hat{h}}
q_{i,1}:=D(x_i)\text{ \ and \ }
q_{i,l}:=D(\tilde{g})D(q_{i,l-1})+(3-2l)q_{i,l-1}D^2(\tilde{g})
\text{ \ for \ } l\ge 2. 
\end{equation}

\begin{lem}\label{lem:q_{i,l}}

\nd {\rm (i)} 
$\theta _{\Delta }(f)=f$ 
for all 
$f\in \{ \tilde{f}_1,\ldots ,\tilde{f}_r,x_{r+2},\ldots ,x_n\} $ 
and $\theta _{\Delta }(\tilde{g})=\tilde{g}+T$. 

\nd {\rm (ii)} 
$\theta _{\Delta }$ belongs to $\mathcal{M}
:=(\sum _{i=1}^nx_iR[[\x ,T]]+TR[[\x ,T]])^{n+1}$.

\nd {\rm (iii)} 
The following $(1)$--$(3)$ 
hold for all $i=1,\ldots ,r+1$ and $l\in \N _+$$:$

$(1)$ $\Delta ^l(x_i)=q_{i,l}D(\tilde{g})^{1-2l}$$;$ \quad 
$(2)$ $q_{i,l}/l!\in \Rx $$;$

$(3)$ $\degw \pi _0(q_{i,l}/l!)=\degw \pi _e(q_{i,l})$, 
where $e:=v_p(l!)$. 

\end{lem}
\begin{proof}
(i) 
The assertion follows from 
(\ref{eq:rem:theta_D}) 
and Example~\ref{ex:JD} (ii).

(ii) 
Set $\frak{F}:=(\tilde{\ff} ,\tilde{g},x_{r+2},\ldots ,x_n,T)$. 
Then, 
since $\tilde{f}_1(0)=\cdots =\tilde{f}_r(0)=\tilde{g}(0)=0$, 
we have $\frak{F}\in \mathcal{M}$. 
In addition, 
(\ref{eq:pf4}) gives $|\partial \frak{F}/\partial (\x ,T)|(0)=
|\partial (\tilde{\ff },\tilde{g})/\partial \x _{\bk }|(0)
\in R^*$. 
Since $\theta _{\Delta }\in 
(\sum _{i=1}^nx_iR_p[[\x ,T]]+TR_p[[\x ,T]])^{n+1}$ 
by definition, 
and 
$\theta _{\Delta }\mathfrak{F}
=(\tilde{\ff} ,\tilde{g}+T,x_{r+2},\ldots ,x_n,T)\in \mathcal{M}$ 
by (i), 
it follows from Corollary~\ref{lem:2023Feb} 
that $\theta _{\Delta }\in \mathcal{M}$.

(iii) 
We prove (1) by induction on $l$. 
The case $l=1$ is clear. 
Assume that $l\ge 2$ and 
$\Delta ^{l-1}(x_i)=q_{i,l-1}D(\tilde{g})^{3-2l}$ 
holds true. 
Then, 
we have
\begin{align*}
&\Delta ^l(x_i)
=\Delta (\Delta ^{l-1}(x_i))
=\Delta (q_{i,l-1}D(\tilde{g})^{3-2l})
=D(\tilde{g})^{-1}D'(q_{i,l-1}D(\tilde{g})^{3-2l}) \\
&\quad 
=D(\tilde{g})^{-1}
\bigl(D(q_{i,l-1})D(\tilde{g})^{3-2l}
+(3-2l)q_{i,l-1}D(\tilde{g})^{2-2l}
D(D(\tilde{g}))\bigr) \\
&\quad =
\bigl( 
D(q_{i,l-1})D(\tilde{g})+(3-2l)
q_{i,l-1}D^2(\tilde{g})
\bigr) D(\tilde{g})^{1-2l}
=q_{i,l}D(\tilde{g})^{1-2l}. 
\end{align*}
Next, by (ii), 
$(\exp T\Delta )(x_i)\in R[[\x ,T]]$ holds for all $i$. 
By (1), we also have 
\begin{equation}\label{eq:pf6}
(\exp T\Delta )(x_i)
=\sum _{l=0}^\infty \dfrac{\Delta ^l(x_i)}{l!}T^l
=x_i+\sum _{l=1}^\infty \dfrac{q_{i,l}}{l!}D(\tilde{g})^{1-2l}T^l 
\text{ \ for \ }i=1,\ldots ,r+1. 
\end{equation}
Thus, $(q_{i,l}/l!)D(\tilde{g})^{1-2l}\in R[[\x ]]$ holds. 
Since 
$D(\tilde{g})\in R[[\x ]]^*$, 
we get $q_{i,l}/l!\in R[[\x ]]$. 
This together with $q_{i,l}\in \Rx $ implies (2). 
Then, 
since $p^e/l!\in R^*$, 
parts (iii) and (i) of Lemma~\ref{lem:reduce coeff} yield 
$\degw \pi _0(q_{i,l}/l!)=\degw \pi _e(p^eq_{i,l}/l!)
=\degw \pi _e(q_{i,l})$, 
proving (3). 
\end{proof}

\nd {\bf Step 3. $(\ff ,g)$-map.} 
By (\ref{eq:pf3}), 
we have $\pi _0(D(\tilde{g}))\in \kx \sm \zs $. 
Hence, 
by Lemma~\ref{lem:q_{i,l}} (iii) (2), we can 
define a homomorphism $E:\kx \to \kxr [[T]]$ of $k$-algebras by 
\begin{equation}\label{eq:pf7}
E(x_i):=x_i+\sum _{l=1}^\infty 
\pi _0\left(\dfrac{q_{i,l}}{l!}\right)
\pi _0(D(\tilde{g}))^{1-2l}T^l
\text{ \ \ for \ \ }i=1,\ldots ,r+1 
\end{equation}
and $E(x_j):=x_j$ for $j=r+2,\ldots ,n$. 
Below, 
we show that $E$ is an $(\ff ,g)$-map.

Set $A:=\Rx [D(\tilde{g})^{-1}]$. 
We first assert that $\theta _\Delta (\Rx )\subset A[[T]]$. 
Indeed, 
(\ref{eq:pf6}) and Lemma~\ref{lem:q_{i,l}} (iii) (2) show that 
$\theta _{\Delta }(x_i)=(\exp T\Delta )(x_i)\in A[[T]]$ 
for $i=1,\ldots ,r+1$, 
and Lemma~\ref{lem:q_{i,l}} (i) gives 
$\theta _{\Delta }(x_j)=x_j\in A$ for $j=r+2,\ldots ,n$. 
Next, since $\pi _0(D(\tilde{g}))\in \kxr ^*$, 
the homomorphism 
$\Rx \stackrel{\pi _0}{\to }(R/\fm )[\x ]=\kx \hookrightarrow \kxr $ 
of $R$-algebras uniquely extends to $A\to \kxr $, 
inducing a natural homomorphism 
$\tilde{\pi }_0:A[[T]]\to \kxr [[T]]$ of $R$-algebras. 
Then, 
the diagram (A) below commutes: 
$$
{\rm (A)}\begin{array}{l}
\xymatrix@C=10mm@R=5mm
{
\Rx  \ar[d]_{\pi _0} \ar[r]^(.45){\theta _\Delta } 
& A[[T]]
 \ar[d]^{\tilde{\pi }_0} \\ 
\kx  \ar[r]^(.41){E} 
& \kxr [[T]].}
\end{array}
\quad 
\begin{array}{ll}
\multirow{2}{*}{\text{(B)}} \!\!& 
\text{\small $\theta _\Delta (\tilde{f}_i)=\tilde{f}_i$ for $i=1,\ldots ,r$ 
and $\theta _\Delta (\tilde{g})=\tilde{g}+T$} \\ 
\smallskip 
 & \text{\small by Lemma~\ref{lem:q_{i,l}}~(i).} \\ 
\multirow{2}{*}{\text{(C)}} \!\!& \small \text{$\pi _0(\tilde{f}_i)=f_i$ for $i=1,\ldots ,r$ and $\pi _0(\tilde{g})=g$ }\\
& \text{\small by the definition of lifts.}
\end{array}
$$
This together with (B) and (C) implies that $E$ satisfies~(M1).

For (M2), we first prove two lemmas. 
Note that (\ref{eq:pf1}) and (\ref{eq:pf3}) imply 
\begin{equation}\label{eq:pf1'}
\degw \pi _0(D(\tilde{g}))+|\w _{\bk }|=\deg d\ff \wedge dg. 
\end{equation}

\begin{lem}\label{lem:gap}
For all $e\in \N $, we have 
$\degw \pi_e(d\tilde{\ff })-\degw d\ff \le e\Lambda _1$ 
and 
$\degw \pi_e(D(\tilde{g}))-\degw \pi _0(D(\tilde{g}))\le e\Lambda _1$. 
\end{lem}
\begin{proof}
By the definition of lifts, 
we have 
$\pi _0\tilde{\ff }=\ff $, 
$\deg \tilde{\ff }=\deg \ff $, 
and $\deg \tilde{g}=\deg g$. 
Now, 
since $\pi _0(d\tilde{\ff })=d\pi _0\tilde{\ff }=d\ff $ 
by Remark~\ref{rem:reduce jacobian}, 
the assertions hold for $e=0$. 
For $e\ge 1$, 
we have $\Lambda _1\le e\Lambda _1$. 
Since 
$\deg \pi _e(d\tilde{\ff })\le 
\deg d\tilde{\ff }\le 
\deg \tilde{\ff }=\deg \ff $ 
by Lemma~\ref{lem:reduce coeff2}~(iv) and 
Remark~\ref{rem:independence} (iii), 
it follows that 
\begin{equation}\label{eq:gap1}
\deg \pi_e(d\tilde{\ff })-\deg d\ff \le \deg \ff -\deg d\ff 
=\Lambda _1-\Lambda _0\le \Lambda _1\le e\Lambda _1. 
\end{equation}
Next, 
we claim that 
$\degw \pi _e(D(\tilde{g}))+|\w _{\bk }|\le \deg \ff +\deg g$ ($\dag $). 
Indeed, 
Lemma~\ref{lem:reduce coeff}~(ii) implies that 
$\degw \pi _e(D(\tilde{g}))+|\w _{\bk }|\le \degw D(\tilde{g})d\x _{\bk }$. 
By (\ref{eq:def:D_ff}), 
we have $D(\tilde{g})d\x _{\bk }\prec d\tilde{\ff }\wedge d\tilde{g}$, 
and so 
$\degw D(\tilde{g})d\x _{\bk }\le \degw d\tilde{\ff }\wedge d\tilde{g}$. 
Moreover, 
Remark~\ref{rem:independence} (iii) gives 
\begin{equation}\label{eq:gap2}
\degw d\tilde{\ff }\wedge d\tilde{g}\le 
\deg \tilde{\ff }+\deg \tilde{g}=\deg \ff +\deg g. 
\end{equation}
Subtracting (\ref{eq:pf1'}) from ($\dag $) 
now yields 
$\degw \pi_e(D(\tilde{g}))-\degw \pi _0(D(\tilde{g}))
\le \Lambda _1\le e\Lambda _1$. 
\end{proof}

\begin{lem}\label{lem:estimation}
Set 
$\Lambda _2:=\degw d\ff +\degw \pi _0(D(\tilde{g}))-|\w _{\bk }|$ 
$(\ddag )$. 
Then, for all $l\in \N _+$, $e\in \N $, and $i=1,\ldots ,r+1$, 
we have 
$\degw \pi _e(q_{i,l})
\le w_i+l\Lambda _2+e\Lambda _1-\degw \pi _0(D(\tilde{g}))$. 
\end{lem}

\begin{proof}
The proof is by induction on $l$. 
For $l=1$, 
by (\ref{eq:def:hat{h}}) and (\ref{eq:def:D_ff}), 
we have 
$q_{i,1}d\x _{\bk }=D(x_i)d\x _{\bk }\prec d\tilde{\ff }\wedge dx_i$. 
Hence, 
Lemma~\ref{lem:reduce coeff2} (i) gives 
$\deg \pi _e(q_{i,1})+|\w _{\bk }|
=\deg \pi _e(q_{i,1}d\x _{\bk })
\le \deg \pi _e(d\tilde{\ff }\wedge dx_i)$. 
By Lemmas~\ref{lem:reduce coeff2}~(viii) and \ref{lem:gap} 
and ($\ddag $), 
we also have 
$\deg \pi _e(d\tilde{\ff }\wedge dx_i)
\le \deg \pi _{e}(d\tilde{\ff })+w_i
\le e\Lambda _1+\degw d\ff +w_i 
=e\Lambda _1+(\Lambda _2-\degw \pi _0(D(\tilde{g}))+|\w _{\bk }|)+w_i$, 
yielding 
$\deg \pi _e(q_{i,1})\le 
w_i+\Lambda _2+e\Lambda _1-\degw \pi _0(D(\tilde{g}))$.

Next, assume that $l\ge 2$. 
By (\ref{eq:def:hat{h}}),  
we can write $q_{i,l}d\x _{\bk }=\omega _1+\omega _2$, 
where 
$$
\omega _1:=D(\tilde{g})\cdot D(q_{i,l-1})d\x _{\bk }
\text{ \ and \ }
\omega _2:=(3-2l)q_{i,l-1}\cdot D^2(\tilde{g})d\x _{\bk }. 
$$
By (\ref{eq:def:D_ff}), 
we have 
$\omega _1\prec D(\tilde{g})d\tilde{\ff }\wedge dq_{i,l-1}=:\eta _1$ 
and 
$\omega _2\prec (3-2l)q_{i,l-1}d\tilde{\ff }\wedge dD(\tilde{g})=:\eta _2$. 
Hence, by Lemma~\ref{lem:reduce coeff2} (ii) and (i), 
there exists $j\in \{ 1,2\} $ such that 
$$
\degw \pi _e(q_{i,l})+|\w _{\bk }|
=\degw \pi _e(q_{i,l}d\x _{\bk })
=\degw \pi _e(\omega _1+\omega _2)
\le \deg \pi _e(\omega _j)\le \deg \pi _e(\eta _j). 
$$ 
By Lemma~\ref{lem:reduce coeff2} (vii) and (iii), 
in both cases $j=1$ and $j=2$, 
there exist $e_1,e_2,e_3\in \N $ such that $e_1+e_2+e_3=e$~and 
\begin{equation}\label{eq:upper bound2}
\deg \pi _e(\eta _j) \le 
\deg \pi _{e_1}(q_{i,l-1})
+\deg \pi _{e_2}(d\tilde{\ff })
+\deg \pi _{e_3}(D(\tilde{g})). 
\end{equation}
Then, 
the induction hypothesis and Lemma~\ref{lem:gap} give 
$$
\deg \pi _{e_1}(q_{i,l-1})
\le w_i+(l-1)\Lambda _2+e_1\Lambda _1-\degw \pi _0(D(\tilde{g})), 
$$
$\degw \pi _{e_2}(d\tilde{\ff })\le e_2\Lambda _1+\degw d\ff$, 
and 
$\degw \pi _{e_3}(D(\tilde{g}))\le e_3\Lambda _1+\degw \pi _0(D(\tilde{g}))$. 
Therefore, 
the right-hand side of (\ref{eq:upper bound2}) is at most 
\begin{align*}
&w_i+(l-1)\Lambda _2 +(e_1+e_2+e_3)\Lambda _1-\degw \pi _0(D(\tilde{g})) 
+\degw d\ff +\degw \pi _0(D(\tilde{g})) \\
&\quad =w_i+l\Lambda _2+e\Lambda _1-\degw \pi _0(D(\tilde{g}))+|\w _{\bk }| 
\quad (\text{by }(\ddag)).
\end{align*}
This proves the required inequality. 
\end{proof}

We now prove (M2). 
For all $i$, 
we have $\deg E_0(x_i)=\deg x_i=w_i=w_i+\nu _{\ff ,g}(0)$. 
Assume that $l\ge 1$, 
and set $e:=v_p(l!)$. 
Then, 
for $i=1,\ldots ,r+1$, we have 
\begin{align*}
&\degw E_l(x_i)=\degw \pi _e(q_{i,l})
+(1-2l)\degw \pi _0(D(\tilde{g})) \quad 
\text{(by (\ref{eq:pf7}), Lemma~\ref{lem:q_{i,l}} (iii) (3))} \\
&\quad 
\le w_i+l\Lambda _2+e\Lambda _1-2l\deg \pi _0(D(\tilde{g}))=
w_i+\nu _{\ff ,g}(l) 
\quad 
\text{(by Lemma~\ref{lem:estimation}, (\ref{eq:pf1'}))}. 
\end{align*}
For $i=r+2,\ldots ,n$, 
we have $\degw E_l(x_i)=\degw 0<w_i+\nu _{\ff ,g}(l)$, 
proving (M2).

This completes the proof of Theorem~\ref{thm:key}, 
and hence of Theorem~\ref{thm:main}.

\begin{rem}\label{rem:slight generalization}\rm

In some special cases, 
we can improve Theorem~\ref{thm:main} 
by taking a smaller constant $\Lambda _1$, 
though this is not required for the rest of the paper.

For example, assume that $\Lambda _0>0$, 
and $f_1^{\w }\approx h^a$ and $g^{\w }\approx h^b$ 
for some $h\in \kx $ and $a,b\in \N _+$ ($*$). 
Due to (T3) and (T4), 
changes to $(\ff ,g)$ in Step 1 preserve ($*$). 
Hence, in Step 2, 
we can choose $\tilde{f}_1$ and $\tilde{g}$ so that 
$\tilde{f}_1^{\w }\in R^*\tilde{h}^a$ and $\tilde{g}^{\w }\in R^*\tilde{h}^b$, 
where $\tilde{h}\in \Rx $ is a lift of $h$. 
By Remark~\ref{rem:independence}~(iv), 
this renders the inequality (\ref{eq:gap2}) strict. 
Since $\Lambda _0>0$ by ($*$), 
the second inequality in (\ref{eq:gap1}) is also strict. 
Then, Lemma~\ref{lem:gap} 
holds with $\Lambda _1$ replaced by 
$\Lambda _1-\ep $ for some $\ep \in \Gamma _+$. 
In this case, 
Theorem~\ref{thm:main} 
holds with $\Lambda _1$ replaced by 
$\Lambda _1-\ep $.

It is noted that ($*$) holds 
if $r=1$ 
and $(f_1^{\w })^b\approx (g^{\w })^a$ for some 
$a,b\in \N _+$ with $\gcd (a,b)=1$ and $p\nmid a $. 
In fact, 
$f_1^{\w }\approx h^a$ and $g^{\w }\approx h^b$ hold for 
$h:=(f_1^{\w })^u(g^{\w })^v$, 
where $u,v\in \Z $ are such that $au+bv=1$. 
We also have $\Lambda _0>0$, 
since $\deg df_1+\deg g>\deg df_1+\deg dg\ge \deg df_1\wedge dg$ 
if $dh=0$, 
and 
$\deg df_1+\deg g=\deg ah^{a-1}dh+\deg g=\deg f_1+\deg g>\deg df_1\wedge dg$ 
otherwise by Remark~\ref{rem:independence} (iv). 
\end{rem}

\section{Consequences of Theorem~\ref{thm:main}}\label{sect:consequence}
\setcounter{equation}{0}

\subsection{Basic concepts}\label{subsect:basic concepts}

Throughout \S ~\ref{sect:consequence}, 
we assume that $p>0$ except for \S ~\ref{sect:cofactor}, 
and let $f,g\in \kx $ satisfy 
$df\wedge dg\ne 0$ 
unless otherwise stated. 
By Remark~\ref{rem:independence} (i), 
$f$ and $g$ are algebraically independent over $k$. 
Set $\bv :=(\degw f,\degw g)\in (\Gamma _+)^2$.

Now, take $\phi \in k[f,g]\sm \zs $. 
Then, 
there exists a unique 
$\theta =\sum _{i,j}c_{i,j}x^iy^j\in k[x,y]$ 
such that 
$\phi =\theta (f,g)$, 
where $c_{i,j}\in k$, 
and $x$ and $y$ are variables. 
We define 
\begin{equation}\label{eq:apparent degree}
\degw ^{f,g}\phi :=\deg _{\bv }\theta 
=\max \{ \degw f^ig^j\mid i,j\in \N ,\ 
c_{i,j}\ne 0\} \ge \deg \phi . 
\end{equation}

\begin{rem}\label{rem:deg^fg}\rm 
$\degw ^{f,g}\phi $ is at least 
$\min \{ \deg f,\deg g\} $ if $\phi \not\in k$, 
and $\deg f$ if $\phi \not\in k[g]$. 
\end{rem}

Write $\phi =P(g)=Q(f)$, where 
\begin{equation}\label{eq:description of q}
P:=\theta (f,T)\in k[f][T]\text{ \ and \ }Q:=\theta (T,g)\in k[g][T]. 
\end{equation}

\begin{rem}\label{rem:chain rule}\rm 
The chain rule yields 

\nd{\rm (i)} 
$d\phi =d\theta (f,g)
=\theta _x(f,g)df+\theta _y(f,g)dg=Q^{(1)}(f)df+P^{(1)}(g)dg$, 
and so

\nd (ii) $df\wedge d\phi =P^{(1)}(g)df\wedge dg$. 
\end{rem}

\begin{lem}\label{lem:PQ}
\nd{\rm (i)} We have $\deg _{\w _g}P=\deg _{\w _f}Q=\degw ^{f,g}\phi $.

\nd{\rm (ii)} 
We have $P^{\w _g}=\theta ^{\bv }(f^{\w },T)$ and 
$Q^{\w _f}=\theta ^{\bv }(T,g^{\w })$.  
\end{lem}

\begin{proof}
Write $P=\sum _jp_jT^j$, 
where $p_j:=\sum _ic_{i,j}f^i$. 
Then, 
(\ref{eq:degw_gP}) yields 
$\deg _{\w _g}P
=\max \{ \deg p_jg^j\mid j\in \N \} 
=\max \{ \max \{ \deg f^ig^j\mid i\in \N ,\,
c_{i,j}\ne 0\} \mid j\in \N \} 
=\degw ^{f,g}\phi $. 
Moreover, 
(\ref{eq:P(T)^w_g1}) gives 
$P^{\w _g}=\sum _{j\in I}p_j^{\w }T^j
=\sum _{j\in I}c_{l_j,j}(f^{\w })^{l_j}T^j$, 
where $l_j\in \N $ satisfies 
$\deg f^{l_j}=\deg p_j$. 
Observe that 
$K:=\{ (l_j,j)\mid j\in I\} 
=\{ (i,j)\in \supp \theta \mid \deg f^ig^j=\deg _{\w _g}P\} $, 
in which the condition 
$\deg f^ig^j=\deg _{\w _g}P$ is equivalent to 
$\deg _{\bv }x^iy^j=\deg _{\bv }\theta $, 
since $\deg _{\w _g}P=\degw ^{f,g}\phi =\deg _{\bv }\theta $ 
as shown above. 
Therefore, we have 
$P^{\w _g}=\sum _{(i,j)\in K}c_{i,j}(f^{\w })^iT^j=\theta ^{\bv }(f^{\w },T)$.

We can prove 
$\deg _{\w _f}Q=\degw ^{f,g}\phi $ and 
$Q^{\w _f}=\theta ^{\bv }(T,g^{\w })$ similarly. 
\end{proof}

We define 
$m_{\w }^{f,g}(\phi ):=\min \{ m_{\w }^g(P),m_{\w }^f(Q)\} $, 
the {\it multiplicity} of $\phi $; 
\begin{align*}
\mathcal{D}(f,g)&:=\{ 
\phi \in k[f,g]\sm \zs \mid 
m_{\w }^{f,g}(\phi )\ge 1\} ; \\
\mathcal{P}(f,g,m)&:=\{ \phi \in k[f,g]\sm \zs \mid 
m_{\w }^{f,g} (\phi )=m,\ \deg \phi \le \deg g\}  
\text{ \ for each \ }m\in \N . 
\end{align*}

\begin{example}
Let $\phi :=P(g)$ in Example~\ref{ex:setting}. 
Then, 
since $\theta =y^4-2(x^3+x^2)y^2+x^4y+x^6$ and $\bv =(2,3)$, 
we have $\theta ^{\bv }=(x^3-y^2)^2$. 
Hence, Lemma~\ref{lem:PQ} (ii) gives 
$Q^{\w _f}=\theta ^{\bv }(T,g^{\w })=(T^3-x_1^6)^2$. 
Thus, we get $m_{\w }^f(Q)=6$ if $p=3$, 
and $m_{\w }^f(Q)=2$ otherwise. 
Together with Example~\ref{ex:setting}, 
we conclude that 
$m_{\w }^{f,g}(\phi )=2$ in all cases. 
\end{example}

\begin{rem}\label{rem:PP}
For $c\in k$ and $\lambda (x)\in k[x]$ with $\deg \lambda (f)<\deg g$, 
set $\hat{f}:=f+c$ and $\hat{g}:=g+\lambda (f)$. 
Note that 
$\hat{f}^{\w }=f^{\w }$, $\hat{g}^{\w }=g^{\w }$, 
and $\hat{\bv }:=(\deg \hat{f},\deg \hat{g})=\bv $ ($*$).

\nd (i) 
Take $0\ne \phi \in k[f,g]=k[\hat{f},\hat{g}]$, 
and $\theta ,\hat{\theta }\in k[x,y]$ 
with $\phi =\theta (f,g)=\hat{\theta }(\hat{f},\hat{g})$. 
Then, 
since $\deg _{\bv }\lambda (x)<\deg _{\bv }y$, 
we have $\theta ^{\bv }=\hat{\theta }(x+c,y+\lambda (x))^{\bv }
=\hat{\theta }^{\bv }$. 
This together with Lemma~\ref{lem:PQ}~(ii) and ($*$) implies that 
$m_{\w }^{f,g}(\phi )=m_{\w }^{\hat{f},\hat{g}}(\phi )$.

\nd (ii) 
Since $\deg g=\deg \hat{g}$, 
(i) implies that 
$\mathcal{P}(f,g,m)=\mathcal{P}(\hat{f},\hat{g},m)$ 
for all $m\in \N $. 
\end{rem}

\begin{rem}\label{rem:(1.3.7)} 
Take $\phi \in k[f,g]\sm \zs $, 
and $\theta \in k[x,y]$ with $\phi =\theta (f,g)$. 
Then, in view of Lemma~\ref{lem:PQ}, 
it follows from Remark~\ref{rem:m(P)=0} that 
\begin{equation}\label{eq:1.3.7}
m_{\w }^{f,g}(\phi )=0
\Leftrightarrow 
\theta ^{\bv }(f^{\w },g^{\w })\ne 0
\Leftrightarrow 
\deg \phi =\degw ^{f,g}\phi 
\Rightarrow 
\phi ^{\w }=\theta ^{\bv }(f^{\w },g^{\w }). 
\end{equation}
This yields the following assertions:

\nd{\rm (i)} 
We have $\phi \in \mathcal{D}(f,g)$ 
if and only if $\degw \phi <\degw ^{f,g}\phi $. 
In this case, 
$f^{\w }$ and $g^{\w }$ 
are algebraically dependent over $k$, 
since $\theta ^{\bv }(f^{\w },g^{\w })=0$.

\nd{\rm (ii)} 
If $\phi ^{\w }\not\in k[f^{\w },g^{\w }]$, 
then we have $\phi \in \mathcal{D}(f,g)$, 
since $\phi ^{\w }\ne \theta ^{\bv }(f^{\w },g^{\w })$. 
\end{rem}

The following well-known lemma holds for arbitrary $f,g\in \kx \sm k$.

\begin{lem}[see, e.g., 
{\rm \cite[Proposition 1.1.4]{JC}}]\label{lem:dependence}
If $f^{\w }$ and $g^{\w }$ are algebraically dependent over $k$, 
then there exist $a,b\in \N _+$ such that 
$\gcd (a,b)=1$ and $(g^{\w })^a\approx (f^{\w })^b$. 
\end{lem}

\begin{rem}\label{rem:dependence}\rm 
Lemma~\ref{lem:dependence} implies that, 
if $\degw f=\degw g$ and $f^{\w }\not\approx g^{\w }$, 
then $f^{\w }$ and $g^{\w }$ are 
algebraically independent over $k$. 
\end{rem}

By Remark~\ref{rem:(1.3.7)}~(i), 
if $\mathcal{D}(f,g)\ne \emptyset $, 
then 
$a,b\in \N _+$ are determined as in Lemma~\ref{lem:dependence}. 
In 
the remainder of 
\S \ \ref{sect:consequence}, 
we prove various key results in the following setting (S): 

\smallskip

\nd 
(S) {\it 
Assume that $\mathcal{D}(f,g)\ne \emptyset $. 
Let $\phi \in \mathcal{D}(f,g)$, 
and let $a,b\in \N _+$ be as in~Lemma~{\rm \ref{lem:dependence}}. 
Set $m:=m_{\w }^{f,g}(\phi )\ge 1$, 
and write~$(\deg f,\deg g)=(a\delta ,b\delta )$, 
where $\delta :=(1/a)\degw f$. 
}

\smallskip

We remark that 
$p\nmid a$ or $p\nmid b$, since $\gcd (a,b)=1$; 
$g^{\w }\not\in k[f^{\w }]$ (resp.\ $f^{\w }\not\in k[g^{\w }]$) 
if and only if $a\ge 2$ (resp.\ $b\ge 2$); 
and $\deg \phi >0$, i.e., $\phi \not\in k$, 
since $\deg \phi <\deg ^{f,g}\phi $.

\begin{lem}\label{lem:lcm}
Assume {\rm (S)}. 
Let $\theta $, $P$, and $Q$ be as above, 
and $\kappa :=(g^{\w })^a/(f^{\w })^b\in k^*$.

\nd{\rm (i)} 
If $p\nmid a$, then 
we have $m_{\w }^f(Q)=p^{v_p(b)}m_{\w }^g(P)$, 
and hence $m=m_{\w }^g(P)$.

\nd{\rm (ii)} 
We have 
$\deg ^{f,g}\phi \ge mab\delta $. 
Moreover, 
exactly one of the following holds$:$

\nd 
$(1)$ 
$\deg ^{f,g}\phi =mab\delta $, 
and 
$\phi \approx (g^a-\kappa f^b)^m+\psi $ 
for some $\psi \in k[f,g]$ with $\deg ^{f,g}\psi <mab\delta $.

\nd 
$(2)$ 
$\degw ^{f,g}\phi \ge (mab+\min \{ a,b\} )\delta $.

\end{lem}
\begin{proof}
By symmetry, 
we may assume that $p\nmid a$ for (ii) as well. 
We first remark the following: 
Let $R$ be an integral domain, 
$L$ an extension field of $Q(R)$, 
and $\mu (T)$ the minimal polynomial of $\alpha \in L$ over $Q(R)$. 
If $\mu (T)$ belongs to $R[T]$, 
then we have 
$\{ \xi (T)\in R[T]\mid \xi (\alpha )=0\} \subset \mu (T)R[T]$. 
In fact, 
dividing $\xi (T)\in R[T]$ by $\mu (T)$ 
yields a quotient and remainder in $R[T]$, 
with zero remainder if $\xi (\alpha )=0$.

Now, 
set $R:=k[f^{\w }]$. 
Then, $T^a-\kappa (f^{\w })^b\in R[T]$ is 
the minimal polynomial of $g^{\w }$ over $Q(R)$, 
since $\gcd (a,b)=1$. 
Since $p\nmid a$, it is separable. 
In addition, 
$P^{\w _g}=\theta ^{\bv }(f^{\w },T)\in R[T]$ 
has the root $g^{\w }$ with multiplicity $m_{\w }^g(P)$ 
by definition. 
Thus, repeatedly applying the remark above, we can write 
$\theta ^{\bv }(f^{\w },T)
=(T^a-\kappa (f^{\w })^b)^{m_{\w }^g(P)}\theta_1(f^{\w },T)$, 
where $\theta _1\in k[x,y]$ with $\theta _1(f^{\w },g^{\w })\ne 0$. 
Then, Lemma~\ref{lem:PQ} (ii) yields 
$Q^{\w _f}=\theta ^{\bv }(T,g^{\w })
=((g^{\w })^a-\kappa T^b)^{m_{\w }^g(P)}\theta _1(T,g^{\w })$. 
This shows that the multiplicity of the root 
$f^{\w }$ in $Q^{\w _f}$ is $p^{v_p(b)}m_{\w }^g(P)$, 
which implies (i). 
For (ii), 
since $m_{\w }^g(P)=m$ by (i), 
we have 
$\theta ^{\bv }=(y^a-\kappa x^b)^m\theta _1$. 
This gives 
$\degw ^{f,g}\phi 
=\deg _{\bv }\theta ^{\bv }
=mab\delta +\deg _{\bv }\theta _1\ge mab\delta $. 
If $\theta _1\in k^*$, 
then (1) holds with $\psi :=\theta ^\circ (f,g)$, 
where $\theta ^\circ :=\theta _1^{-1}\cdot (\theta -\theta ^{\bv })$, 
since 
$\theta _1^{-1}\cdot \theta =(y^a-\kappa x^b)^m+\theta ^\circ $ 
and $\deg _{\bv }\theta ^\circ <\deg _{\bv }\theta =mab\delta $. 
Otherwise, (2) holds since 
$\deg _{\bv }\theta _1\ge 
\min \{ \deg _{\bv }x,\deg _{\bv }y\} =\min \{ a,b\} \delta $. 
\end{proof}

\begin{lem}\label{lem:p nmid m(P)}

Let $g\in \kx \sm k$ and $P\in \kx [T]\sm \zs $.

\nd{\rm (i)} 
If $(P^{\w _g})^{(1)}\ne 0$, then we have 
$\deg _{\w _g}P^{(1)}=\deg _{\w _g}P-\deg g$ 
and $(P^{(1)})^{\w _g}=(P^{\w _g})^{(1)}$. 

\nd{\rm (ii)} 
If $p\nmid m_{\w }^g(P)$, then we have 
$(P^{\w _g})^{(1)}\ne 0$ 
and $m_{\w }^g(P^{(1)})=m_{\w }^g(P)-1$. 
\end{lem}
\begin{proof}
\!(i) 
Differentiating (\ref{eq:P(T)^w_g1}) gives 
$(P^{\w _g})^{(1)}=\sum _{i\in I}ip_i^{\w }T^{i-1}
=\sum _{i\in K}(ip_iT^{i-1})^{\w _g}$, 
where $K:=
\{ i\in \N \mid \deg _{\w _g}ip_iT^{i-1}=\deg _{\w _g}P-\deg g\} $. 
Since $(P^{\w _g})^{(1)}\ne 0$, 
we have $K\ne \emptyset $. 
Then, 
since 
$\max \{ \deg _{\w _g}ip_iT^{i-1}\mid i\in \N \} =\deg _{\w _g}P-\deg g$ 
by (\ref{eq:degw_gP}), 
applying Remark~\ref{rem:principle} to $(ip_iT^{i-1})_{i\ge 0}$ proves (i).

(ii) 
Differentiating (\ref{eq:P(T)^w_g2}) yields 
\begin{equation}\label{eq:P_1}
(P^{\w _g})^{(1)}=(T-g^{\w })^{m_{\w }^g(P)-1}P_1, 
\text{ where }
P_1:=m_{\w }^g(P)P_0+(T-g^{\w })P_0^{(1)}. 
\end{equation}
From $P_1(g^{\w })=m_{\w }^g(P)P_0(g^{\w })\ne 0$, 
we see that 
$P_1\ne 0$, 
and so $(P^{\w _g})^{(1)}\ne 0$ by (\ref{eq:P_1}). 
This implies that $(P^{(1)})^{\w _g}=(P^{\w _g})^{(1)}$ by (i), 
which combined with (\ref{eq:P_1}) and $P_1(g^{\w })\ne 0$ 
proves $m_{\w }^g(P^{(1)})=m_{\w }^g(P)-1$. 
\end{proof}

\subsection{First applications}\label{sect:1st conseq}
By (\ref{eq:Legendre}), 
$m+v_p(m!)=\frac{pm-\ol{m}}{p-1}$ 
holds for all $m\in \N $, 
which equals $m$ if $0\le m<p$. 
For $f_1,\ldots ,f_r\in \{ h\in \kx \mid dh\ne 0\} $, 
we define 
$$
\ep (f_1,\ldots ,f_r):=
\max \{ \deg f_i-\deg df_i\mid i=1,\ldots ,r\} \ge 0. 
$$
Then, 
the case $r=1$ of Theorem~\ref{thm:main} 
can be stated as follows.

\begin{thm}\label{thm:pSUineq0}
For every $P\in k[f][T]\sm \zs $, we have 
$$
\deg P(g)\ge \deg _{\w _g}P
+\frac{pm_{\w }^g(P)-\ol{m_{\w }^g(P)}}{p-1}
(\deg df\wedge dg-\deg f-\deg g)+m_{\w }^g(P)\ep (f).
$$
\end{thm}

In this subsection, 
we derive various results from Theorem~\ref{thm:pSUineq0}. 
To this end, 
for $a,b,m\in \N _+$, 
we define $\lambda _p(a,b):=pab-ab-pa-pb$ 
and 
$$
\xi _p(a,b,m):=
mab-\dfrac{pm-\ol{m}}{p-1}(a+b)
=\dfrac{m\lambda _p(a,b)+\ol{m}(a+b)}{p-1}.
$$

\begin{thm}\label{thm:pSUineq}

Under the setting {\rm (S)},
we have 
\begin{equation}\label{eq:thm:pSUineq}
\deg \phi \ge 
\xi _p(a,b,m)\delta 
+\frac{pm-\ol{m}}{p-1}\deg df\wedge dg
+\left\{ \!
\begin{array}{ll}
0& \text{\rm always}\\
m\ep (f)& \text{\rm if $p\nmid a$}\\
m\ep (f,g)& \text{\rm if $p\nmid a$ and $p\nmid b$}. 
\end{array}
\right.
\end{equation}
\end{thm}

\begin{proof}
The case $p\nmid a$ follows from Theorem~\ref{thm:pSUineq0} applied 
to $P$ in (\ref{eq:description of q}), 
since 
$\deg _{\w _g}P=\deg ^{f,g}\phi \ge mab\delta $ and $m_{\w }^g(P)=m$ 
by Lemmas~\ref{lem:PQ} (i) and \ref{lem:lcm}. 
When $p\nmid b$, 
the same result with $\ep (f)$ replaced by $\ep (g)$ holds. 
Then, 
the remaining cases follow, 
since $p\nmid a$ or $p\nmid b$, 
$\ep (f)\ge 0$ and $\ep (g)\ge 0$, 
and $\max \{ \ep (f),\ep (g)\} =\ep (f,g)$. 
\end{proof}

\begin{example}\label{cor:nonsingular}\rm

Assume (S) and $a=2$, 
and consider the following four cases: 
\begin{equation}\label{eq:a--d}
\begin{aligned}
&\text{\rm (a) 
$b\ge 3$, $b$ is odd, and $m=1$,}\quad & &
\text{\rm (b) $(b,m)=(3,2)$ and $p\ge 3$,} \\
&\text{\rm (c) $(b,p,m)=(3,7,7)$,} & &
\text{\rm (d) $(b,p,m)=(7,3,3)$}.
\end{aligned}
\end{equation}
Then, Theorem~\ref{thm:pSUineq} 
yields the following inequalities:

\nd 
(a) $\deg \phi \ge (b-2)\delta +\deg df\wedge dg 
+\left\{ \!
\begin{array}{ll}
\ep (g)& \text{\rm if $p=2$}\\
\ep (f)& \text{\rm if $p\ge 3$}\\
\ep (f,g)& \text{\rm if $p\ge 3$ and $p\nmid b$}. 
\end{array}
\right.$

\nd 
(b) $\deg \phi\ge 2(\delta +\deg df\wedge dg)
+\left\{ \!
\begin{array}{ll}
2\ep (f)& \! \text{\rm if $p=3$}\\
2\ep (f,g)& \! \text{\rm if $p\ge 5$}. 
\end{array}
\right.$

\nd 
(c) $\deg \phi\ge 2\delta +8\deg df\wedge dg+7\ep (f,g)$.

\nd 
(d) $\deg \phi\ge 6\delta +4\deg df\wedge dg+3\ep (f,g)$. 
\end{example}

\begin{lem}\label{lem:nonsingular}
Assume {\rm (S)}, 
$a=2$, 
and one of {\rm (a)}--{\rm (d)} in {\rm (\ref{eq:a--d})} holds. 
If $\deg \phi \le b\delta $, i.e., 
$\phi \in \mathcal{P}(f,g,m)$, 
then {\rm (1)} in Lemma~{\rm \ref{lem:lcm} (ii)} holds for $\phi $. 
\end{lem}

\begin{proof}
The proof of Theorem~\ref{thm:pSUineq} shows that, 
if (2) in Lemma~\ref{lem:lcm} (ii) holds for $\phi $, 
then the inequality (\ref{eq:thm:pSUineq}) 
with $\min \{ a,b\} \delta $ added to the right-hand side still holds.

Now, 
since $\deg \phi \le b\delta $ by assumption, 
none of the inequalities in Example~\ref{cor:nonsingular} 
remain valid if $\min \{ a,b\} \delta =2\delta $ 
is added to their right-hand sides. 
Hence, 
(1) in Lemma~\ref{lem:lcm} (ii) must hold for $\phi $. 
\end{proof}

\begin{rem}\label{rem:1/7}\rm 
Assume that 
$(\deg f,\deg g)=(2\delta ,3\delta )$ 
for some $\delta \in \Gamma _+$.

\nd 
(i) 
If $p=7$ and $\mathcal{P}(f,g,7)\ne \emptyset $, 
then 
we have 
$\deg df\wedge dg\le (1/8)\delta $ and 
$\ep (f,g)<(1/7)\delta $. 
In fact, 
applying Example~\ref{cor:nonsingular} (c) to $\phi \in \mathcal{P}(f,g,7)$ 
yields 
$8\deg df\wedge dg+7\ep (f,g)\le \deg \phi -2\delta \le \delta $, 
since $\deg \phi \le \deg g=3\delta $. 
This implies the claimed inequalities, 
as $\ep (f,g)\ge 0$ and $\deg df\wedge dg>0$.

\nd (ii) 
Assume that $\mathcal{P}(f,g,2)\ne \emptyset $. 
Then, 
similarly to (i), 
it follows from Example~\ref{cor:nonsingular}~(b) that 
$\deg df\wedge dg\le (1/2)\delta $ if $p\ge 3$, 
and $\ep (f,g)<(1/2)\delta $ if $p\ge 5$. 
\end{rem}

Next, for $u,v\in \R $, 
we define 
$I(u,v):=\{ (a,b)\in \R ^2\mid a\ge u,\ b\ge v\}$.

\begin{rem}\label{rem:increasing function}\rm 
Let $\alpha ,u,v\in \R $ with $\alpha >0$. 
Then, the following polynomial function 
$\varphi (x,y)$ defined on $I(u,v)$ is non-decreasing in $x$ and $y$: 
$$
\varphi (x,y)=\alpha (xy-vx-uy)=\alpha (x-u)(y-v)-\alpha uv. 
$$
\end{rem}

\begin{cor}\label{cor:pSU0}

Under the setting {\rm (S)},
the following assertions hold.

\nd{\rm (i)} 
We have $\deg \phi \ge 
(ab-a-b)\delta +\deg df\wedge dg=:M_0$ 
if $1\le m<p$, 
and $\deg \phi >pM_1$ if $m\ge p$, 
where 
\begin{align*}
M_1:=\left(ab-\dfrac{p}{p-1}(a+b)\right) \delta 
+\dfrac{p}{p-1}\deg df\wedge dg. 
\end{align*}

\nd{\rm (ii)} 
If $c\in \Q $ satisfies $\deg \phi \le c\delta $, 
then $\zeta _i(a,b,c)<0$ holds for some $i\in \{ 0,1\} $, 
where $\zeta _0(a,b,c):=ab-a-b-c$ and 
\begin{equation}\label{eq:zeta}
\zeta _1(a,b,c):=ab-\dfrac{p}{p-1}(a+b)-\dfrac{c}{p}
=\zeta _0(a,b,c)-\dfrac{a+b}{p-1}+\dfrac{(p-1)c}{p}. 
\end{equation}

\nd{\rm (iii)} 
If $p\ge 7$, 
and $f^{\w }\not\in k[g^{\w }]$ and $g^{\w }\not\in k[f^{\w }]$, 
then we have $\deg \phi >\deg df\wedge dg$. 
\end{cor}

\begin{proof}
(i) 
Since $\deg \phi >0$ by (S), 
we may assume that $M_0>0$ if $1\le m<p$, 
and $M_1>0$ if $m>p$. 
Then, Theorem~\ref{thm:pSUineq} yields 
$\deg \phi \ge mM_0\ge M_0$ if $1\le m<p$ 
since $\frac{pm-\ol{m}}{p-1}=m$, 
and $\deg \phi \ge 
mM_1+\frac{\ol{m}}{p-1}((a+b)\delta -\deg df\wedge dg)>mM_1\ge pM_1$ 
if $m\ge p$ 
since $(a+b)\delta =\deg f+\deg g>\deg df\wedge dg$ by 
Remark~\ref{rem:independence} (iv).

(ii) 
Since $\deg df\wedge dg>0$, 
the assertion follows from (i).

(iii) 
$f^{\w }\not\in k[g^{\w }]$, $g^{\w }\not\in k[f^{\w }]$, 
and $\gcd (a,b)=1$ imply that $(a,b)\in I(2,3)\cup I(3,2)$. 
Since 
$I(2,3)\cup I(3,2)
\subset I(\frac{p}{p-1},\frac{p}{p-1})\subset I(1,1)$ 
and $p\ge 7$, 
it follows from 
Remark~\ref{rem:increasing function} that 
$ab-\frac{p}{p-1}(a+b)\ge 6-\frac{5p}{p-1}
\ge 6-\frac{5\cdot 7}{6}>0$ 
and $ab-a-b\ge 6-5>0$. 
Hence, the assertion follows from (i). 
\end{proof}

\begin{lem}\label{lem:increasing function}
Let $\xi _0$ and $\xi _1$ be as in 
Corollary~{\rm \ref{cor:pSU0}~(ii)}, 
and $i\in \{ 0,1\} $.

\nd{\rm (i)} 
Let $(u,v)\in \{ (4,3),(5,2)\} $. 
If $p\ge 5$, 
then $\zeta _i(x,y,4y/3)>0$ holds on $I(u,v)$.

\nd{\rm (ii)} 
Let $(u,v,z)\in \{ (7,3,2x/s),(7,3,x+y),(3,4,3x/2)\} $, 
where $s \ge 3$. 
If $p\ge 7$, 
then $\zeta _i(x,y,z)>0$ holds on $I(u,v)$. 
\end{lem}

\begin{proof}
(i) 
By definition, we have 
$\zeta _0(x,y,\frac{4}{3}y)=xy-x-\frac{7}{3}y$ and 
$\zeta _1(x,y,\frac{4}{3}y)=xy-\frac{p}{p-1}x-(\frac{p}{p-1}+\frac{4}{3p})y$. 
Since 
$I(u,v)\subset 
I(\frac{7}{3},1)\cap I(\frac{p}{p-1}+\frac{4}{3p},\frac{p}{p-1})$, 
Remark~\ref{rem:increasing function} implies that 
$\zeta _i(x,y,\frac{4}{3}y)\ge \zeta _i(u,v,\frac{4}{3}v)$ 
on $I(u,v)$. 
Since 
$\frac{u+v}{p-1}\le \frac{7}{4}<\frac{32}{15}\le \frac{p-1}{p}\frac{4}{3}v$, 
(\ref{eq:zeta}) gives 
$\zeta _i(u,v,\frac{4}{3}v)\ge \zeta _0(u,v,\frac{4}{3}v)$. 
By computation, 
we have $\zeta _0(u,v,\frac{4}{3}v)>0$.

(ii) 
For the first two cases, 
take $(a,b)\in I(7,3)$. 
Then, 
since 
$\frac{2}{s}a<a+b$ and $\frac{p}{p-1}+\frac{1}{p}<2$, 
we have 
$\zeta _i(a,b,\frac{2}{s}a)>\zeta _i(a,b,a+b)\ge ab-2(a+b)$. 
Moreover, 
$ab-2(a+b)\ge 7\cdot 3-2(7+3)>0$ 
holds 
by Remark~\ref{rem:increasing function}. 
The last case can be verified as in (i), 
using 
$I(3,4)\subset 
I(1,\frac{5}{2})\cap I(\frac{p-1}{p},\frac{p-1}{p}+\frac{3}{2p})$ 
and 
$\frac{3+4}{p-1}\le \frac{7}{6}<\frac{27}{7}
\le \frac{p-1}{p}\frac{3\cdot 3}{2}$. 
\end{proof}

\begin{thm}\label{prop:Exc29}
Assume that $p\ge 5$, 
$f^{\w }\not\in k[g^{\w }]$, and 
$\deg f>\gamma $, where $\gamma :=(2/3)\deg g$. 
If there exists $\phi \in \mathcal{D}(f,g)$ 
such that $\deg \phi =2\gamma $, 
then we have either 
{\rm (a)} $\deg f=(9/4)\gamma $, 
or {\rm (b)} 
$\deg f=(9/8)\gamma $ and $\deg df\wedge dg+\ep (f,g)\le (1/8)\gamma $. 
\end{thm}
\begin{proof}
The setting (S) holds with $b\ge 2$, 
since $f^{\w }\not\in k[g^{\w }]$. 
Since $a\delta =\deg f>\gamma =(2/3)\deg g=(2b/3)\delta $, 
it follows that $3a>2b\ge 4$, 
so $a\ge 2$. 
We assert that $a\ge 3$. 
Indeed, if $a=2$, then $\gcd(a,b)=1$ implies $b\ge 3$, 
which contradicts $3a>2b$.

Now, since 
$\deg \phi =2\gamma =(4b/3)\delta $ 
by assumption, 
Corollary~\ref{cor:pSU0} (ii) yields 
$\zeta _i(a,b,4b/3)<0$ for some $i\in \{ 0,1\} $. 
Hence, 
we obtain $(a,b)\not\in I(4,3)\cup I(5,2)$ 
by Lemma~\ref{lem:increasing function} (i). 
Thus, together with $a\ge 3$, $b\ge 2$, $3a>2b$, and $\gcd (a,b)=1$, 
we conclude that $(a,b)\in \{ (3,2),(3,4)\} $, i.e., 
$\deg f=(3a/(2b))\gamma \in \{ (9/4)\gamma ,(9/8)\gamma \} $.

If $(a,b)=(3,4)$, 
then since $p\ge 5$, 
Theorem~\ref{thm:pSUineq} yields 
$(16/3)\delta =\deg \phi 
\ge (12m-7\frac{pm-\ol{m}}{p-1})\delta 
+\frac{pm-\ol{m}}{p-1}\deg df\wedge dg+m\ep (f,g)
>(12m-7\frac{pm}{p-1})\delta >3m\delta $ ($*$). 
This gives $m=1$. 
Then, ($*$) reduces to 
$\deg df\wedge dg+\ep (f,g)\le (16/3-5)\delta =(1/8)\gamma $. 
\end{proof}

The following result is also deduced from 
Theorem~\ref{thm:pSUineq0}.

\begin{lem}\label{lem:deg P'(g)}
Assume {\rm (S)}, 
$p\nmid a$, and $m\in \{ 1,\ldots ,p-1\} $, 
and let $P$ be as in~{\rm (\ref{eq:description of q})}. 

\nd{\rm (i)} 
We have $\deg P^{(1)}(g)\ge \eta (a,b,m)\delta $, 
where $\eta (a,b,m):=m(ab-a-b)+a$.

\nd{\rm (ii)} 
If $b>a\ge 2$, 
then we have $\deg P^{(1)}(g)\ge b\delta $ 
and $\deg df\wedge d\phi \ge b\delta +\deg df\wedge dg$.

\nd{\rm (iii)} 
If $b>a\ge 2$ and $p\nmid b$, 
and if $\deg \phi \le b\delta $ and $m=1$, 
i.e., $\phi \in \mathcal{P}(f,g,1)$, 
then we have 
$\deg f-\deg df=\deg g-\deg dg$. 
\end{lem}
\begin{proof}
(i) Since $p\nmid a$, 
Lemma~\ref{lem:lcm} (i) gives $m_{\w }^g(P)=m$, 
so $p\nmid m_{\w }^g(P)$. 
Hence, 
Lemmas~\ref{lem:p nmid m(P)} (ii) and (i) and \ref{lem:lcm} (ii) yield 
$m_{\w }^g(P^{(1)})=m-1$ 
and $\deg _{\w _g}P^{(1)}=\deg _{\w _g}P-b\delta 
\ge (mab-b)\delta$. 
Thus, 
since $\frac{p(m-1)-\ol{m-1}}{p-1}=m-1$, 
it follows from Theorem~\ref{thm:pSUineq0} that 
$\deg P^{(1)}(g)
\ge (mab-b)\delta -(m-1)(\deg f+\deg g)
=\eta (a,b,m)\delta $.

(ii) 
By Remark~\ref{rem:increasing function}, 
we have $\eta (a,b,m)-b\ge \eta (2,3,m)-3=m-1\ge 0$, 
since $(a,b)\in I(2,3)\subset I(\frac{m+1}{m},\frac{m-1}{m})$. 
Hence, the first part follows from (i). 
The last part is due to Remark~\ref{rem:chain rule} (ii).

(iii) 
We have $\deg P^{(1)}(g)\ge b\delta $ by (ii), 
$b\delta \ge \deg \phi $ by assumption, 
and $\deg \phi \ge \deg d\phi $ by Remark~\ref{rem:independence} (ii), 
so we get $\deg P^{(1)}(g)dg>\deg d\phi $. 
Hence, it follows from 
Remark~\ref{rem:chain rule} (i) 
that $\deg P^{(1)}(g)dg=\deg Q^{(1)}(f)df$. 
Thus, 
it suffices to verify that 
$\deg P^{(1)}(g)g=\deg Q^{(1)}(f)f$. 
From $p\nmid a$ and $m=1$, 
we obtain $m_{\w }^g(P)=1$ via Lemma~\ref{lem:lcm} (i). 
Hence, 
parts (ii) and (i) of Lemma~\ref{lem:p nmid m(P)} yield 
$m_{\w }^g(P^{(1)})=m_{\w }^g(P)-1=0$ and 
$\deg _{\w _g}P^{(1)}+\deg g=\deg _{\w _g}P$. 
Since $m_{\w }^g(P^{(1)})=0$ 
implies that $\deg _{\w _g}P^{(1)}=\deg P^{(1)}(g)$ 
by Remark~\ref{rem:m(P)=0}, 
it follows that $\deg P^{(1)}(g)g=\deg _{\w _g}P$. 
Since $p\nmid b$ by assumption, 
$\deg Q^{(1)}(f)f=\deg _{\w _f}Q$ holds analogously. 
Then, 
the required equality follows from Lemma~\ref{lem:PQ} (i). 
\end{proof}

\subsection{Possible degrees and multiplicities}

We first analyze the set 
$$
\mathcal{S}_p:=\{ (a,b,m)\in (\N _+)^3\mid 
\gcd (a,b)=1,\ b>a\ge 2,\ 
\xi _p(a,b,m)<b\} . 
$$

\begin{definition}
Assume that $p\in \{ 2,3,5\} $. 
We say that $(a,b)$ is {\it singular} if 

$\bullet $ 
$(a,b)\in \{ (2,s)\mid s\ge 3\text{ and $s$ is odd}\} \cup 
\{ (3,4),(3,5)\} $ 
when $p=2$;  

$\bullet $ 
$(a,b)\in \{ (2,3),(2,5)\} $ 
when $p=3$; 

$\bullet $ 
$(a,b)=(2,3)$ when $p=5$. 
\end{definition}

If $(a,b)$ is singular, 
then $\lambda _p(a,b)$ is negative, 
since $\lambda _2(a,b)=(a-2)(b-2)-4$ and 
$\lambda _p(2,b)=(p-2)(b-2)-4$. 
This implies that 
$(a,b,p^e)\in \mathcal{S}_p$ for all $e\gg 0$.

\begin{prop}\label{thm:admissible}

Let $(a,b,m)$ be an element of $\mathcal{S}_p$.

\nd{\rm (i)} 
If $p=2$, 
then $(a,b)$ is singular.

\nd{\rm (ii)} 
If $p\ge 3$ or $m=1$, 
then $a=2$, $b\ge 3$, and $b$ is odd.

\nd{\rm (iii)} 
If $p\ge 3$ and $1\le m<p$, 
then we have $m=1$ or $(b,m)=(3,2)$.

\nd{\rm (iv)} 
Assume that $(a,b)$ is not singular. 
Then, 
we have $p\ge 3$ by {\rm (i)}, 
and so $a=2$, $b\ge 3$, and $b$ is odd by {\rm (ii)}. 
In each case, 
$m$ must be as follows.

$(1)$ 
If $p\ge 5$ and $b\ge 5$, or $p=3$ and $b\ge 9$, 
then $m=1$.

$(2)$ 
If $p=3$ and $b=7$, 
then $m\in \{ 1,3\} $.

$(3)$ 
If $p=7$ and $b=3$, 
then $m\in \{ 1,2,7\} $.

$(4)$ 
If $p\ge 11$ and $b=3$, 
then $m\in \{ 1,2\} $.

\begin{center}
{\rm Table 1\quad 
\begin{tabular}{|c|cc|c|c|}
\hline
 & $p=3$ & \multicolumn{1}{|c|}{$p=5$} & $p=7$ & $p\ge 11$ \\ \hline
$b=3$ & \multirow{2}{*}{singular}& &(3) $m\in \{ 1, 2, 7\} $ & 
(4) $m\in \{ 1, 2\} $ \\ \cline{1-1}\cline{3-5}
$b=5$ &  & \multicolumn{3}{|c|}{}\\ \cline{1-2}
$b=7$ &(2) $m\in \{ 1, 3\} $&  \multicolumn{3}{|c|}{(1) $m=1$}\\ \cline{1-2}
$b\ge 9$ & \multicolumn{4}{|c|}{}\\ \hline 
\end{tabular}

}

\end{center}

\nd 
{\rm Remark: (a)--(b) in (\ref{eq:a--d}) cover Cases (1)--(4).}

\end{prop}

\begin{proof}
We first remark the following: 
(1$^*$) For (i) and (ii), 
it suffices to verify that 
$(a,b)\not\in I(3,7)\cup I(4,5)$ 
and $(a,b)\not\in I(3,4)$, respectively, 
since $\gcd (a,b)=1$ and $b>a\ge 2$. 
(2$^*$) 
By Remark~\ref{rem:increasing function}, 
for fixed $p$ and $l\in \N _+$, 
the function $\xi _p'(x,y,l):=\xi _p(x,y,l)-y$ 
is non-decreasing in $x$ and $y$ on 
$I_{p,l}:=I(\frac{p}{p-1}+\frac{1}{l},\frac{p}{p-1})$,~since 
$$
\xi_p'(x,y,l)
=l(xy-vx-(v+l^{-1})y)
\text{ \ \ with \ \ }
v:=\dfrac{p-\ol{l}/l}{p-1}<
\dfrac{p}{p-1}. 
$$
(3$^*$) 
$a$, $b$, and $m$ satisfy $\xi _p'(a,b,m)<0$, 
since $\xi _p(a,b,m)<b$.

Now, 
observe that, for all $l\in \N _+$, we have 
$I(3,7)\cup I(4,5)\subset I_{2,l}$, 
$\xi _2'(3,7,l)=l+10\ol{l}-7>0$, 
and $\xi _2'(4,5,l)=2l+9\ol{l}-5>0$. 
Hence, 
(1$^*$)--(3$^*$) yield (i). 
If $p\ge 3$ or $m=1$, 
then 
we have 
$(5p-12)m\ge (5p-12)\ol{m}$, 
and so $\xi _p'(3,4,m)
=\frac{1}{p-1}((5p-12)m+7\ol{m})-4
\ge 5\ol{m}-4>0$. 
Since $I(3,4)\subset I_{p,m}$ also holds, 
(ii) follows from 
(1$^*$)--(3$^*$). 
For (iii), note that $a=2$ and $b\ge 3$ by (ii). 
Since $1\le m<p$ implies that 
$\xi _p'(2,b,m)=m(b-2)-b=(m-1)(b-2)-2$, 
we get $m=1$ or $(b,m)=(3,2)$ by (3$^*$). 
For the first case of (iv) (1), 
if $p\ge 5$ and $l\ge 2$, 
then we have 
$I(2,5)\subset I_{p,l}$ and 
$\xi _p'(2,5,l)=\frac{1}{p-1}((p-5)(3l-5)+5l+7\ol{l}-20)>0$. 
By (2$^*$), 
this implies that 
$\xi _p'(2,b,l)>0$ for all $p,b\ge 5$ and $l\ge 2$, 
proving $m=1$. 
For all $l\ge 2$, 
we have $I(2,9)\subset I_{3,l}$ and $\xi _3'(2,9,l)>0$, 
since 
$\xi _3'(2,b,l)=\frac{1}{2}((b-6)l+(b+2)\ol{l}-2b)$~($\dag $). 
Hence, 
the last case of (1) follows similarly. 
(2)--(4) are verified using ($\dag $) and 
$\xi _p'(2,3,l)=\frac{1}{p-1}((p-7)(l-3)+l+5\ol{l}-18)$. 
\end{proof}

We now present the main result of this subsection.

\begin{thm}\label{ex:ineq1}
If $p\ge 7$, 
$\deg f\le \deg g$, $g^{\w }\not\in k[f^{\w }]$, 
and $\mathcal{P}(f,g,m)\ne \emptyset $ for some $m\ge 1$, 
then there exists an odd $b\ge 3$ such that 
the following {\rm (i)}--{\rm (iii)}~hold$:$

\nd{\rm (i)} 
We have $(g^{\w })^2\approx (f^{\w })^b$.

\nd{\rm (ii)} 
Exactly one of $m=1$, $(b,m)=(3,2)$, or $(b,p,m)=(3,7,7)$ holds.

\nd{\rm (iii)} 
$\deg \phi \ge (b-2)\delta +\deg df\wedge dg +\ep (f)$ 
holds for all $\phi \in \mathcal{P}(f,g,m)$, 
where $\delta :=(1/2)\deg f$. 
\end{thm}

\begin{proof}
Take any $\phi \in \mathcal{P}(f,g,m)$. 
Then, (S) holds with $\deg \phi \le b\delta $. 
Since $\deg \phi >\xi _p(a,b,m)\delta $ by Theorem~\ref{thm:pSUineq}, 
we get $\xi _p(a,b,m)<b$. 
Since $\deg f\le \deg g$, $g^{\w }\not\in k[f^{\w }]$, 
and $\gcd (a,b)=1$ imply that $b>a\ge 2$, 
we obtain $(a,b,m)\in \mathcal{S}_p$. 
Therefore, 
since $p\ge 7$, 
Proposition~\ref{thm:admissible} (ii) and (iv) yield (i) and (ii), 
respectively. 
By (ii), 
the inequality (a), (b), or (c) in Example~\ref{cor:nonsingular} holds. 
This implies (iii). 
\end{proof}

\begin{cor}\label{claimm:ineq1}
Assume that $p\ge 7$, 
$\deg f\le \deg g$, and $g^{\w }\not\in k[f^{\w }]$. 
Then, 
$\deg \phi >(1/2)\deg f$ holds for all 
$\phi \in k[f,g]\sm k$. 
\end{cor}

\begin{proof}
Since $\phi \not\in k$ and $\deg f\le \deg g$, 
Remark~\ref{rem:deg^fg} gives $\deg ^{f,g}\phi \ge \deg f$. 
Now, if $\deg \phi \le (1/2)\deg f$, 
then we have $\deg \phi <\deg ^{f,g}\phi $ and $\deg \phi <\deg g$. 
This implies that 
$\phi \in \mathcal{P}(f,g,m)$ for some $m\ge 1$. 
Hence, 
Theorem~\ref{ex:ineq1}~(iii) yields 
$\deg \phi >(b-2)\delta \ge \delta =(1/2)\deg f$, 
a contradiction. 
\end{proof}

\begin{rem}\label{rem:ep<1/7}\rm 
Assume that $p\ge 7$, 
$(\deg f,\deg g)=(2\delta ,3\delta )$ 
for some $\delta \in \Gamma _+$, 
and $\mathcal{P}(f,g,m)\ne \emptyset $ for some $m\ge 2$. 
Then, 
it follows from Theorem~\ref{ex:ineq1} (ii) that $m=2$ or $p=m=7$. 
In either case, 
$\deg df\wedge dg\le (1/2)\delta $ holds 
by Remark~\ref{rem:1/7} (ii) and~(i). 
\end{rem}

\subsection{Modification theorem}\label{sect:modifi}

We remark the following, where $F\in \kx ^3$: \\
(1) If $dF\ne 0$, 
then $dg\wedge df_2\ne 0$ holds for all $g\in f_1+kf_3$, 
since $dg\wedge df_2\wedge df_3=dF\ne 0$. 

\nd 
(2) The condition 
$(\sharp )$ implies $(\flat )$, 
where $\delta :=(1/2)\deg f_2$, 
and $(\sharp )$ and $(\flat )$ are defined as follows:

$(\sharp )$ 
$2\delta <\deg f_3<\deg f_1=3\delta $; \ 
$(\flat )$ 
$\deg f_3<\deg f_1=3\delta $ 
and $f_3^{\w },(f_1f_3)^{\w }\not\in k[f_2^{\w }]$.

\nd 
In fact, 
$\deg f_1,\deg f_2\in \Z \delta $ and $\deg f_3\not\in \Z \delta $ 
imply that $f_3^{\w },(f_1f_3)^{\w }\not\in k[f_2^{\w }]$.

The purpose of this subsection is to prove the following theorem.

\begin{thm}\label{thm:modification}
Assume that $F\in \kx ^3$ satisfies $dF\ne 0$ and $(\flat )$.

\nd{\rm (i)} 
If $p\ge 3$, 
then there exists $g\in f_1+kf_3$ such that 
$\mathcal{P}(f_2,g,1)=\mathcal{P}(f_2,g,2)=\emptyset $.

\nd{\rm (ii)} 
If $p=7$, 
$\deg df_1\wedge df_2>(1/2)\delta $, and 
$\mathcal{P}(f_2,f_1,m)\ne \emptyset $ for some $m\ge 1$, 
then 
$\deg df_2\wedge d\phi >3\delta +\deg dg\wedge df_2$ 
holds for all 
$g\in f_1+k^*f_3$ and $\phi \in \mathcal{P}(f_2,g,7)$. 

\end{thm}

\begin{rem}\label{rem:parity}\rm 
We record the following simple remark for later use: 
Assume that 
$(\deg f,\deg g)=(2\delta ,b\delta )$ 
for some $\delta \in \Gamma _+$ and odd $b\ge 1$. 
Then, 
for every $(q_1,q_2)\in k[f]^2\sm \zs $, 
we have $\deg q_1g\ne \deg q_2$ 
and $(q_1g+q_2)^{\w }\in \{ (q_1g)^{\w },q_2^{\w }\} $. 
\end{rem}

\begin{lem}\label{prop:modification2}
Let $G\in \kx ^3$ be such that $dG\ne 0$, 
and set $\delta :=(1/2)\deg g_2$ and $h:=g_1+g_3$. 
Assume that $p\ge 3$, 
and $\deg g_3<\deg g_1=b\delta $ 
for some odd~$b\ge 3$.

\nd{\rm (i)} 
If $g_3^{\w },(g_1g_3)^{\w }\not\in k[g_2^{\w }]$ 
and $\mathcal{P}(g_2,g_1,1)\ne \emptyset $, 
then we have $\mathcal{P}(g_2,h,1)=\emptyset $.

\nd{\rm (ii)} 
If $b=3$ and $\mathcal{P}(g_2,g_1,2)\ne \emptyset $, 
then we have $\mathcal{P}(g_2,h,2)=\emptyset $. 

\end{lem}
\begin{proof}
Since $\deg g_1>\deg g_3$, 
we have $\deg h=\deg g_1=b\delta $ 
and $h^{\w }=g_1^{\w }$ ($\dag $).

(i) 
By assumption, 
we can take $\phi \in \mathcal{P}(g_2,g_1,1)$. 
Suppose the contrary, 
that we can also take 
$\hat{\phi }\in \mathcal{P}(g_2,h,1)$. 
Then, 
we have $\deg (\phi -\hat{\phi })\le \deg g_1$ ($\ddag $). 
Moreover, 
Lemma~\ref{lem:nonsingular} (a) implies that 
$\phi \approx g_1^2-\kappa g_2^b+q_{1}g+q_{0}$ 
and $\hat{\phi }\approx h^2-\kappa g_2^b+\hat{q}_{1}h+\hat{q}_{0}$ 
for some 
$q_{1},\hat{q}_{1}\in \sum _{j=0}^{(b-1)/2}kg_2^j$, 
$q_{0},\hat{q}_{0}\in \sum _{j=0}^{b-1}kg_2^j$, 
and a common 
$\kappa \in k^*$ due to ($\dag $). 
By scaling, 
we can replace $\approx $ with $=$. 
Now, 
set $\phi _1 :=h^2-\kappa g_2^b+q_{1}h+q_{0}$. 
Then, 
since $\phi _1 -\hat{\phi }=(q_{1}-\hat{q}_{1})h+q_{0}-\hat{q}_{0}$, 
Remark~\ref{rem:parity} and ($\dag $) yield 
$(\phi _1 -\hat{\phi })^{\w }
\in \{ ((q_{1}-\hat{q}_{1})h)^{\w },(q_{0}-\hat{q}_{0})^{\w }\} 
\subset k[g_2^{\w }]g_1^{\w }\cup k[g_2^{\w }]$. 
Since $h=g_1+g_3$, 
we have $\phi _1=2g_1g_3+(g_3+q_{1})g_3+\phi $. 
Since $p\ge 3$ and 
$\deg (g_3+q_{1})<\deg g_1$, 
this gives $(\phi _1 -\hat{\phi })^{\w }=2(g_1g_3)^{\w }$ 
by ($\ddag $). 
Therefore, $g_3^{\w }$ or $(g_1g_3)^{\w }$ 
belongs to $k[g_2^{\w }]$, 
a contradiction.

(ii) 
By Remark~\ref{rem:1/7} (ii), 
the assumptions 
$p\ge 3$ and $\mathcal{P}(g_2,g_1,2)\ne \emptyset $ 
imply that $\deg dg_1\wedge dg_2\le (1/2)\delta $. 
Now, 
supposing $\mathcal{P}(g_2,h,2)\ne \emptyset $, 
there exists $P\in k[g_2][T]$ such that 
$m_{\w }^{g_2,h}(P(h))=2\in \{ 1,\ldots ,p-1\} $
and $\deg P(h)\le \deg h=3\delta $. 
Then, 
it follows from 
Lemma~\ref{lem:deg P'(g)} (i) 
and Remark~\ref{rem:independence} (iii) 
that 
$$
\deg P^{(1)}(h)\ge \eta (2,3,2)\delta =4\delta 
>\deg dg_1\wedge dg_2+\deg P(h)
\ge \deg dg_1\wedge dg_2\wedge dP(h). 
$$ 
Since $dg_1\wedge dg_2\wedge dP(h)
=P^{(1)}(h)dg_1\wedge dg_2\wedge dh=P^{(1)}(h)dG$ 
by Remark~\ref{rem:chain rule}~(ii), 
and $dG\ne 0$ by assumption, 
we are led to a contradiction. 
\end{proof}

\begin{proof}[Proof of Theorem~{\rm \ref{thm:modification}}]
(i) 
We fix $m\in \{ 1,2\} $. 
Since $\# k\ge p\ge 3$, 
it suffices to verify that 
$C_m:=\{ 
\alpha \in k\mid 
\mathcal{P}(f_2,f_1+\alpha f_3,m)\ne \emptyset \} $ 
has at most one element. 
Suppose that there exist 
different $\alpha ,\beta \in C_m$. 
Then, 
the assumptions of Lemma~\ref{prop:modification2} (i) or (ii) 
hold for 
$G:=(f_1+\alpha f_3,f_2,(\beta -\alpha )f_3)$, 
since $dG\approx dF\ne 0$, 
$g_1^{\w }=f_1^{\w }$, and $g_3^{\w }\approx f_3^{\w }$. 
This yields 
$\mathcal{P}(f_2,f_1+\beta f_3,m)=\mathcal{P}(g_2,g_1+g_3,m)=\emptyset $, 
a contradiction.

(ii) 
For $c\in k^*$, set $g:=f_1+cf_3$. 
Since $\deg f_3<\deg f_1$, 
we have 
$\deg g=\deg f_1=3\delta $ and $g^{\w }=f_1^{\w }$ ($1^*$). 
Now, 
take $\phi \in \mathcal{P}(f_2,g,7)$. 
Then, 
Lemma~\ref{lem:nonsingular}~(c) 
implies that $\phi \approx (g^2-\kappa f_2^3)^7+P(g)$ 
for some $\kappa \in k^*$ 
and $P\in k[f_2][T]$ with 
$\deg ^{f_2,g}P(g)<42\delta $ ($2^*$). 
By scaling, 
we can replace $\approx $ with $=$. 
Then, 
Remark~\ref{rem:chain rule}~(ii) gives 
$df_2\wedge d\phi =df_2\wedge dP(g)=P^{(1)}(g)df_2\wedge dg$. 
Hence, it suffices to verify that~$\deg P^{(1)}(g)>3\delta $.

Suppose the contrary, 
that $\deg P^{(1)}(g)\le 3\delta $. 
Then, 
$P^{(1)}(g)\in \mathcal{P}(f_2,g,m')$ holds for some $m'\ge 0$. 
We first show that $\deg ^{f_2,g}P^{(1)}(g)\le 12\delta $ ($3^*$). 
If $m'\ge 1$, then it follows from 
Theorem~\ref{ex:ineq1} (ii) and 
Lemma~\ref{lem:nonsingular} 
(a), (b), and (c) that 
$m'\in \{ 1,2,7\} $ and $\deg ^{f_2,g}P^{(1)}(g)=6m'\delta $. 
Since $\deg ^{f_2,g}P^{(1)}(g)\le \deg ^{f_2,g}(g)<42\delta $ 
by ($2^*$), 
we get $m'\in \{ 1,2\} $, 
and hence $\deg ^{f_2,g}P^{(1)}(g)\le 12\delta $. 
If $m'=0$, then (\ref{eq:1.3.7}) and the supposition give 
$\deg ^{f_2,g}P^{(1)}(g)=\deg P^{(1)}(g)\le 3\delta $, 
proving ($3^*$).

By ($2^*$), 
we can write $P(g)=\sum _{i=0}^{13}p_ig^i$, 
where $p_i\in k[f_2]$. 
Set $\pi :=\sum _{i\ne 0,7}p_ig^i$. 
Then, 
by (\ref{eq:degw_gP}), we have 
$\deg \pi \le \deg _{\w _g}P^{(1)}+\deg g$. 
This combined with 
Lemma~\ref{lem:PQ} (i), ($3^*$), and ($1^*$) yields 
$\deg \pi \le 15\delta $. 
Since $\deg \phi \le 3\delta $ by the choice of $\phi $, 
we obtain $\deg d(\pi -\phi )\le \deg (\pi -\phi )\le 15\delta $ 
($4^*$).

By assumption, 
we can 
also 
take $\phi _0\in \mathcal{P}(f_2,f_1,m)$. 
Since $p=7$ and 
$\deg df_1\wedge df_2>(1/2)\delta $, 
we must have 
$m=1$ by Remark~\ref{rem:ep<1/7}. 
Hence, 
Lemma~\ref{lem:nonsingular}~(a) yields 
$\phi _0\approx f_1^2-\kappa f_2^3+q_1f_1+q_0$, 
where $q_1\in kf_2+k$, $q_0\in kf_2^2+kf_2+k$, 
and the same $\kappa \in k^*$ as for $\phi $ due to ($1^*$). 
By scaling, we can replace $\approx $ with $=$. 
Now, set 
\begin{equation}\label{eq:pf:modifi77}
\phi  _1:=g^2-\kappa f_2^3+q_1g+q_0
=2cf_1f_3+c(cf_3+q_1)f_3+\phi _0. 
\end{equation}
Then, 
since 
$2c\ne 0$, 
$\deg (cf_3+q_1)<\deg f_1$, 
and $\deg \phi _0\le \deg f_1$, 
we have 
$\phi _1^{\w }\approx (f_1f_3)^{\w }$ ($5^*$) 
and 
$\deg \phi  _1=\deg f_1f_3>\deg f_1=3\delta $ ($6^*$).

For $i=0,1$, 
take $\lambda _{i}(x)\in k[x]$ with 
$\lambda _{i}(f_2)=q_i^7-p_{7i}$. 
Then, it follows 
from (\ref{eq:pf:modifi77}) and the equalities 
$\phi =(g^2-\kappa f_2^3)^7+P(g)$ and $P(g)=\pi +p_7g^7+p_0$ 
that 
\begin{equation}\label{eq:modify h^7}
\phi  _1^7+\pi -\phi =q_1^7g^7+q_0^7-(p_7g^7+p_0)
=\lambda _1(f_2)g^7+\lambda _0(f_2). 
\end{equation}
Differentiating (\ref{eq:modify h^7}) yields 
$d(\pi -\phi )
=(\lambda _1^{(1)}(f_2)g^7+\lambda _0^{(1)}(f_2))df_2$. 
Since 
$\deg g^7=21\delta $ by ($1^*$), 
this equality 
together with ($4^*$) and Remark~\ref{rem:parity} 
implies that 
$\lambda _1^{(1)}(f_2)=0$ ($7^*$) 
and $\deg \lambda _0^{(1)}(f_2)<15\delta $ ($8^*$).

From 
(\ref{eq:modify h^7}), ($6^*$), ($4^*$), 
and Remark~\ref{rem:parity}, 
we have 
either 
$(\phi  _1^7)^{\w }=(\lambda _1(f_2)g^7)^{\w }$ ($9^*$) or 
$(\phi  _1^7)^{\w }=\lambda _0(f_2)^{\w }$ ($10^*$). 
Case ($10^*$) implies that 
$\deg \lambda _0(f_2)=\deg \phi  _1^7>\deg \lambda _0^{(1)}(f_2)f_2$ 
by ($6^*$) and ($8^*$). 
This shows that $\lambda _0(x)$ is of degree $7l$ for some $l\ge 1$, 
hence 
yielding $(\phi  _1^7)^{\w }=\lambda _0(f_2)^{\w }\approx (f_2^{\w })^{7l}$, 
i.e., 
$(\phi _1^{\w }/(f_2^{\w })^{l})^7\in k^*$. 
Since $\phi _1^{\w }/(f_2^{\w })^{l}\in \kxr $, 
it follows that $\phi _1^{\w }/(f_2^{\w })^{l}\in k^*$. 
By ($5^*$), 
this contradicts $(f_1f_3)^{\w }\not\in k[f_2^{\w }]$. 
Case ($9^*$) implies that $(f_3^{\w })^7\approx \lambda _1(f_2)^{\w }$ 
by ($5^*$) and ($1^*$). 
This together with ($7^*$) leads to the contradiction 
$f_3^{\w }\in k[f_2^{\w }]$, 
similarly to Case ($10^*$). 
\end{proof}

\subsection{Further applications}

In this subsection, 
we discuss a different type of application of Theorem~\ref{thm:main}.

\begin{thm}\label{cor:h-q}
Assume {\rm (S)} and $p\nmid a$, 
and let $h\in \kx $ be such that 
$h^{\w }=\phi ^{\w }$ and $df\wedge dg\wedge dh\ne 0$. 
Set $N_0:=(a-1)\deg g-\deg df\wedge dh$.

\nd{\rm (i)} 
We have $\deg (h-\phi )>\min \{ N_0,pN_1\} $, 
where $N_1:=N_0-(p-1)^{-1}\deg fgh$.

\nd{\rm (ii)} 
If $m=p$, 
then we have $\deg (h-\phi )>pN_0-\deg fgh$. 

\end{thm}

\begin{proof}
Let $P$ be as in (\ref{eq:description of q}). 
Then, it follows from $h^{\w }=\phi ^{\w }$, 
(S), and Lemma~\ref{lem:PQ} (i) that 
$\deg h=\deg \phi <\deg ^{f,g}\phi =\deg _{\w _g}P$. 
Hence, 
setting $\hat{P}:=h-P\in k[f,h][T]$, 
we obtain $\hat{P}^{\w _g}=-P^{\w _g}$. 
This implies that $m_{\w }^g(\hat{P})=m_{\w }^g(P)$ ($1^*$). 
We also have 
$\deg _{\w _g}\hat{P}=\deg _{\w _g}P=\deg ^{f,g}\phi \ge mab\delta $ ($2^*$) 
by Lemma~\ref{lem:lcm} (ii). 
Since $p\nmid a$ by assumption, 
Lemma~\ref{lem:lcm} (i) gives 
$m_{\w }^g(P)=m$ ($3^*$).

Now, set $\gamma :=(m+v_p(m!))\deg df\wedge dg\wedge dh >0$. 
Then, 
applying Theorem~\ref{thm:main} with $\ff :=(f,h)$ yields 
\begin{equation*}\label{eq:thm:second conseq}
\begin{aligned}
&\deg (h-\phi )=\deg \hat{P}(g)
\ge \deg _{\w _g}\hat{P}-m\Lambda _0-v_p(m!)\Lambda _1 
&& \text{(by ($1^*$), ($3^*$))} \\ 
&\quad 
\ge mab\delta -m(\deg df\wedge dh+\deg g)-v_p(m!)\deg fgh +\gamma 
&& \text{(by ($2^*$))} \\ 
&\quad >mN_0-v_p(m!)\deg fgh =:K 
&&(\text{by }\gamma >0). 
\end{aligned}
\end{equation*}
This implies (ii) since $v_p(p!)=1$. 
For (i), 
since $h-\phi =\hat{P}(g)\ne 0$, 
we may assume that $N_0,N_1\ge 0$. 
Then, 
$K=mN_0\ge N_0$ holds for $1\le m<p$. 
Since $v_p(m!)<\frac{m}{p-1}$ by (\ref{eq:Legendre}), 
$K>mN_1\ge pN_1$ holds for $m\ge p$, 
proving (i). 
\end{proof}

\begin{thm}\label{prop:Exc27,28}
Assume that $p\ge 7$. 
Let $F\in \kx ^3$ be such that $dF\ne 0$, 
and set $\gamma :=(1/2)\deg f_2$. 
If $(s-2)\gamma +\deg df_1\wedge df_2\le \deg f_3<\deg f_1=s\gamma $ 
for some odd $s\ge 3$, 
then the following assertions hold.

\nd{\rm (i)} 
If $f_2^{\w }\not\in k[f_3^{\w }]$, $f_3^{\w }\not\in k[f_2^{\w }]$, 
and there exists $\phi \in \mathcal{D}(f_2,f_3)$ 
such that $\phi ^{\w }=f_1^{\w }$, 
then we have 
$(7/3)\gamma <\deg (f_1-\phi )<3\gamma $ 
and $\deg f_3=(4/3)\gamma $.

\nd{\rm (ii)} 
If there exists $\phi \in \mathcal{D}(f_1,f_3)$ 
such that $\phi ^{\w }=f_2^{\w }$, 
then we have $f_1^{\w }\approx (f_3^{\w })^2$. 
\end{thm}
\begin{proof}
By assumption, we have 
(1$^*$) $\gamma \le (s-2)\gamma <\deg f_3<s\gamma $, 
(2$^*$) $\deg f_1=s\gamma$ and $\deg f_2=2\gamma $, 
%$(\deg f_1,\deg f_2)=(s\gamma ,2\gamma )$, 
and (3$^*$) $\deg df_1\wedge df_2\le \deg f_3-(s-2)\gamma $.

(i) 
The assumptions of Theorem~\ref{cor:h-q}, 
except for $p\nmid a$, 
hold with $(f,g,h):=(f_2,f_3,f_1)$. 
We first show that $3\le a\le 6$, 
which implies $p\nmid a$ since $p\ge 7$. 
Since $\deg f_3=(b/a)\deg f_2=(2b/a)\gamma $, 
(1$^*$) implies that $1\le s-2<2b/a<s$ (4$^*$). 
From $f_2^{\w }\not\in k[f_3^{\w }]$ and $f_3^{\w }\not\in k[f_2^{\w }]$, 
we have $a,b\ge 2$. 
Now, if $a=2$, 
then~(4$^*$) yields $b=s-1$, 
contradicting $\gcd (a,b)=1$. 
If $a\ge 7$, then (4$^*$) yields $(a,b)\in I(7,3)$. 
Hence, 
by Lemma~\ref{lem:increasing function} (ii), 
$\zeta _i(a,b,a+b)>0$ holds for $i=0,1$. 
However, 
the assumptions $\phi ^{\w }=f_1^{\w }$, 
(2$^*$), and (1$^*$) imply that $\deg \phi =\deg f_1
=2\gamma +(s-2)\gamma <\deg f_2+\deg f_3=(a+b)\delta $. 
This contradicts Corollary~\ref{cor:pSU0}~(ii).

Now, 
applying Theorem~\ref{cor:h-q} (i) with $(f,g,h):=(f_2,f_3,f_1)$, 
we obtain 
\begin{gather}
s\gamma =\deg f_1>\deg (f_1-\phi )>\min \{ N_0,pN_1\} , 
\text{ where } \label{eq:exc:pf:alpha beta0} \\
N_0=(a-1)\deg f_3-\deg df_1\wedge df_2
\ge (a-2)\deg f_3+(s-2)\gamma \label{eq:exc:pf:alpha beta} 
\quad (\text{by }(3^*)). 
\end{gather}
As for $N_1$, 
note that $\deg f_1f_2f_3=\deg f_3+(s+2)\gamma $ by (2$^*$). 
Hence, 
(\ref{eq:exc:pf:alpha beta}) implies that 
$N_1\ge (a-2-\frac{1}{p-1})\deg f_3+(s-2-\frac{s+2}{p-1})\gamma $. 
Since 
$a\ge 3$, 
we get $(a-2-\frac{1}{p-1})\deg f_3>(1-\frac{1}{p-1})(s-2)\gamma $ 
by (1$^*$). 
Thus, 
$N_1$ is greater than $\zeta (s,p)\gamma $, 
where 
\begin{align*}
\zeta (s,p):=
\left( 1-\dfrac{1}{p-1}\right) (s-2)
+\left(s-2-\dfrac{s+2}{p-1}\right) 
=2\left( 1-\dfrac{1}{p-1}\right) s-4. 
\end{align*}
Since $\zeta (3,7)>3/7$, and 
$\zeta (x,y)-x/y$ is an increasing function in $x$ and $y$ on $I(3,7)$, 
we have $\zeta (s,p)>s/p$, 
and so 
$pN_1>s\gamma $. 
Therefore, 
(\ref{eq:exc:pf:alpha beta0}) and (\ref{eq:exc:pf:alpha beta}) 
yield 
\begin{equation}\label{eq:Exc28}
s\gamma >\deg (f_1-\phi )>N_0\ge (a-2)\deg f_3+(s-2)\gamma  .
\end{equation}
Substituting $\deg f_3=(2b/a)\gamma $ 
into (\ref{eq:Exc28}) gives $\frac{a-2}{a}b<1$. 
Since $a\ge 3$ and $b\ge 2$ as shown above, 
it follows that $(a,b)=(3,2)$, i.e., $\deg f_3=(4/3)\gamma $. 
This implies that $s=3$ by (1$^*$). 
Then, (\ref{eq:Exc28}) yields 
$3\gamma >\deg (f_1-\phi )>(4/3)\gamma +\gamma =(7/3)\gamma $.

(ii) 
We prove $f_1^{\w }\approx (f_3^{\w })^2$ 
by applying Theorem~\ref{cor:h-q} (i) 
with $(f,g,h):=(f_1,f_3,f_2)$ and showing $(a,b)=(2,1)$. 
The assumptions of Theorem~\ref{cor:h-q} hold except for $p\nmid a$. 
We first show that $p\nmid a$. 
Since $\deg f_3=(b/a)\deg f_1=(bs/a)\gamma $, 
it follows from (1$^*$) and $s\ge 3$ that $1<bs/(a(s-2))\le 3b/a$ ($5^*$). 
Now, if $p\mid a$, 
then since $a\ge p\ge 7$, 
we have $(a,b)\in I(7,3)$ by ($5^*$). 
Hence, 
by Lemma~\ref{lem:increasing function}~(ii), 
$\zeta _i(a,b,2a/s)>0$ holds for $i=0,1$. 
This contradicts Corollary~\ref{cor:pSU0}~(ii), 
since the assumptions $\phi ^{\w }=f_2^{\w }$ and (2$^*$) 
imply that $\deg \phi =\deg f_2=(2/s)\deg f_1=(2a/s)\delta $.

Theorem~\ref{cor:h-q} (i) now yields 
$2\gamma =\deg f_2>\deg (f_2-\phi )>\min \{ N_0,pN_1\} $ (6$^*$). 
Here, 
$N_0$ and $N_1$ have the same form as in (i), 
since $g=f_3$ and $df\wedge dh=\pm df_1\wedge df_2$
in both (i) and (ii). 
Now,
supposing $(a,b)\ne (2,1)$, 
we must have $a\ge 3$, 
since $b<a$ by the assumption that $\deg f_3<\deg f_1$. 
Then, 
$pN_1>s\gamma $ holds as shown in (i), 
while (\ref{eq:exc:pf:alpha beta}), (1$^*$), and $s\ge 3$ yield 
$N_0>2\gamma $. 
This contradicts~(6$^*$). 
\end{proof}

\subsection{Degrees of differential forms}\label{sect:cofactor}

We first extend the results of \cite[\S ~1.3.5]{JC} 
(see also \cite[Lemma 3.5]{tame3}) to the case where $p>0$. 
Let $k$ be any field, 
and $H=(h_1,h_2,h_3)\in \kx ^3$ such that $dH\ne 0$. 
Then, $\deg h_i>0$ holds for $i=1,2,3$. 
Take $a,c\in k$, 
$\Psi (x)\in k[x]$, 
and $\phi \in k[h_3]$, 
and set
\begin{equation}\label{eq:k1k2!}
h_1':=h_1+ah_3^2+ch_3+\Psi (h_2)\quad\text{and}\quad 
h_2':=h_2+\phi . 
\end{equation}

\begin{lem}\label{note:cofactor}
Assume that $d\phi =0$, i.e., $dh_2'=dh_2$.

\nd {\rm (i)} 
If $2a\neq 0$ and 
$\deg h_3+\deg dh_2\wedge dh_3>\deg dh_1'\wedge dh_2'$, 
then we have 
$$
\deg dh_1\wedge dh_2=\deg h_3+\deg dh_2\wedge dh_3. 
$$

\nd {\rm (ii)} 
If $2a=0$ and $\deg dh_2\wedge dh_3>\deg dh_1'\wedge dh_2'$, 
then we have 
$$
\deg dh_1\wedge dh_2=\left\{
\begin{array}{ll}
\deg dh_2\wedge dh_3 & \text{ if }
c\ne 0 \\
\deg dh_1'\wedge dh_2' & \text{ if }c=0 .
\end{array}
\right. 
$$
\end{lem}
\begin{proof}
It follows from (\ref{eq:k1k2!}) that 
\begin{equation}\label{eq:differential form A'}
dh_1'\wedge dh_2
=dh_1\wedge dh_2+2ah_3dh_3\wedge dh_2+cdh_3\wedge dh_2. 
\end{equation}
Now, for (i), 
we have $\deg 2ah_3dh_3\wedge dh_2=\deg h_3+\deg dh_2\wedge dh_3$, 
which is greater than 
$\deg dh_1'\wedge dh_2=\deg dh_1'\wedge dh_2'$ 
and $\deg dh_3\wedge dh_2$. 
Hence, 
(\ref{eq:differential form A'}) 
yields the required equality. 
(ii) can be proved similarly, 
using (\ref{eq:differential form A'}). 
\end{proof}

The proof of the following theorem 
is independent of the characteristic of $k$ 
(see \cite[Theorem 5.2]{SUineq}; see also 
\cite[Theorem 1.3.9]{JC} and \cite[Lemma 5]{SU}).

\begin{thm}
\label{thm:another}
Let $\eta _1,\eta _2,\eta _3\in \Omega _{\kx /k}$. 
For $i=1,2,3$, 
we define 
$\delta _i:=\deg \eta _i+\deg \eta _j\wedge \eta _l$, 
where $1\le j<l\le 3$ and $j,l\ne i$. 
Then, at least two of $\delta _1$, $\delta _2$, and $\delta _3$ 
are equal to $\max \{ \delta _1,\delta _2,\delta _3\} $. 
\end{thm}

Now, 
consider the following conditions:

\smallskip 

(a) 
$\deg dh_1'+\deg dh_2\wedge dh_3
>\deg dh_3+\deg dh_1'\wedge dh_2'$; 

\smallskip

(b) 
$\deg dh_1'-\deg dh_2'>\deg \Psi ^{(1)}(h_2)$; 

\smallskip 

(c) $\deg h_3>\deg dh_1'-\deg dh_2'>0$; 

\smallskip 

(d) $\phi =bh_3+e$ for some $b,e\in k$. 

\smallskip

\nd 
Since $\phi \in k[h_3]$ by assumption, 
we have $dh_2\wedge dh_3=dh_2'\wedge dh_3$. 
Hence, 
applying Theorem~\ref{thm:another} 
with $(\eta _1,\eta _2,\eta _3):=(dh_1',dh_2',dh_3)$, 
we see that (a) implies

\smallskip

(a$'$) 
$\deg dh_1'+\deg dh_2\wedge dh_3
=\deg dh_2'+\deg dh_1'\wedge dh_3$.

\begin{lem}\label{lem:determinant}
\nd{\rm (i)} 
If {\rm (a)} and {\rm (b)} hold, 
then we have

$(1)$ 
$\deg dh_1\wedge dh_3=\deg dh_1'-\deg dh_2'+\deg dh_2\wedge dh_3;$

$(2)$ $\deg dh_1\wedge dh_3>\deg \Psi ^{(1)}(h_2)dh_2\wedge dh_3$.

\nd{\rm (ii)} 
If $\deg dh_2\wedge dh_3>\deg dh_1'\wedge dh_2'$ 
and {\rm (a)}--{\rm (d)} hold, 
then we have

$(1)$ $\deg h_3+\deg dh_2\wedge dh_3
>\deg dh_1\wedge dh_3 
>\deg dh_2\wedge dh_3>\deg dh_1'\wedge dh_2';$

$(2)$ 
$\deg dh_1\wedge dh_2=\left\{
\begin{array}{cl}
\deg h_3+\deg dh_2\wedge dh_3 & \text{ if }2a\neq 0 \\
\deg dh_1\wedge dh_3 & \text{ if }2a=0\text{ and }b\neq 0 \\
\deg dh_2\wedge dh_3 & \text{ if }2a=b=0
\text{ and }c\neq 0 \\
\deg dh_1'\wedge dh_2' & \text{ if }2a=b=c=0 .
\end{array}
\right. $

\end{lem}
\begin{proof}
(i) Since 
$dh_1'\wedge dh_3=dh_1\wedge dh_3+\Psi ^{(1)}(h_2)dh_2\wedge dh_3$ 
by (\ref{eq:k1k2!}), 
and $\deg dh_1'\wedge dh_3 
=\deg dh_1'-\deg dh_2'+\deg dh_2\wedge dh_3 
>\deg \Psi ^{(1)}(h_2)dh_2\wedge dh_3$ ($*$) 
by (a$'$) and~(b), 
we have 
$\deg dh_1'\wedge dh_3=\deg dh_1\wedge dh_3$. 
This together with ($*$) yields (1) and (2).

(ii) 
(1) follows from 
(i) (1), (c), and the assumption. 
By (\ref{eq:k1k2!}) and (d), we have
\begin{align*}
dh_1'\wedge dh_2'
=dh_1'\wedge (dh_2+bdh_3)
=dh_1'\wedge dh_2
+b(dh_1\wedge dh_3+\Psi ^{(1)}(h_2)dh_2\wedge dh_3). 
\end{align*}
Combining this equality with (\ref{eq:differential form A'}), 
(ii) (1), and (i) (2) yields (2). 
\end{proof}

Next, let $A$ be a $k$-subalgebra of $\kx $.

\begin{definition}\rm 
We define $A^{\w }:=\sum _{f\in A}kf^{\w }$. 
The $k$-vector space $A^{\w }$ is a $k$-subalgebra of $\kx $, 
since $f^{\w }g^{\w }=(fg)^{\w }\in A^{\w }$ 
holds for all $f,g\in A$. 
\end{definition}

\begin{rem}\label{rem:homogeneous}\rm 
If $f,f_1,\ldots ,f_r\in \kx \sm \zs $ satisfy 
$f^{\w }=\sum _{i=1}^rf_i^{\w }$, 
then we have 
$f^{\w }=\sum _{i\in I}f_i^{\w }=(\sum _{i\in I}f_i)^{\w }$, 
where $I:=\{ i\mid i=1,\ldots ,r,\ \deg f_i=\deg f\} $. 
\end{rem}

\begin{lem}\label{def:initial algebra}\rm 
$f\in \kx $ 
satisfies $f^{\w }\in A^{\w }$ 
if and only if $f^{\w }=\phi ^{\w }$ for some $\phi \in A$. 
\end{lem}

\begin{proof}
The ``if" part is clear, 
and the ``only if" part follows from  
Remark~\ref{rem:homogeneous}. 
\end{proof}

For all $f_1,\ldots ,f_r\in \kx $, 
we have 
$k[f_1^{\w },\ldots ,f_r^{\w }]\subset k[f_1,\ldots ,f_r]^{\w }$, 
but the $k$-algebra 
$k[f_1,\ldots ,f_r]^{\w }$ is generally difficult to analyze.

\begin{thm}\label{ex:total}
Let $F\in \kx ^3$ be such that $dF\ne 0$, 
and set $\delta :=(1/2)\deg f_2$. 
Assume that 
\begin{equation}\label{eq:ex:total}
\deg df_1\wedge df_2\le \dfrac{1}{2}\delta ,\ \ 
\deg f_1=3\delta ,\ \ 
\deg f_3=\dfrac{3}{2}\delta ,
\text{ \ and \ }a:=f_1^{\w }(f_3^{\w })^{-2}\in k^*. 
\end{equation}
Take any $f_1'\in f_1-af_3^2+kf_3+kf_2+k$. 
Then, the following assertions hold. 

\nd {\rm (i)} 
If $p\ge 3$, 
then 
$(1)$ $\deg df_1'\wedge df_2=\deg f_3+\deg df_2\wedge df_3$ 
and $(2)$ $\delta <\deg f_1'<3\delta $. 

\smallskip

\nd For {\rm (ii)--(iv)} below, 
assume further that 
$p\ge 7$ and $(f_1')^{\w }\not\approx f_i^{\w }$ for $i=2,3$.

\nd {\rm (ii)} 
We have $f_3^{\w }\not\in k[f_1',f_2]^{\w }$.

\nd {\rm (iii)} 
If $f_2^{\w }\in k[f_1',f_3]^{\w }$, 
then we have $\deg f_1'=(9/4)\delta $, 
and $\mathcal{P}(f_3,f_1',m)\ne \emptyset $ for some $m\in \{ 1,2\} $.

\nd {\rm (iv)} 
There exists $f_1''\in f_1'+kf_2$ 
such that $(f_1'')^{\w }=(f_1')^{\w }$, 
$f_2^{\w }\not\in k[f_1'',f_3]^{\w }$, 
and $f_3^{\w }\not\in k[f_1'',f_2]^{\w }$.

\end{thm}
\begin{proof}
Write $f_1'=f_1-af_3^2-cf_3-lf_2-l'$, 
where $c,l,l'\in k$. 
Then, we have 
\begin{equation}\label{eq:big ex}
df_1=df_1'+2af_3df_3+cdf_3+ldf_2
\text{ \ and \ }
df_1\wedge df_3=df_1'\wedge df_3+ldf_2\wedge df_3. 
\end{equation}

(i) 
(1) is proved by applying Lemma~\ref{note:cofactor} (i) 
with 
$H:=(f_1',f_2,f_3)$, $h_1':=f_1$, and $h_2':=f_2$, 
since $p\ge 3$, 
$a\ne 0$, 
and $\deg f_3+\deg df_2\wedge df_3>\deg f_3>\deg df_1\wedge df_2$ 
by (\ref{eq:ex:total}). 
For (2), 
set $P:=f_1-aT^2-cT-lf_2-l'\in k[f_1,f_2][T]$. 
Then, we have 
$\deg _{\w _{f_3}}P=3\delta $, $P^{\w _{f_3}}=f_1^{\w }-aT^2$, 
and $m_{\w }^{f_3}(P)=1$ 
since $f_1^{\w }=a(f_3^{\w })^2$ and $p\nmid 2$. 
Hence, 
it follows from Theorem~\ref{thm:main} that 
$\deg f_1'=\deg P(f_3)\ge 
\deg _{\w _{f_3}}P-(\deg df_1\wedge df_2+\deg f_3-\deg dF)
>(3-1/2-3/2)\delta =\delta $. 
Clearly, we have $\deg f_1'<\deg f_1=3\delta $.

(ii) 
We first claim that 
$f_3^{\w }\not\in k[(f_1')^{\w },f_2^{\w }]$, 
i.e., 
$f_3^{\w }$ is not a $k$-linear combination of 
$((f_1')^{\w })^i(f_2^{\w })^j$ 
for $i,j\in \N $ with $\deg {(f_1')^if_2^j}=\deg f_3$ 
(see Remark~\ref{rem:homogeneous}). 
In fact, 
we have 
$\deg f_3<\deg f_2=2\delta <\deg (f_1')^2$ 
by (\ref{eq:ex:total}) and~(i)~(2), 
and $f_3^{\w }\not\approx (f_1')^{\w }$ by assumption. 
Now, supposing $f_3^{\w }\in k[f_1',f_2]^{\w }$, 
there exists $\phi \in k[f_1',f_2]$ such that $\phi ^{\w }=f_3^{\w }$ 
by Lemma~\ref{def:initial algebra}. 
Then, by the claim and Remark~\ref{rem:(1.3.7)} (ii), 
we have $\phi \in \mathcal{D}(f_1',f_2)$. 
In addition, 
$(f_1')^{\w }\not\in k[f_2^{\w }]$ and 
$f_2^{\w }\not\in k[(f_1')^{\w }]$ hold, 
since $(f_1')^{\w }\not\approx f_2^{\w }$, 
$\deg f_1'<3\delta <\deg f_2^2$, 
and $\deg f_2<\deg (f_1')^2$. 
Thus, since $p\ge 7$, 
Corollary~\ref{cor:pSU0} (iii) yields 
$\deg f_3=\deg \phi >\deg df_1'\wedge df_2$. 
This contradicts (i) (1).

(iii) 
Since $f_2^{\w }\not\approx (f_1')^{\w }$ 
and $\deg f_3<\deg f_2<\min \{ \deg (f_1')^2,\deg f_3^2\} $, 
we have 
$f_2^{\w }\not\in k[(f_1')^{\w },f_3^{\w }]$ 
similarly to (ii). 
Now, 
by assumption, 
there exists 
$\phi \in k[f_1',f_3]$ such that $\phi ^{\w }=f_2^{\w }$. 
It satisfies 
$\phi \in \mathcal{D}(f_1',f_3)$ by Remark~\ref{rem:(1.3.7)}~(ii), 
and $\deg \phi =\deg f_2=2\delta $. 
We also have $\deg f_1'>\delta =(2/3)\deg f_3$, 
and $(f_1')^{\w }\not\in k[f_3^{\w }]$ 
since $(f_1')^{\w }\not\approx f_3^{\w }$ and $\deg f_1'<\deg f_3^2$. 
Thus, 
either (a) or (b) in Theorem~\ref{prop:Exc29} 
holds with $(f,g,\gamma ):=(f_1',f_3,\delta )$. 
We claim that (b) contradicts Theorem~\ref{thm:another} 
with $(\eta _1,\eta _2,\eta _3):=(df_1',df_2,df_3)$. 
In fact, 
(b) implies that 
$\deg df_1'+\deg df_2\wedge df_3
\le (9/8)\delta +\deg df_2\wedge df_3=:\delta _1'$ 
and 
$\deg df_2+\deg df_1'\wedge df_3
\le \deg f_2+(1/8)\delta -\ep (f_3)
=(2+1/8-3/2)\delta +\deg df_3
=(5/8)\delta +\deg df_3=:\delta _2'$. 
However, since $\deg f_3=(3/2)\delta $, 
(i) (1) yields 
$\deg df_3+\deg df_1'\wedge df_2
=\deg df_3+(3/2)\delta +\deg df_2\wedge df_3
>\max \{ \delta _1',\delta _2'\} $. 
Therefore, 
(a) must hold, proving 
$\deg f_1'=(9/4)\delta $. 
Setting $\delta ':=(3/4)\delta $, we have 
\begin{equation}\label{eq:totalExample}
\deg f_3=2\delta ',\ \ \deg f_1'=3\delta ',\ \ 
\deg \phi =\deg f_2=\frac{8}{3}\delta ',
\text{ \ and \ }
\phi \in \mathcal{P}(f_3,f_1',m) 
\end{equation}
with $m\ge 1$. 
It remains to prove that $m\le 2$. 
If not, 
then since $p\ge 7$, 
it follows from 
(\ref{eq:totalExample}) and Theorem~\ref{ex:ineq1}~(ii) that $p=m=7$. 
Hence, 
we have 
$\gamma _2':=\deg df_1'\wedge df_3\le (1/8)\delta '$ 
and $\ep:=\ep (f_1',f_3)<(1/7)\delta '$ 
by Remark~\ref{rem:1/7} (i). 
Thus, 
(\ref{eq:totalExample}) gives 
\begin{equation}\label{eq:totalExample2}
\deg f_3df_3\ge 4\delta '-\ep >3\delta '\ge 
\max \{\deg df_1',\deg df_2+\gamma _2'\} >\deg df_2. 
\end{equation}
The first part of (\ref{eq:big ex}) and (\ref{eq:totalExample2}) 
yield $\deg df_1=\deg f_3df_3$. 
Since 
$\deg f_3>\deg df_1\wedge df_2$ by (\ref{eq:ex:total}), 
this gives 
$\alpha :=\deg df_1+\deg df_2\wedge df_3
>\deg f_3df_3
>\deg df_3+\deg df_1\wedge df_2$. 
The last part of (\ref{eq:big ex}) 
implies that 
$\gamma _2'+\deg df_2\wedge df_3>\deg df_1\wedge df_3$. 
Since $\deg df_1=\deg f_3df_3>\deg df_2+\gamma _2'$ 
by (\ref{eq:totalExample2}), 
it follows that 
$\alpha >\deg df_2+\gamma _2'+\deg df_2\wedge df_3
>\deg df_2+\deg df_1\wedge df_3$. 
This contradicts Theorem~\ref{thm:another} as before.

(iv) 
Since 
$f_3^{\w }\not\in k[f_1',f_2]^{\w }=k[f_1'+cf_2,f_2]^{\w }$ 
holds for all $c\in k$ by (ii), 
it suffices to find $f_1''\in f_1'+kf_2$ 
satisfying the first two conditions. 
If $f_2^{\w }\not\in k[f_1',f_3]^{\w }$, 
then taking $f_1'':=f_1'$ suffices. 
Assume that $f_2^{\w }\in k[f_1',f_3]^{\w }$. 
Then, we have $\deg f_1'=(9/4)\delta $ by (iii). 
This implies that $(\sharp )$, 
and hence $(\flat )$, in \S~\ref{sect:modifi} 
holds for $(f_1',f_3,f_2)$ by (\ref{eq:totalExample}). 
Hence, 
by Theorem~\ref{thm:modification}~(i), 
there exists $f_1''\in f_1'+kf_2$ 
such that 
$\mathcal{P}(f_3,f_1'',1)=\mathcal{P}(f_3,f_1'',2)=\emptyset $. 
Since $\deg f_1'=(9/4)\delta >\deg f_2$, 
we have $(f_1'')^{\w }=(f_1')^{\w }$, 
and so $(f_1'')^{\w }\not\approx f_i^{\w }$ for $i=2,3$. 
Hence, 
(iii) still holds with $f_1'$ replaced by $f_1''$. 
Then, we obtain $f_2^{\w }\not\in k[f_1'',f_3]^{\w }$, 
for otherwise (iii) yields 
$\mathcal{P}(f_3,f_1'',m)\ne \emptyset $ 
for some $m\in \{ 1,2\} $, 
contradicting the choice of $f_1''$. 
\end{proof}

\part{Theory of Shestakov-Umirbaev reduction}

The goal of Part 2 is to establish Theorem~\ref{thm:mainSU}. 
For readability, 
we first prove Theorem~\ref{thm:mainSU} in \S~\ref{sect:proof}, 
assuming three key lemmas that will be 
proved in \S ~\ref{sect:WSU}. 
The proofs of the key lemmas 
involve novel concepts and techniques 
not required for~$p=0$.

\section{Outline of the proof of Theorem~$\ref{thm:mainSU}$}\label{sect:proof}
\setcounter{equation}{0}

\subsection{Notation and basic facts}

We define 
$\E :=\bigcup _{l=1}^n\E _l$, 
where $\E _l$ is the set of all elementary 
automorphisms $E\in \Aut _k\kx $ such that $E(x_i)=x_i$ for $i\ne l$. 
For $F,F'\in \Aut _k\kx $ and $l\in \{ 1,\ldots ,n\} $, 
we write $F\sim _lF'$ 
if there exists $E\in \E _l$ such that $F'=FE$, 
which is equivalent to $F'\sim _lF$ since $E^{-1}\in \E _l$. 
We write $F\sim F'$ if $F\sim _lF'$ 
for some $l\in \{ 1,\ldots ,n\} $. 
Set $S_l:=\{ f_1,\ldots ,f_n\} \sm \{ f_l\} $ 
for $l=1,\ldots ,n$.

\begin{rem}\label{rem:sim}\rm 
(i) 
Let $F,F'\in \Aut _k\kx $ and $l\in \{ 1,\ldots ,n\} $. 
Then, 
$F\sim _lF'$ holds if and only if 
$F'=(f_1,\ldots ,f_{l-1},f_l-\phi ,f_{l+1},\ldots ,f_n)$ 
for some $\phi \in k[S_l]$. 

\nd (ii) 
If $\deg F>\deg F'$ holds in (i), 
then we have $f_l^{\w }=\phi ^{\w }\in k[S_l]^{\w }$. 
\end{rem}

We define 
$\D _n(k):=\{ (c_1x_1,\ldots ,c_nx_n)\in \Aut _k\kx 
\mid c_1,\ldots ,c_n\in k^*\} $, 
and ${\rm E}_n(k)$ to be the subgroup of $\T _n(k)$ generated by $\E $.

We need the following facts and observations.

\begin{rem}\label{rem:ch2}\rm

\nd (i) 
We have $\D_n(k){\rm E}_n(k)=\T _n(k)$ 
(see \cite[\S 1.1.1, Exercise 2]{JC}).

\nd (ii) ${\rm E}_n(k)$ is equal to the set of all finite compositions of 
elements of $\E $, 
since $E^{-1}\in \E $ for all $E\in \E $.

\nd (iii) 
$\{ \sum _{l=1}^ni_lw_l \mid i_1,\ldots ,i_n\in \N \} \subset \Gamma $ 
is a well-ordered set 
(see \cite[Lemma 1.1.12]{JC}).

\nd (iv) 
Let $A$ be a $k$-subalgebra of $\kx $ 
and $h\in \kx \sm A$. 
Then, by (iii), 
there exists $f\in h+A\subset \kx \sm \zs $ with the least $\w $-degree. 
It satisfies $f^{\w }\not\in A^{\w }$, 
for otherwise Lemma~\ref{def:initial algebra} gives 
$\phi \in A$ such that $f^{\w }=\phi ^{\w }$, 
i.e., 
$\deg (f-\phi )<\deg f$.

\nd (v) 
Let $F\in \Aut _k\kx $. 
If $f_1^{\w },\ldots ,f_n^{\w }$ are algebraically dependent over $k$, 
then we have $\deg F>|\w |$ 
by (\ref{eq:deg F}) and Remark~\ref{rem:independence} (iv).  
\end{rem}

In the rest of the paper, 
we assume that $n=3$ and $p\ge 7$ unless otherwise stated. 
For $l=1,2,3$ and $\phi \in k[S_l]$, 
we write 
$\degw ^{S_l}\phi :=\degw ^{f_i,f_j}\phi $ and 
$m_{\w }^{S_l}(\phi ):=m_{\w }^{f_i,f_j}(\phi )$, 
where $1\le i<j\le 3$ and $i,j\ne l$.

\subsection{Key lemmas}\label{subsect:fundamental lemmas}
Let $(F,G)\in (\Aut _k\kx )^2$. 
Observe that (SU1)--(SU3) 
imply the following conditions (SU$1'$)--(SU$3'$), 
respectively:

\nd 
\begin{enumerate}[leftmargin=12mm]

\item [(SU$1'$)] 
$g_1\in f_1+k[f_2,f_3]$, $g_2\in f_2+k[f_3]$, 
and $g_3\in f_3+k[g_1,g_2]$; 

\item [(SU$2'$)] $\deg f_i\leq \deg g_i$ for $i=1,2$; 

\item [(SU$3'$)] $\deg g_2<\deg g_1$ 
and $g_1^{\w }\not\in k[g_2^{\w }]$. 
\end{enumerate}

\begin{definition}[{\cite{tame3}}]\label{def:wsup}\rm 
(i) 
We call $(F,G)$ 
a {\it weak Shestakov-Umirbaev pair} ({\it weak SU pair}) if 
(SU$1'$)--(SU$3'$) and (SU4)--(SU6) hold.

\nd (ii) We say that 
a weak SU pair $(F,G)$ is {\it proper} if 
$(g_1-f_1,g_2-f_2)\not\in k[f_2]\times k$. 
\end{definition}

By definition, every SU pair is a weak SU pair. 
We prove the following three key lemmas in \S ~\ref{sect:WSU}.

\begin{llem}\label{prop:equivalence}

Let $(F,G)\in (\Aut _k\kx )^2$ be a weak SU pair.

\nd{\rm (i)} 
There exist 
$E_1\in \E _1$ and $E_2\in \E _2$ 
such that 
$(F,GE_1E_2)$ is an SU pair and $\deg G=\deg GE_1$.

\nd{\rm (ii)} We have $\deg F>\deg G$. 
\end{llem}

\begin{llem}\label{prop:structure2}
Assume that $F\in \Aut _k\kx $ satisfies 
the following conditions 
$\llparenthesis	1 \rrparenthesis$--$\llparenthesis 4 \rrparenthesis$$:$ 

\begin{enumerate}[leftmargin=10mm]

\item [$\llparenthesis	1 \rrparenthesis$] 
$(f_1^{\w })^2\approx (f_2^{\w })^s$ for some odd $s\geq 3;$

\item [$\llparenthesis	2 \rrparenthesis$] 
$(s-2)\delta +\deg df_1\wedge df_2\leq \deg f_3<s\delta $ 
for $s$ in $\llparenthesis 1 \rrparenthesis$ 
and $\delta :=(1/2)\deg f_2;$

\item [$\llparenthesis	3 \rrparenthesis$] 
$\mathcal{P}(f_2,f_1,m)\ne\emptyset $ 
for some $m\ge 1;$

\item [$\llparenthesis	4 \rrparenthesis$] 
$f_3^{\w }\not\in k[f_2^{\w }]$. 

\end{enumerate}

\nd 
Then, the following assertions hold$:$

\nd{\rm (i)} If $(F_{\sigma },G)$ is an SU pair 
for some $\sigma \in \sym _3$ and $G\in \Aut _k\kx $, 
then we have $\sigma =\id $ and $(f_1,f_2)=(g_1,g_2)$.

\nd{\rm (ii)} 
If $f_2^{\w }\in k[S_2]^{\w }$, 
then we have $f_1^{\w }\approx (f_3^{\w })^2$.

\nd{\rm (iii)} 
Assume that $f_1^{\w }\in k[S_1]^{\w }$. 
Then, 

\nd {\rm (1$^\star $)} 
$f_2^{\w }$ and $f_3^{\w }$ are algebraically dependent over $k$.

\nd 
Moreover, 
there exists $f_1'\in f_1+k[S_1]$ satisfying 
the following conditions$:$

\nd {\rm (2$^\star $)} 
$(f_1',f_2,f_3)$ admits no elementary reduction$;$ 

\nd {\rm (3$^\star $)} 
If $((f_1',f_2,f_3)_\sigma ,G)$ 
is a weak SU pair 
for some $\sigma \in \sym _3$ and $G\in \Aut _k\kx $, 
then we have $\sigma =\id $, 
and $(F,G)$ is a weak SU pair. 

\end{llem}

\begin{llem}\label{prop:c1-c6}
Let $(F,G)\in (\Aut _k\kx )^2$ be a weak SU pair, 
and $E\in \E $ such that $\deg FE\leq \deg F$.

\nd{\rm (i)} 
If $E\in \E _1\cup \E _2$, 
then $(FE,G)$ is a weak SU pair.

\nd{\rm (ii)} 
Assume that $(F,G)$ is proper and $E\in \E _3$.

\nd $(1)$ If $\deg f_3\ne \deg f_j$ for $j=1,2$, 
then $(FE,G)$ is a weak SU pair.

\nd $(2)$ If $\deg f_3=\deg f_1$ or $\deg f_3=\deg f_2$, 
then there exist $\tau \in \sym _3$ and $H\in {\rm D}_3(k)$ 
such that $((FE)_\tau ,GH)$ is a weak SU pair. 

\end{llem}

Lemmas~\ref{prop:equivalence}, \ref{prop:structure2}, 
and \ref{prop:c1-c6} 
are the $p\ge 7$ counterparts to 
Lemmas 1.2.2 and 1.2.11 in \cite{JC} 
and Proposition 4.4 in \cite{tame3}, respectively, 
with Lemma~\ref{prop:structure2} requiring an essential modification. 
In the rest of \S ~\ref{sect:proof}, 
we prove Theorem~\ref{thm:mainSU} assuming these lemmas, 
which follows the line of argument for $p=0$ 
(see \cite[\S~1.2]{JC}).

\subsection{Reduction sequence}\label{subsect:AA}
For $\sigma \in \sym _3$, 
we define 
$X_\sigma :=(x_{\sigma (1)},x_{\sigma (2)},x_{\sigma (3)})$.

\begin{rem}\label{rem:Xsigma}\rm 
Let $F,G\in \Aut _k\kx $, 
$\sigma \in \sym _3$, $H\in \D _3(k)$, 
and $l\in \{ 1,2,3\} $. 

\nd (i) We have $FX_\sigma =F_\sigma $.

\nd (ii) 
We have 
$X_{\sigma }^{-1}\E _lX_{\sigma }=\E _{\sigma ^{-1}(l)}$, 
$H^{-1}\E _lH=\E _l$, 
and $X_{\sigma }^{-1}\D _3(k)X_{\sigma }=\D_3(k)$.

\nd (iii) If $F\rightsquigarrow _lG$, 
i.e., 
$G=FE$ for some $E\in \E _l$, 
then since 
$G_\sigma =F_\sigma (X_\sigma ^{-1}EX_\sigma )$ and $GH=FH(H^{-1}EH)$, 
we have 
$F_\sigma \rightsquigarrow _{\sigma ^{-1}(l)}G_\sigma $ 
and $FH\rightsquigarrow _lGH$ 
by (ii).

\nd (iv) 
If $(F_\sigma ,G_\sigma )$ is a weak SU pair, 
then so is $((FH)_\sigma ,(GH)_\sigma )
=(F_\sigma H',G_\sigma H')$, 
where $H':=X_\sigma ^{-1}HX_\sigma \in \D_3(k)$ by (ii). 
\end{rem}

Now, let $F\in \Aut _k\kx $. 
For $l\ge 1$, 
we call $(F_i)_{i=1}^l\in (\Aut _k\kx )^l$ 
a {\it reduction sequence} of $F$ 
if $F_1=F$, $\deg F_l=|\w |$, 
and for $i=2,\ldots ,l$, 
one of the following~holds:

1. $F_{i-1}\rightsquigarrow F_i$ and $\deg F_{i-1}>\deg F_i$;

2. There exists $\sigma _i\in \sym _3$ such that 
$((F_{i-1})_{\sigma _i},(F_i)_{\sigma _i})$ is a weak SU pair.

\nd 
We define 
$\mathcal{A}:=\{ F\in \Aut _k\kx \mid 
\text{$F$ has a reduction sequence}\} . $

\begin{rem}\label{rem:A}\rm 
\nd (i) 
We have 
${\rm D}_3(k)\subset \{ F\in \Aut _k\kx \mid \deg F=|\w |\} 
\subset \mathcal{A}$.

\nd (ii) 
If $F\in \mathcal{A}$, 
then $F_\sigma ,FH\in \mathcal{A}$ hold for all 
$\sigma \in \sym _3$ and $H\in {\rm D}_3(k)$. 
Indeed, 
if $(F_i)_{i=1}^l$ is a reduction sequence of $F$, 
then 
$((F_i)_\sigma )_{i=1}^l$ and 
$(F_iH)_{i=1}^l$ are reduction sequences of $F_\sigma $ and $FH$, 
respectively, 
by Remark~\ref{rem:Xsigma} (iii) and (iv).

\nd (iii) 
If $F\in \Aut _k\kx $ satisfies 
one of the following conditions, 
then we have $F\in \A $:

(A1) 
$F\sim G$ and $\deg F>\deg G$ hold for some $G\in \A $;

(A2) 
$(F_{\sigma },G)$ is a weak SU pair 
for some $\sigma \in \sym _3$ and $G\in \A $.

\nd 
Indeed, 
a reduction sequence of $G$ 
extends to one for $F$ in Case (A1), 
and to one for $F_\sigma $ in Case (A2), 
implying $F\in \A $ by (ii).

\nd (iv) 
Let $F\in \A $ be such that $\deg F>|\w |$. 
Then, $F$ satisfies (A1) or (A2), 
since $F$ has a reduction sequence 
$(F_i)_{i=1}^l$ with $l\ge 2$, 
and $(F_2)_\sigma \in \A $ holds for all $\sigma \in \sym _3$.

\nd (v) 
(A1) implies that $F$ admits an elementary reduction. 
Moreover, 
(A2) implies that $F$ admits an SU reduction 
by virtue of Lemma~\ref{prop:equivalence}~(i). 
\end{rem}

\begin{example}\label{exc:AAA}\rm 
Let $(F,E)\in \A \times \E $. 
If $\deg FE>\deg F$, 
then since $FE\rightsquigarrow F$, 
we have $FE\in \A $ by Remark~\ref{rem:A} (iii) (A1) 
applied with $(F,G):=(FE,F)$. 
\end{example}

By Remark~\ref{rem:A} (iv) and (v), 
to establish Theorem~\ref{thm:mainSU}, 
it suffices to prove that $\T _3(k)\subset \A $. 
The following proposition is essential.

\begin{prop}\label{prop:key}
Let $(F,E)\in \A \times \E $. 
If $\deg FE\le \deg F$, then we have~$FE\in \A $. 
\end{prop}

Example~\ref{exc:AAA} and Proposition~\ref{prop:key} 
together 
imply that $\A \E \subset \A $, 
which yields $\A {\rm E}_3(k)\subset \A $ by Remark \ref{rem:ch2}~(ii). 
Since ${\rm D}_3(k)\subset \A $ by Remark~\ref{rem:A} (i), 
it follows from Remark~\ref{rem:ch2} (i) 
that $\T_3(k)={\rm D}_3(k){\rm E}_3(k)\subset \A {\rm E}_3(k)\subset \A $. 
Thus, 
we are reduced to proving Proposition~\ref{prop:key}. 
To this end, 
we also need the following proposition.

\begin{prop}\label{prop:key2}
If $F\in \mathcal{A}$ satisfies 
$\llparenthesis 1 \rrparenthesis$--$\llparenthesis 4 \rrparenthesis$ 
in Lemma~{\rm \ref{prop:structure2}}, 
then we have 
$$
3\in I_F:=\{ l\in \{ 1,2,3\} \mid 
F\sim _lG\text{ and }\deg F>\deg G\text{ hold for some }G\in \A \} . 
$$
\end{prop}

\begin{rem}\label{rem:prop key2}\rm 
Let $F\in \Aut _k\kx $. 

\nd 
(i) $F$ satisfies (A1) if and only if $I_F\ne \emptyset $.

\nd (ii) 
By Remarks~\ref{rem:Xsigma} (iii) and \ref{rem:A} (ii), 
$I_{F_\sigma }=\{ \sigma ^{-1}(l)\mid l\in I_F\} $ 
holds for all $\sigma \in \sym _3$.

\nd (iii) 
By Remark~\ref{rem:sim}, 
$l\in I_F$ implies that 
$f_l^{\w }=\phi ^{\w }\in k[S_l]^{\w }$ 
for some $\phi \in k[S_l]$.

\nd (iv) 
If $f_3^{\w }\not\in k[f_1,f_2]^{\w }$, 
then we have 
$\llparenthesis 4 \rrparenthesis$, 
and $3\not\in I_F$ by (iii). 
Hence, Proposition~\ref{prop:key2} implies that, 
if $F\in \A $ satisfies 
$\llparenthesis 1 \rrparenthesis$--$\llparenthesis 3 \rrparenthesis$, 
then we have $f_3^{\w }\in k[f_1,f_2]^{\w }$. 
\end{rem}

\begin{rem}\label{exc:used for Claim B}\rm
(i) Let $(F,G)\in (\Aut _k\kx )^2$ be a non-proper SU pair. 
Then, we have $(f_1,f_2)=(g_1,g_2)$, 
since (SU1) yields 
$g_1-f_1\in 
(kf_3^2+kf_3)\cap k[f_2]
=F((kx_3^2+kx_3)\cap k[x_2])=\zs $ 
and $g_2-f_2\in kf_3\cap k=\zs $.

\nd{\rm (ii)} 
Let $(F,G)\in (\Aut _k\kx )^2$ be a weak SU pair 
such that $(f_1,f_2)=(g_1,g_2)$ and~$G\in \A $. 
Then, we have $3\in I_F$. 
Indeed, 
under the assumption that $(f_1,f_2)=(g_1,g_2)$, 
(SU$1'$) and (SU5) imply that $F\sim _3G$ and $\deg G<\deg F$, 
respectively. 
\end{rem}

\subsection{Induction}
By Remark~\ref{rem:ch2} (iii), 
$\{ \deg F\mid F\in \mathcal{A}\} $ 
is a well-ordered set, 
whose minimum element is $|\w |$ by 
(\ref{eq:deg F}) and Remark~\ref{rem:A} (i). 
Since Propositions~\ref{prop:key} and~\ref{prop:key2} 
are statements for $F\in \A $, 
we prove them by simultaneous induction on $\deg F$.

First, assume that $\deg F=|\w |$. 
Then, Proposition~\ref{prop:key} holds by Remark~\ref{rem:A}~(i), 
since $\deg FE\le \deg F=|\w |$ 
implies that $\deg FE=|\w |$. 
Since $\llparenthesis 1 \rrparenthesis$ 
implies that $\deg F>|\w |$ 
by Remark~\ref{rem:ch2} (v), 
Proposition~\ref{prop:key2} holds vacuously.

Next, we fix $\mu \in \{ \deg F\mid F\in \mathcal{A}\} \sm \{ |\w |\} $. 
We aim to establish that 
if Propositions~$\ref{prop:key}$ and $\ref{prop:key2}$ 
hold for $\deg F<\mu $, 
then they also hold for $\deg F=\mu $. 
To this end, it suffices to prove the following claim.

\begin{klaim}\it 
Let $(\dag )$ and $(\ddag )$ be the following statements$:$ 

$(\dag )$ {\rm 
Proposition~{\rm \ref{prop:key}} 
holds for $\deg F<\mu ;$ \ \ 
$(\ddag )$ 
Proposition~{\rm \ref{prop:key2}} holds for $\deg F\le \mu $.} 

\nd Then, the following assertions hold$:$

\nd 
{\rm (A)} Assume {\rm ($\dag $)}. 
Then, Proposition~$\ref{prop:key2}$ 
holds for $\deg F=\mu $.

\nd {\rm (B)} Assume {\rm ($\dag $)} and {\rm ($\ddag $)}. 
Then, Proposition~$\ref{prop:key}$ 
holds for $\deg F=\mu $. 
\end{klaim}

\begin{rem}
\label{rem:induction}\rm 
Assume ($\dag $). 
Then, in view of Example~\ref{exc:AAA}, 
we have $F\E \subset \A $ 
for all $F\in \A $ with $\deg F<\mu $. 
By repeatedly applying this, 
we see that 
if $F_1,\ldots ,F_r\in \Aut _k\kx $ satisfy 
$F_1\in \A $, 
$\deg F_i<\mu $ for $i=1,\ldots ,r-1$, 
and $F_1\rightsquigarrow F_2\rightsquigarrow \cdots \rightsquigarrow F_r$, 
then $F_r\in \A $. 
We call such a sequence 
$F_1,\ldots ,F_r$ an {\it inductive path}. 
\end{rem}

\begin{rem}\label{rem:sharpflat}\rm 
(i) 
Assume ($\dag $), 
and let $F\in \Aut _k\kx $ be such that $\deg F\le \mu $. 
If $F$ satisfies (A2), 
then $F$ also satisfies the following condition: 
\begin{enumerate}
\item[\rm (A3)] 
$(F_{\sigma },G)$ is an SU pair 
for some $\sigma \in \sym _3$ and $G\in \A $. 
\end{enumerate}
Indeed, 
applying Lemma~\ref{prop:equivalence} (i) to 
$(F_{\sigma },G)$ in (A2), 
we obtain 
$E_1\in \E _1$ and $E_2\in \E _2$ such that 
$(F_{\sigma },GE_1E_2)$ is an SU pair and $\deg G=\deg GE_1$. 
Since $G\in \A $, 
and $\deg G<\deg F_\sigma \le \mu $ by Lemma~\ref{prop:equivalence}~(ii), 
$G$, $GE_1$, $GE_1E_2$ is an inductive path, 
implying $GE_1E_2\in \A $.

\nd (ii) 
Assume ($\dag $), 
and let $F\in \A $ be such that $|\w |<\deg F\le \mu $. 
Then, 
by Remarks~\ref{rem:A}~(iv) and \ref{rem:sharpflat} (i), 
$F$ satisfies (A1) (i.e., $I_F\neq \emptyset $) or (A3). 
\end{rem}

\begin{lem}\label{claim:1}
Assume $(\dag )$, 
and let $F\in \Aut _k\kx $ be such that $\deg F\le \mu $. 

\noindent
{\rm (i)} 
If $l\in I_F$, then $FE\in \mathcal{A}$ 
holds for all $E\in \E _l$.

\noindent
{\rm (ii)} 
If distinct $i,j,l\in \{ 1,2,3\} $ satisfy 
$l\in I_F$ and $f_i^{\w }\in k[f_j^{\w }]$, 
then we have $i\in I_F$. 

\end{lem}

\begin{proof}
We may assume that $(i,j,l):=(1,2,3)$. 
Then, 
since $3\in I_F$, 
there exists $\phi \in k[S_3]$ such that 
$G:=(f_1,f_2,f_3-\phi )\in \A $ 
and $\deg G<\deg F\le \mu $ by definition.

(i) Let $E\in \E _3$. Then, $G$, $FE$ is an inductive path, 
proving $FE\in \A $.

(ii) Since $f_1^{\w }\in k[f_2^{\w }]$, 
there exist $\alpha \in k^*$ and $t\in \N _+$ such that 
$f_1^{\w }=\alpha (f_2^{\w })^t=(\alpha f_2^t)^{\w }$. 
Then, 
$G$, $(f_1-\alpha f_2^t,f_2,f_3-\phi )$, $(f_1-\alpha f_2^t,f_2,f_3)=:F'$ 
is an inductive path, so $F'\in \A $. 
Since $F\sim _1F'$ and $\deg F>\deg F'$, 
we obtain $1\in I_F$. 
\end{proof}

\subsection{Proofs of Claim (A) and (B)}\label{subsect:kouhan}

We now deduce (A) from Lemma~\ref{prop:structure2}. 
Assume ($\dag $), 
and let $F\in \A $ satisfy 
$\llparenthesis 1 \rrparenthesis$--$\llparenthesis 4 \rrparenthesis$ 
and $\deg F=\mu $. 
We need to show that $3\in I_F$. 
Since $\deg F=\mu $, 
we can apply Remark~\ref{rem:sharpflat} and Lemma~\ref{claim:1} to $F$. 
Hence, by Remark~\ref{rem:sharpflat}~(ii), 
$F$ satisfies at least one of 
$1\in I_F$, 
$2\in I_F$, 
$3\in I_F$, 
or (A3). 
Thus, it suffices to verify the chain of implications 
$2\in I_F\Rightarrow 1\in I_F\Rightarrow $ 
(A3) $\Rightarrow 3\in I_F$.

\smallskip

\nd $2\in I_F\Rightarrow 1\in I_F$\textbf{:} 
By Remark~\ref{rem:prop key2} (iii) 
and Lemma~\ref{prop:structure2} (ii), 
we have 
$f_1^{\w }\approx (f_3^{\w })^2$. 
Hence, 
applying Lemma~\ref{claim:1}~(ii) with $(i,j,l)=(1,3,2)$ 
yields $1\in I_F$.

\nd $1\in I_F\Rightarrow $ (A3)\textbf{:} 
By Remark~\ref{rem:prop key2} (iii), 
we see that 
Lemma~\ref{prop:structure2} (iii) applies to $F$. 
Let $F':=(f_1',f_2,f_3)$ be as therein. 
Then, since $1\in I_F$ and~$F\sim _1F'$, 
Lemma~\ref{claim:1}~(i) yields $F'\in \A $. 
By (1$^\star $) and Remark~\ref{rem:ch2} (v), 
we have $\deg F'>|\w |$. 
By (2$^\star $), 
$F'$ does not satisfy (A1). 
Thus, 
by Remark~\ref{rem:A} (iv), $F'$ must satisfy (A2), 
i.e., 
$(F'_\sigma ,G)$ is a weak SU pair 
for some $\sigma \in \sym _3$ and $G\in \A $. 
Then, 
$(F,G)$ is a weak SU pair by (3$^\star $), 
so $F$ satisfies (A2). 
Therefore, 
$F$ satisfies (A3) 
by Remark~\ref{rem:sharpflat}~(i).

\nd (A3) $\Rightarrow 3\in I_F$\textbf{:} 
The assertion follows from 
Lemma~\ref{prop:structure2} (i) 
and Remark~\ref{exc:used for Claim B} (ii).

\smallskip

To prove (B), 
we first deduce the following lemma from 
Theorem~\ref{ex:ineq1}. 
Here, for distinct $s,t,u\in \{ 1,2,3\} $, 
$\mathfrak{f}(s,t,u)$ 
denotes 
the conditions 
$\max \{ \deg f_t,\deg f_u\} <\deg f_s$ and 
$f_s^{\w },f_u^{\w }\not\in k[f_t^{\w }]$.

\begin{lem}\label{exc:ineq1}\rm
\nd {\rm (i)} 
Let $F\in \Aut _k\kx $. 
If $f_3^{\w }\in k[S_3]^{\w }$, 
and either $\mathfrak{f}(1,2,3)$ or

$1^\bullet $ \ 
$\deg f_2<\deg f_1=\deg f_3$, 
$f_1^{\w }\not\in k[f_2^{\w }]$, 
and $f_3^{\w }\not\in k[f_1^{\w },f_2^{\w }]$

\nd 
hold, 
then we have 
(1) $\mathcal{P}(f_2,f_1,m)\ne \emptyset $ for some $m\ge 1$, 
and (2) $(f_1^{\w })^2\approx (f_2^{\w })^b$ for some odd $b\ge 3$.

\nd{\rm (ii)} 
Let $F\in \Aut _k\kx $ and $\phi \in k[S_1]$. 
If $\deg \phi \le \deg f_1$, and 

$2^\bullet $ \ 
$\mathfrak{f}(i,j,1)$ and $\phi \not\in k[f_j]$

\nd 
hold for some $(i,j)\in \{ (2,3),(3,2)\} $, 
then $(f_i,f_j,f_1)$ satisfies 
$\llparenthesis 1 \rrparenthesis$--$\llparenthesis 4 \rrparenthesis$. 
\end{lem}

\begin{proof}
(i) 
Note that 
$f_3^{\w }\not\in k[f_1^{\w },f_2^{\w }]$ holds in either case, 
since $\deg f_3<\deg f_1$ and $f_3^{\w }\not\in k[f_2^{\w }]$ 
in Case $\mathfrak{f}(1,2,3)$. 
Now, 
since $f_3^{\w }\in k[S_3]^{\w }$, 
there exists 
$\phi \in k[S_3]$ such that $\phi ^{\w }=f_3^{\w }$ 
by Lemma~\ref{def:initial algebra}. 
Then, 
Remark~\ref{rem:(1.3.7)} (ii) gives $m_{\w }^{S_3}(\phi )\ge 1$. 
Since $\deg \phi =\deg f_3\le \deg f_1$, 
we also have $\phi \in \mathcal{P}(f_2,f_1,m_{\w }^{S_3}(\phi ))$, 
proving (1). 
Then, 
since $\deg f_2<\deg f_1$ and $f_1^{\w }\not\in k[f_2^{\w }]$, 
Theorem~\ref{ex:ineq1} (i) yields~(2).

(ii) 
Since $\phi \in k[S_1]=k[f_i,f_j]$, and 
$\deg \phi \le \deg f_1<\deg f_i$ ($*$), 
we have $\phi \in \mathcal{P}(f_j,f_i,m_{\w }^{S_1}(\phi ))$. 
Since $\phi \not\in k[f_j]$, 
Remark~\ref{rem:deg^fg} gives $\deg f_i\le \deg ^{S_1}\phi $, 
and hence $\deg \phi <\deg ^{S_1}\phi $ by ($*$). 
This proves $\llparenthesis 3 \rrparenthesis$ 
in view of Remark~\ref{rem:(1.3.7)} (i). 
Then, 
since $\deg f_j<\deg f_i$ and $f_i^{\w }\not\in k[f_j^{\w }]$, 
Theorem~\ref{ex:ineq1} (i) yields $\llparenthesis 1 \rrparenthesis$, 
and Theorem~\ref{ex:ineq1} (iii) combined with $(*)$ yields 
$\llparenthesis 2 \rrparenthesis$. 
Since $\mathfrak{f}(i,j,1)$ implies $f_1^{\w }\not\in k[f_j^{\w }]$, 
we get 
$\llparenthesis 4 \rrparenthesis$. 
\end{proof}

We now prove (B). 
Assume ($\dag $) and $(\ddag )$, 
and let $F\in \A $ satisfy $\deg F=\mu $. 
Then, by Remark~\ref{rem:sharpflat} (ii), 
$F$ satisfies either 
(1) (A3) and $I_F=\emptyset $, or (2) $I_F\ne \emptyset $. 
Now, 
take $E\in \E $ with $\deg FE\le \deg F$. 
Our goal is to show that 
$FE\in \A $. 
By Remark~\ref{rem:Xsigma} (ii) and (i), 
$E':=X_\sigma ^{-1}EX_\sigma \in \E $ 
and $F_\sigma E'=(FE)_\sigma $ 
hold for all $\sigma \in \sym _3$. 
Since $(FE)_\sigma \in \A $ implies $FE\in \A $ 
by Remark~\ref{rem:A} (ii), 
we may replace $F$ with $F_\sigma $ 
for any $\sigma \in \sym _3$. 
Hence, in view of Remark~\ref{rem:prop key2} (ii) as well, 
it suffices to consider 
the following two cases ($1'$) and ($2'$):

\smallskip

\nd ($1'$) 
\underline{$(F,G)$ is an SU pair for some $G\in \mathcal{A}$, 
and $I_F=\emptyset $}\textbf{:} 
By Remark~\ref{exc:used for Claim B} (i) and~(ii), 
$I_F=\emptyset $ implies 
that the SU pair $(F,G)$ is proper. 
Hence, by the assumption that $\deg FE\le \deg F$ 
and Lemma~\ref{prop:c1-c6} (i) or (ii), 
we obtain 
$\tau \in \sym _3$ 
and $H\in {\rm D}_3(k)$ such that 
$((FE)_\tau ,GH)$ is a weak SU pair, 
where properness is required for (ii). 
Since $G\in \A $, 
we have $GH\in \A $ by Remark~\ref{rem:A} (ii). 
Thus, 
applying Remark~\ref{rem:A} (iii) (A2) 
with $(F_\sigma ,G):=((FE)_\tau ,GH)$ 
yields $FE\in \A $.

\smallskip

\nd 
($2'$) 
\underline{$3\in I_F$}\textbf{:} 
Since $\deg F=\mu $, 
we can apply Lemma~\ref{claim:1} to $F$. 
Hence, 
if $E\in \E _3$, 
then $FE\in \A $ holds by Lemma~\ref{claim:1} (i). 
Below, we assume that $E\in \E _1\cup \E _2$.

Since $3\in I_F$, 
we see from Remark~\ref{rem:prop key2} (iii) 
that $f_3^{\w }\in k[S_3]^{\w }$ (a) 
and there exists $g_3\in f_3+k[S_3]$ 
such that $\deg g_3<\deg f_3$ (b). 
By Remark~\ref{rem:ch2} (iv), 
we may assume that $g_3^{\w }\not\in k[S_3]^{\w }$ (c). 
Set $G:=(f_1,f_2,g_3)$. 
Then, we have 
$G\in \A $ (d) by Lemma~\ref{claim:1}~(i), 
and $\deg G<\deg F=\mu $ (e) by (b).

\smallskip

We first derive $FE\in \A $ from ($\ddag $) 
when $F$ satisfies 
$1^{\bullet }$, 
$\mathfrak{f}(1,2,3)$, 
or $\mathfrak{f}(2,1,3)$.

\nd $1^\bullet $ or $\mathfrak{f}(1,2,3)$\textbf{:} 
We show that $(F,G)$ is a weak SU pair. 
Then, 
since $E\in \E _1\cup \E _2$, 
without invoking properness, 
Lemma~\ref{prop:c1-c6}~(i) and (d) 
yield $FE\in \A $ as in Case~($1'$).

The conditions other than (SU6) are verified directly. 
For (SU6), by (a), (1) and (2) in 
Lemma~\ref{exc:ineq1}~(i) hold for $F$. 
It follows that $G$ satisfies 
$\llparenthesis 1 \rrparenthesis$ and $\llparenthesis 3 \rrparenthesis$, 
and $(\deg f_1,\deg f_2)=(b\delta ,2\delta )$ 
with $\delta :=(1/2)\deg f_2$. 
Now, suppose that (SU6) does not hold, 
i.e., 
$(b-2)\delta +\deg df_1\wedge df_2\le \deg g_3$. 
Then, 
together with (b) and $\deg f_3\le \deg f_1=b\delta $, 
$G$ satisfies $\llparenthesis 2 \rrparenthesis$. 
Thus, 
in view of (d) and (e), 
applying ($\ddag $) and Remark~\ref{rem:prop key2} (iv) to $G$ 
yields $g_3^{\w }\in k[f_1,f_2]^{\w }=k[S_3]^{\w }$. 
This contradicts (c).

\nd $\mathfrak{f}(2,1,3)$\textbf{:}  
Let $\tau :=(1,2)$. Then, 
since $E\in \E _1\cup \E _2$, 
Remark~\ref{rem:Xsigma}~(ii) gives 
$E':=X_\tau ^{-1}EX_\tau \in \E _1\cup \E _2$. 
Since 
$F_\tau$ satisfies $\mathfrak{f}(1,2,3)$, 
and $G_\tau \in \A $ by (d) and Remark~\ref{rem:A} (ii), 
we can also verify that 
$(F_\tau ,G_\tau )$ is a weak SU pair as in the preceding case. 
Thus, 
$(FE)_\tau =F_\tau E'\in \A $ holds as before, 
yielding $FE\in \A $ via Remark~\ref{rem:A}~(ii).

\smallskip

By symmetry, we assume that $E\in \E _1$ below. 
Write $FE=(f_1+\phi _1,f_2,f_3)$, 
where $\phi _1\in k[S_1]$ 
with $\deg \phi _1\le \deg f_1$ (f), 
since $\deg FE\le \deg F$ by assumption. 
Now, assume that $2^\bullet $ 
holds for some $(i,j)\in \{ (2,3),(3,2)\} $. 
Then, 
by (f) and Lemma~\ref{exc:ineq1} (ii), 
$F':=(f_i,f_j,f_1)$ satisfies 
$\llparenthesis 1 \rrparenthesis$--$\llparenthesis 4 \rrparenthesis$. 
Since $F'\in \A $ by Remark~\ref{rem:A} (ii) 
and $\deg F'=\deg F=\mu $, 
applying ($\ddag $) to $F'$ yields $3\in I_{F'}$. 
By Remark~\ref{rem:prop key2} (ii), 
this implies that $1\in I_F$. 
Hence, since $E\in \E _1$, we obtain $FE\in \A $ by Lemma~\ref{claim:1}~(i).

\smallskip

For the remaining cases, 
we consider the following three conditions:

\maru{1} $f_1^{\w }\in k[f_2^{\w }]$; \quad 
\maru{2} $\phi _1\in k[f_2]$;

\maru{3} 
$f_3^{\w }=af_1^{\w }+b(f_2^{\w })^l$ 
for some $(a,b)\in k^2\sm \zs $ and $l\in \N _+$.

\begin{klaim}\label{claim:2}\it 
If one of the following 
{\rm p}, {\rm q}, or {\rm r} holds, 
then we have $FE\in \mathcal{A}$$:$

\nd {\rm p}. 
One of {\rm \maru{1}}, {\rm \maru{2}}, or {\rm \maru{3}} 
holds$;$ 

\nd {\rm q}. 
$f_2^{\w }\in k[f_1^{\w }]$, 
and one of 
{\rm \maru{1}}, {\rm \maru{2}}, or {\rm \maru{3}} 
holds with $f_2$ and $f_3$ interchanged$;$

\nd {\rm r}. 
$\deg f_1<\deg f_2=\deg f_3$, 
$\deg f_3<\deg f_1=\deg f_2$, 
or 
$\deg f_1=\deg f_2=\deg f_3$. 
\end{klaim}

\begin{proof}
Recall that $3\in I_F$ by assumption.

\nd 
p. \maru{1} 
Since $3\in I_F$ and $f_1^{\w }\in k[f_2^{\w }]$, 
applying Lemma~\ref{claim:1} (ii) with $(i,j,l)=(1,2,3)$ 
yields $1\in I_F$. 
Since $E\in \E _1$, 
this gives $FE\in \A $ by Lemma~\ref{claim:1} (i).

\nd 
\maru{2} 
From (d), (e), (f), 
and $\phi _1\in k[f_2]$, 
we see that 
$G$, $(f_1+\phi _1,f_2,g_3)$, $(f_1+\phi _1,f_2,f_3)=FE$ 
is an inductive path, 
yielding $FE\in \A $.

\nd 
\maru{3} 
Set $f:=f_3-(af_1+bf_2^l)$ 
and $F':=(f_1,f_2,f)$. 
Then, 
we have $\deg f<\deg f_3$, 
and $F'\in \A $ by Lemma~\ref{claim:1}~(i), 
since $3\in I_F$. 
Now, 
if $a=0$, 
then similarly to \maru{2}, 
$F'$, $(f_1+\phi _1,f_2,f)$, $FE $ is an inductive path, 
so $FE\in \A $. 
If $a\ne 0$, 
then we have $\deg f_1=\deg f_3$. 
Hence, $F'$, 
$(f_1+a^{-1}f+a^{-1}bf_2^l,f_2,f)=(a^{-1}f_3,f_2,f)$, 
$(a^{-1}f_3,f_2,-a(f_1+\phi _1))
=(FEH)_{(1,3)}$ 
is an inductive path, 
where $H:=(-ax_1,x_2,a^{-1}x_3)\in \D _3(k)$. 
This yields $FE\in \A $ via Remark~\ref{rem:A} (ii).

\nd 
q. 
Since $3\in I_F$ and $f_2^{\w }\in k[f_1^{\w }]$, 
applying Lemma~\ref{claim:1}~(ii) with $(i,j,l)=(2,1,3)$ 
yields $2\in I_F$. 
Then, 
the assertions hold analogously to Case p.

\nd 
r. 
This case can be proved similarly to \cite[p.~105, Claim~5]{tame3}, 
using Remarks~\ref{rem:deg^fg} 
and \ref{rem:(1.3.7)} (i) 
and Claim p. 
Hence, we omit the details. 
\end{proof}

Therefore, 
we may assume the following, 
where $'$ denotes the negation:

p$'$. \maru{1}$'$, \maru{2}$'$, and \maru{3}$'$ hold;

q$'$. If $f_2^{\w }\in k[f_1^{\w }]$, then 
\maru{1}$'$, \maru{2}$'$, and \maru{3}$'$ 
hold with $f_2$ and $f_3$ interchanged.

\begin{klaim}\label{rem:five cases}\it 

Assume {\rm p}$'$ and {\rm q}$'$. 
Then, the following statements are true$:$

\nd 
{\rm (1)} If $\max \{ \deg f_2,\deg f_3\} <\deg f_1$, 
then $\mathfrak{f}(1,2,3)$ holds.

\nd 
{\rm (2)} If $\max \{ \deg f_1,\deg f_2\} <\deg f_3$, 
then $2^\bullet $ holds for $(i,j)=(3,2)$.

\nd 
{\rm (3)} If $\max \{ \deg f_1,\deg f_3\} <\deg f_2$ 
and $f_2^{\w }\in k[f_1^{\w }]$, 
then $2^\bullet $ holds for $(i,j)=(2,3)$.

\nd 
{\rm (4)} If $\max \{ \deg f_1,\deg f_3\} <\deg f_2$ 
and $f_2^{\w }\not\in k[f_1^{\w }]$, 
then $\mathfrak{f}(2,1,3)$ holds.

\nd 
{\rm (5)} If $\deg f_2<\deg f_1=\deg f_3$, 
then $1^{\bullet }$ holds. 
\end{klaim}

\begin{proof}
(1), (2), and (3) are straightforward, 
since \maru{3}$'$ implies $f_3^{\w }\not\in k[f_2^{\w }]$.

(4) 
Supposing $f_3^{\w }\in k[f_1^{\w }]$, 
we have $f_3^{\w }\approx (f_1^{\w })^l$, 
where $l\ge 2$ by \maru{3}$'$. 
Then, since $\deg f_3<\deg f_2$, 
$f_2^{\w }\not\in k[(f_1^{\w })^l]=k[f_3^{\w }]$, 
and $\phi _1\not\in k$ by \maru{2}$'$, 
Corollary~\ref{claimm:ineq1} yields 
$\deg \phi _1>(1/2)\deg f_3\ge (1/l)\deg f_3=\deg f_1$, 
contradicting~(f).

(5) 
We have $f_1^{\w }\not\in k[f_2^{\w }]$ by \maru{1}$'$, 
and $f_3^{\w }\not\in k[f_1^{\w },f_2^{\w }]$ by \maru{3}$'$. 
\end{proof}

Now, observe that (1)--(5) and r cover all possibilities. 
Since $FE\in \A $ holds in all these cases, 
the proof of (B) is completed. 
We remark that, if $\w $ is independent, 
then Claim r is not required. 
In fact, 
\maru{1}$'$, \maru{3}$'$, and (\ref{eq:approx}) 
imply that $\deg f_1$, $\deg f_2$, and $\deg f_3$ 
are mutually distinct. 
Hence, one of (1)--(4) must hold.

\section{Properties of weak Shestakov-Umirbaev pairs}\label{sect:WSU}
\setcounter{equation}{0}

In this section, 
we establish key properties of the weak SU pairs 
and prove the three lemmas in \S~\ref{subsect:fundamental lemmas}. 
Throughout, 
let $(F,G)\in (\Aut _k\kx )^2$ be a weak SU pair 
unless otherwise stated, 
and let $s$, $\delta $, 
$b$, $e$, $a$, $c$, and $\psi $ 
be as defined in Claims~\ref{claim:A}, \ref{claim:g2}, 
and \ref{claim:g1} below. 
It is shown in Claims~\ref{claim:A} and~\ref{claim:g2} 
that $(F,G)$ satisfies (SU2) and (SU3).

\subsection{Descriptions of $g_1$ and $g_2$}\label{subsect:g1g2}

By (SU1$'$), 
we have $\phi _3:=f_3-g_3\in k[g_1,g_2]$. 
Moreover, (SU5) gives $\deg \phi _3=\deg f_3$ and 
$\phi _3^{\w }=f_3^{\w }$ ($5^{\rm s}$). 
From (SU4) and ($5^{\rm s}$), we get 
$\deg \phi _3\le \deg g_1$ and 
$\phi _3^{\w }\not\in k[g_1^{\w },g_2^{\w }]$. 
Thus, by Remark~\ref{rem:(1.3.7)} (ii), 
we obtain 
\begin{equation}\label{eq:m(F,G)}
\phi _3\in \mathcal{P}(g_2,g_1,m_G^F), 
\text{ \ where \ } 
m_G^F\ge 1. 
\end{equation}
Therefore, 
together with (SU3$'$) and the assumption that $p\ge 7$, 
Theorem~\ref{ex:ineq1} yields the following claim, 
in which (iii) also uses ($5^{\rm s}$).

\begin{claim}\label{claim:A}
There exists an odd $s\ge 3$ such that the following 
{\rm (i)}--{\rm (iii)} hold$:$

\nd{\rm (i)} 
We have $(g_1^{\w })^2\approx (g_2^{\w })^s$. 
Hence, {\rm (SU3)}~holds.

\nd{\rm (ii)} 
Exactly one of 
$m_G^F=1$, 
$(s,m_G^F)=(3,2)$, 
or $(s,p,m_G^F)=(3,7,7)$ holds.

\nd{\rm (iii)} 
We have 
$\deg f_3\ge (s-2)\delta +\deg dg_1\wedge dg_2+\ep (g_2)$, 
where $\delta :=(1/2)\deg g_2$. 
\end{claim}

We note that 
Claim~\ref{claim:A} (i) and (iii) and (SU4) 
imply the following:

($3^{\rm s}$) $(\deg g_1,\deg g_2)=(s\delta ,2\delta )$. 
\quad 
($3^{\rm u}$) $\delta \le (s-2)\delta <\deg f_3\le s\delta $. 

($4^{\rm s}$) $\ep (g_1)\ge \max \{ \deg df_3-\deg dg_1,0\} $. 
\quad 
($4^{\rm u}$) $g_1^{\w },g_2^{\w }\not\in kf_3^{\w }$.

\begin{claim}\label{claim:g2}
\nd {\rm (i)} 
$g_2-f_2\in kf_3+k$, i.e., 
$g_2=f_2+bf_3+e$ for some $b,e\in k$.

\nd{\rm (ii)} 
If $\deg f_2<\deg f_3$, 
then $b=0$. 
Hence, if $\deg f_2\ne \deg f_3$, 
then $g_2^{\w }=f_2^{\w }$.

\nd{\rm (iii)} We have $\deg f_2=\deg g_2=2\delta $ 
and $f_2^{\w }\not\approx f_3^{\w }$. 
Hence, {\rm (SU2)} holds.

\nd{\rm (iv)} 
$f_i^{\w }\not\in k[f_j^{\w }]$ holds for $(i,j)=(2,3),(3,2)$.

\nd{\rm (v)} 
If $f_2^{\w }$ and $f_3^{\w }$ are algebraically dependent over $k$, 
then we have $\deg f_3<s\delta $. 
\end{claim}
\begin{proof}
We have 
$g_2-f_2\in k[f_3]$ by (SU$1'$), 
$\deg f_3>\delta $ by ($3^{\rm u}$), 
and $\deg (g_2-f_2)\le \deg g_2=2\delta $ 
by (SU2$'$) and ($3^{\rm s}$). 
Hence, (i) holds. 
(ii) and (iii) follow from (i), 
($4^{\rm u}$), and $\deg g_2\ge \deg f_2$. 
Since $f_2^{\w }\not\approx f_3^{\w }$ by (iii), 
it follows from Remark~\ref{rem:dependence} that, 
if $\deg f_2=\deg f_3$, 
then $f_2^{\w }$ and $f_3^{\w }$ 
are algebraically independent over $k$. 
In this case, (iv) and (v) are clear. 
Assume that $\deg f_2\ne \deg f_3$. 
Then, 
since $\deg f_2=2\delta <\deg f_3^2$ 
by (iii) and ($3^{\rm u}$), 
we obtain $f_2^{\w }\not\in k[f_3^{\w }]$. 
Moreover, 
the last part of 
(ii) gives $g_2^{\w }=f_2^{\w }$. 
Hence, (SU4) yields 
$f_3^{\w }\not\in k[g_2^{\w }]=k[f_2^{\w }]$, 
proving (iv). 
If $f_2^{\w }$ 
(and hence $g_2^{\w }$) 
and $f_3^{\w }$ are algebraically dependent over $k$, 
then so are $g_1^{\w }$ and $f_3^{\w }$ by 
Claim~\ref{claim:A} (i). 
Since $g_1^{\w }\not\approx f_3^{\w }$ by~($4^{\rm u}$), 
this implies that 
$\deg f_3\ne \deg g_1$ by Remark~\ref{rem:dependence}. 
Hence, we obtain (v) by (SU4) and~($3^{\rm s}$). 
\end{proof}

\begin{definition}\rm 
We say that a weak SU pair 
$(F,G)$ is of {\it ps type} if $p\mid s$, 
and of {\it pm7 type} if 
$(s,p,m_G^F)=(3,7,7)$. 
\end{definition}

By (\ref{eq:m(F,G)}) and ($3^{\rm s}$), 
the setting (S) defined in \S ~\ref{subsect:basic concepts} 
holds with 
\begin{equation}\label{eq:(S) for phi3}
(f,g,\phi ,m,a,b):=(g_2,g_1,\phi _3,m^F_G,2,s). 
\end{equation}

\begin{claim}\label{claim:pm7}

\nd{\rm (i)} 
If $(F,G)$ is of non-pm7 type, 
then 
\begin{equation}\label{eq:pfSUreduction1.5}
\deg dg_2\wedge d\phi _3\ge 
\deg g_1+\deg dg_1\wedge dg_2. 
\end{equation}

\nd{\rm (ii)} 
If $m_G^F=1$ and $p\nmid s$, 
then $\deg dg_1-\deg dg_2=\deg g_1-\deg g_2=(s-2)\delta $.

\nd{\rm (iii)} 
If $(s,m_G^F)=(3,2)$, 
then $\deg f_3>2\delta $ 
and 
$(1/2)\delta <\deg dg_1-\deg dg_2<(3/2)\delta $.

\nd{\rm (iv)} 
If $(F,G)$ is of pm7 type, 
then $\deg f_3>2\delta $, 
$\ep (g_1,g_2)<(1/7)\delta $, 
$\deg dg_1\wedge dg_2\le (1/8)\delta $, 
and 
$(6/7)\delta <\deg dg_1-\deg dg_2<(8/7)\delta $.

\nd{\rm (v)} 
If $(F,G)$ is of ps type, 
then 
$\deg f_3>5\delta $ by {\rm ($3^{\rm u}$)}, 
and $m_G^F=1$ by Claim~{\rm \ref{claim:A} (ii)}.

\end{claim}
\begin{proof}
(i) follows from Lemma~\ref{lem:deg P'(g)} (ii) 
since $m_G^F\in \{ 1,2\} \subset \{ 1,\ldots ,p-1\} $. 
(ii) follows from (\ref{eq:m(F,G)}) and Lemma~\ref{lem:deg P'(g)} (iii). 
For (iii), 
by ($5^{\rm s}$) and Example~\ref{cor:nonsingular}~(b), 
we have $\deg f_3=\deg \phi _3>2\delta $. 
Observe that Remark~\ref{rem:independence} (ii) and ($3^{\rm s}$) give 
\begin{equation}\label{eq:1/2<<3/2}
\begin{gathered}
\deg dg_1-\deg dg_2
\ge \deg dg_1-\deg g_2
=(3\delta -\ep (g_1))-2\delta 
=\delta -\ep (g_1), \\
\deg dg_1-\deg dg_2
\le \deg g_1-\deg dg_2
=3\delta -(2\delta -\ep (g_2))
=\delta +\ep (g_2). 
\end{gathered}
\end{equation}
By (\ref{eq:m(F,G)}) and Remark~\ref{rem:1/7} (ii), 
we have $\ep (g_1,g_2)<(1/2)\delta $. 
Thus, (\ref{eq:1/2<<3/2}) yields 
$(1/2)\delta <\deg dg_1-\deg dg_2<(3/2)\delta $. 
(iv) can be proved similarly, 
using Example~\ref{cor:nonsingular} (c), 
Remark~\ref{rem:1/7}~(i), 
and~(\ref{eq:1/2<<3/2}). 
(v) is clear since $s\ge p\ge 7$. 
\end{proof}

Next, we define $\gamma _i:=\gamma _i^F:=\deg df_j\wedge df_l$ 
for $i=1,2,3$, 
where $1\le j<l\le 3$ and $j,l\ne i$. 
We note that $\gamma _1=\deg dg_2\wedge df_3$, 
since $g_2\in f_2+k[f_3]$ by (SU$1'$).

\begin{claim}\label{claim:D}
\nd{\rm (i)} 
If $(\ref{eq:pfSUreduction1.5})$ holds 
$($e.g., if $(F,G)$ is of non-pm$7$ type$)$, 
then we have 
$\gamma _1=\deg dg_2\wedge d\phi _3$. 
Hence, 
$\gamma _1\ge \deg g_1+\deg dg_1\wedge dg_2>s\delta $ 
holds by $(\ref{eq:pfSUreduction1.5})$.

\nd{\rm (ii)} 
If $(F,G)$ is of pm$7$ type, 
then we have $\gamma _1>(95/56)\delta $. 
\end{claim}
\begin{proof}
(i) 
Since $f_3=g_3+\phi _3$, we have 
$\gamma _1=\deg dg_2\wedge df_3
=\deg (dg_2\wedge dg_3+dg_2\wedge d\phi _3)$. 
Since 
$\deg dg_2\wedge d\phi _3\ge \deg g_1+\deg dg_1\wedge dg_2
>\deg g_2+\deg g_3\ge \deg dg_2\wedge dg_3$ 
by (\ref{eq:pfSUreduction1.5}), 
(SU6), 
and Remark~\ref{rem:independence} (iii), 
it follows that $\gamma _1=\deg dg_2\wedge d\phi _3$.

(ii) 
Recall that $m_G^F=p=7$. 
Hence, 
by (\ref{eq:(S) for phi3}) and ($5^{\rm s}$), 
applying Theorem~\ref{cor:h-q} (ii) with 
$(f,g,h):=(g_2,g_1,f_3)$ 
yields 
$\deg g_3=\deg (f_3-\phi _3)>
16\delta -7\deg dg_2\wedge df_3-\deg f_3\ge 
13\delta -7\gamma _1$, 
since $\deg f_3\le \deg g_1=3\delta $ by (SU4). 
Since $\deg g_3<(3-2+1/8)\delta =(9/8)\delta $ 
by (SU6) and Claim~\ref{claim:pm7} (iv), 
it follows that~$\gamma _1>(95/56)\delta $. 
\end{proof}

\begin{claim}\label{claim:S1}
Let $\phi \in k[S_1]$ be such that $\deg \phi \le s\delta $. 
Then, we have $\deg ^{S_1}\phi \le s\delta $. 
\end{claim}
\begin{proof}
Supposing $\deg ^{S_1}\phi >s\delta $, 
we have 
$\deg \phi <\deg ^{S_1}\phi $, 
since $\deg \phi \le s\delta $ by assumption. 
By Remark~\ref{rem:(1.3.7)}~(i), 
this implies that $\phi \in \mathcal{D}(f_2,f_3)$. 
Hence, 
Claim~\ref{claim:g2} (iv) and Corollary~\ref{cor:pSU0}~(iii) yield 
$\deg \phi >\gamma _1$. 
If $(F,G)$ is of non-pm7 type, 
then this leads to the contradiction $\deg \phi >s\delta $ 
via Claim~\ref{claim:D}~(i). 
Below, we assume that $(F,G)$ is of pm7 type, 
and so $s=3$.

Since 
$\mathcal{D}(f_2,f_3)\ne \emptyset $, 
there exist $i,j\in \N _+$ 
such that $\gcd (i,j)=1$ and 
$(f_2^{\w })^j\approx (f_3^{\w })^i$. 
Then, 
since $\deg f_3=(j/i)\deg f_2=(2j/i)\delta $ 
by Claim~\ref{claim:g2} (iii), 
we get $2\delta <(2j/i)\delta <3\delta $ 
by Claims~\ref{claim:pm7} (iv) and \ref{claim:g2} (v). 
This implies that $(i,j)\in I(3,4)$. 
Hence, 
by Lemma~\ref{lem:increasing function}~(ii), 
$\zeta _l(i,j,3i/2)>0$ holds for $l=0,1$. 
However, the assumption yields 
$\deg \phi \le 3\delta =(3i/2)\cdot (1/i)\deg f_2$, 
which contradicts Corollary~\ref{cor:pSU0}~(ii). 
\end{proof}

By (SU$2'$), we have $\deg f_1\le \deg g_1$. 
We denote by ``Case $<$" and ``Case $=$" 
the cases where $\deg f_1<\deg g_1$ and $\deg f_1=\deg g_1$, 
respectively.

\begin{claim}\label{claim:g1}
\nd{\rm (i)} 
Set $I:=\{ (0,1),(0,2)\} \cup \{ (i,0)\mid 
i=0,\ldots ,(s-1)/2\} $. 
Then, we have $g_1-f_1\in \sum _{(i,j)\in I}kf_2^if_3^j$, 
i.e., 
$g_1=f_1+af_3^2+cf_3+\psi $ for some 
$a,c\in k$ and $\psi \in \sum _{i=0}^{(s-1)/2}kf_2^i$. 
Since $\deg f_2=2\delta $, 
we have $\deg \psi \le (s-1)\delta <\deg g_1$.

\nd{\rm (ii)} 
If $a\ne 0$, 
then $s=3$, $\deg f_3\le (3/2)\delta $, and $m_G^F=1$.

\nd{\rm (iii)} 
In Case $<$, we have $a\ne 0$ and $\deg f_3=(1/2)\deg g_1=(3/2)\delta $.

\nd{\rm (iv)} 
Assume that $\deg f_1=\deg g_1=\deg f_3$. 
Then, we have $f_1^{\w }\not\approx f_3^{\w }$. 
If moreover $c\ne 0$, 
then $f_1^{\w }$ and $g_1^{\w }$ are algebraically independent over $k$.

\end{claim}
\begin{proof}
(i) 
We have 
$\phi _1:=g_1-f_1\in k[S_1]$ by (SU1$'$). 
It satisfies $\deg \phi _1\le s\delta $, 
since $\deg f_1\le \deg g_1=s\delta $. 
Hence, Claim~\ref{claim:S1} yields $\deg ^{S_1}\phi _1\le s\delta $. 
This implies that $\phi _1\in \sum _{(i,j)\in J}kf_2^if_3^j$, 
where $J:=\{ (i,j)\in \N ^2\mid \deg f_2^if_3^j\le s\delta \} $. 
Since $\deg f_2=2\delta $ 
by Claim~\ref{claim:g2} (iii), 
$\deg f_3>(s-2)\delta $ by ($3^{\rm u}$), 
and $s\ge 3$, 
we obtain $J\subset I$.

(ii) 
If $a\ne 0$, then we have $(0,2)\in J$. 
This gives 
$2(s-2)\delta <\deg f_3^2\le s\delta $. 
Since $s\ge 3$, 
it follows that $s=3$ and $\deg f_3\le (3/2)\delta $. 
The latter implies that $m_G^F=1$, 
for if $m_G^F\ge 2$, 
then Claim~\ref{claim:pm7} (iii) and (iv) 
yield $\deg f_3>2\delta $, 
a contradiction.

(iii) 
Since $\deg f_1<\deg g_1$, 
$\deg \psi <\deg g_1$, 
and $g_1^{\w }\not\approx f_3^{\w }$ by ($4^{\rm u}$), 
we must have 
$a\ne 0$ and $2\deg f_3=\deg f_3^2=\deg g_1=s\delta $. 
From $a\ne 0$, 
we obtain $s=3$ by (ii).

(iv) 
The assumption yields $g_1^{\w }=f_1^{\w }+cf_3^{\w }$. 
Hence, ($4^{\rm u}$) gives $f_1^{\w }\not\approx f_3^{\w }$. 
Since $g_1^{\w }=f_1^{\w }+cf_3^{\w }\not\approx f_1^{\w }$ if $c\ne 0$, 
the last part follows from Remark~\ref{rem:dependence}. 
\end{proof}

For convenience, we summarize 
the properties of $(F,G)$ obtained so far as follows 
(see ($3^{\rm s}$), ($3^{\rm u}$), and 
Claims~\ref{claim:g2} (ii) and (iii), 
\ref{claim:pm7} (iii) and (iv), 
and \ref{claim:g1} (ii) and (iii)).

\begin{center}
Table 2 \ ($m:=m_G^F$)

\begin{tabular}{|c||cccc|}
\hline
 & \multicolumn{1}{|c|}{Case $<$} & 
\multicolumn{3}{|c|}{Case $=$ \ \ ($\deg f_1=\deg g_1=s\delta $)} \\ \hline
\hline
$s$ & \multicolumn{3}{|c|}{$s=3$} & \multicolumn{1}{|c|}{$s\ge 5$, 
ps type} \\ \hline
\multirow{2}{*}{$m$} & 
\multicolumn{2}{c}{\multirow{2}{*}{$m=1$}}& 
\multicolumn{1}{|c|}{$m=2$, $p=m=7$} &  \\  \cline{4-4} \vspace{-1.5mm}
& \multicolumn{4}{|c|}{} \\ \hline 
$f_3$ & 
\multicolumn{1}{|c|}{$\deg f_3=(3/2)\delta $} & 
\multicolumn{1}{|c|}{$\delta <\deg f_3\le 2\delta $} & 
\multicolumn{1}{|c|}{$2\delta <\deg f_3\le 3\delta $} & 
\multicolumn{1}{|c|}{$\deg f_3>3\delta $} \\ \hline 
$f_2$ & \multicolumn{4}{|c|}{$\deg f_2=2\delta $} 
\\ \hline
$a$, $b$ & $a\ne 0$ & 
\multicolumn{1}{|c|}{---} &
\multicolumn{2}{|c|}{$a=b=0$} \\ \hline 
$\clubsuit $ & $3\deg f_2=4\deg f_3$ & 
\multicolumn{3}{|c|}{$2\deg f_1=s\deg f_2$} \\ \hline 
\end{tabular}

\end{center}

\medskip

Case $<$ of the following claim will be verified in \S ~\ref{subsect:dodf}, 
after Claim~\ref{claim:P5}.

\begin{claim}\label{claim:H}
\nd {\rm (i)} 
We have $\deg f_2<\deg f_1$ and $\deg f_3\leq \deg f_1$.

\nd {\rm (ii)} 
$f_i^{\w }\not\in k[f_j^{\w }]$ holds for $(i,j)=(1,2),(2,1)$. 

\nd {\rm (iii)} 
We have $f_3^{\w }\not\in k[f_1^{\w }]$.

\nd{\rm (iv)} 
If $f_1^{\w }\in k[f_3^{\w }]$, 
then $s=3$, 
$f_1^{\w }\approx (f_3^{\w })^2$, 
and $\deg f_3=(3/2)\delta $.

\end{claim}
\begin{proof}[Proof for Case $=$] 
Since $\deg f_1=\deg g_1=s\delta $ ($*$) and $s\ge 3$, 
(i) follows from $\deg f_2=2\delta $ and (SU4). 
(ii) follows from $\clubsuit $. 
For (iii), 
suppose that $f_3^{\w }\in k[f_1^{\w }]$. 
Then, 
since $\deg f_3\leq \deg f_1$ by (i),  
we have $f_3^{\w }\approx f_1^{\w }$, 
so $\deg f_3=\deg f_1=\deg g_1$ by ($*$). 
This contradicts Claim~\ref{claim:g1} (iv). 
By~(iii), 
$f_1^{\w }\in k[f_3^{\w }]$ implies 
that $f_1^{\w }\approx (f_3^{\w })^l$ for some $l\ge 2$. 
Then, ($*$) and ($3^{\rm u}$) give 
$s\delta =\deg f_1=\deg f_3^l>l(s-2)\delta $. 
Since $s\ge 3$, 
this yields $(s,l)=(3,2)$ and $\deg f_3=(s/l)\delta =(3/2)\delta $, 
proving (iv). 
\end{proof}

\begin{proof}[Proof of Lemma~{\rm \ref{prop:equivalence} (i)}]
First, assume that 
$(F,G)\in (\Aut _k\kx )^2$ satisfies 
$g_3\in f_3+k[g_1,g_2]$, 
(SU2)--(SU6), 
and the following conditions: 

(a) 
There exists $g_1'\in M:=(g_1+k[g_2])\cap (f_1+kf_3^2+kf_3)$ 
such that $(g_1')^{\w }=g_1^{\w }$;

(b) There exists $e\in k$ such that $g_2':=g_2-e\in f_2+kf_3$.

\nd 
Then, 
we can find $E_1\in \E _1$ and $E_2\in \E _2$ 
such that $(g_1',g_2',g_3)=GE_1E_2$ 
and $\deg G=\deg GE_1$. 
Since 
$k[g_1',g_2']=k[g_1,g_2]$, 
$(g_i')^{\w }=g_i^{\w }$ for $i=1,2$, 
and $dg_1'\wedge dg_2'=dg_1\wedge dg_2$, 
we easily see that $(F,(g_1',g_2',g_3))$ is an SU pair.

Now, 
let $(F,G)\in (\Aut _k\kx )^2$ be a weak SU pair. 
It satisfies $g_3\in f_3+k[g_1,g_2]$ by (SU$1'$); 
(SU2)--(SU6) by Claims~\ref{claim:A} (i) and \ref{claim:g2} (iii) 
and definition; 
and (b) by Claim~\ref{claim:g2}~(i). 
Hence, 
it suffices to verify (a). 
If $b\ne 0$, then $s=3$ by Table~2, 
so Claim~\ref{claim:g1}~(i) gives $\psi =lf_2+l'$ 
with $l,l'\in k$. 
Since $g_2=f_2+bf_3+e$, 
this yields 
$g_1':=g_1-(l(g_2-e)+l')=g_1-(\psi +lbf_3)=f_1+af_3^2+(c-lb)f_3\in M$. 
If $b=0$, then we have $\psi \in k[f_2]=k[g_2]$, 
and so $g_1':=g_1-\psi =f_1+af_3^2+cf_3\in M$. 
Since $\max \{ \deg g_2,\deg \psi \} <\deg g_1$, 
the equality $(g_1')^{\w }=g_1^{\w }$ holds in both cases. 
\end{proof}

\subsection{Differential forms}\label{subsect:dodf}

To establish further properties of the weak SU pairs, 
we need the following claim. 
Recall that $s\ge 7$ and $m_G^F=1$ 
if $(F,G)$ is of ps type.

\begin{claim}\label{claim:G}
\nd{\rm (i)} 
If $(F,G)$ is of non-ps type, 
then we have the following$:$ 
\begin{align}
&\left\{
\begin{array}{ll}
\gamma _2=\deg g_1-\deg g_2+\gamma _1=(s-2)\delta +\gamma _1
& \text{ if \ }m_G^F=1 \\
(1/2)\delta +\gamma _1<\gamma _2<(3/2)\delta +\gamma _1
& \text{ if \ }(s,m_G^F)=(3,2) \\
(6/7)\delta+\gamma _1<\gamma _2<(8/7)\delta +\gamma _1
&  \text{ if \ pm7 type. } 
\end{array}
\right.
\label{eq:pfSUred5.5 p}  \\
&\gamma _3 =\left\{
\begin{array}{cl}
\deg f_3+\gamma _1
& \text{ if }a\neq 0 \\
\gamma _2 & \text{ if }a=0\text{ and }b\neq 0 \\
\gamma _1 & \text{ if }a=b=0\text{ and }c\neq 0 \\
\deg dg_1\wedge dg_2 & \text{ if }a=b=c=0.  
\end{array}
\right. \label{eq:P12-1}\\
&\deg f_3+\gamma _1>
\gamma _2
>\gamma _1>\deg dg_1\wedge dg_2. 
\label{eq:gamma order}
\end{align}

\nd \nd{\rm (ii)} 
If $(F,G)$ is of ps type, 
then we have $a=b=0$, 
\begin{equation}\label{eq:P12-1 ps}
\gamma _3 =\left\{
\begin{array}{cl}
\gamma _1 & \text{ if }
c\neq 0 \\
\deg dg_1\wedge dg_2 & \text{ if }
c=0  
\end{array},
\right. 
\quad \text{and}\quad 
\gamma _1\ge \deg g_1+\deg dg_1\wedge dg_2. 
\end{equation}
\end{claim}
\begin{proof}
We prove the claim by applying 
Lemmas~\ref{note:cofactor} and \ref{lem:determinant} 
with $H:=F$ and $(h_1',h_2'):=(g_1,g_2)$. 
By Claims~\ref{claim:g2} (i) and \ref{claim:g1} (i), 
the conditions 
(\ref{eq:k1k2!}) and (d) are satisfied. 
We further verify the following: 
(1) (a) and $\deg dh_2\wedge dh_3>\deg dh_1'\wedge dh_2'$ 
always hold. 
(2)~If $(F,G)$ is of non-ps type, 
then (b) and (c) hold.

For (1), 
by ($4^{\rm s}$), 
it suffices to verify that 
$\gamma _1-\deg dg_1\wedge dg_2>\ep (g_1)$. 
If $(F,G)$ is of non-pm7 type, 
then Claim~\ref{claim:D} (i) gives 
$\gamma _1-\deg dg_1\wedge dg_2\ge \deg g_1>\ep (g_1)$. 
Otherwise, 
Claims~\ref{claim:D} (ii) and \ref{claim:pm7} (iv) yield 
$\gamma _1-\deg dg_1\wedge dg_2>(95/56-1/8)\delta >(1/7)\delta >\ep (g_1)$.

For (2), 
note that 
$\deg f_3>(s-2)\delta +\ep (g_2)-\ep (g_1)=\deg dg_1-\deg dg_2$ 
by Claim~\ref{claim:A} (iii) and ($3^{\rm s}$). 
Since $(F,G)$ is of non-ps type, 
Claim~\ref{claim:pm7} (ii)--(iv) give 
$\deg dg_1-\deg dg_2>(s-3)\delta \ge 0$. 
By Claim~\ref{claim:g1}~(i), 
we have $\Psi (x)\in \sum _{i=0}^{(s-1)/2}kx^i$, 
and so $\deg \Psi ^{(1)}(f_2)\le \frac{s-3}{2}\deg f_2=(s-3)\delta $. 
These yield (b) and (c).

(i) 
From (1), (2), 
and Lemma~\ref{lem:determinant} (i) (1) and (ii), 
we obtain 
$\gamma _2=\deg dg_1-\deg dg_2+\gamma _1$ ($\dag $), 
(\ref{eq:P12-1}), and (\ref{eq:gamma order}). 
Combining ($\dag $) with 
Claim~\ref{claim:pm7}~(ii)--(iv) yields~(\ref{eq:pfSUred5.5 p}).

(ii) 
By Table 2, 
the ps type satisfies $a=b=0$, 
and so $dg_2=df_2$. 
Hence, 
(1) and Lemma~\ref{note:cofactor}~(ii) yield 
the first part of (\ref{eq:P12-1 ps}). 
The last part is due to Claim~\ref{claim:D} (i). 
\end{proof}

\begin{proof}[Proof of Lemma~{\rm \ref{prop:equivalence} (ii)}]
Since $\deg f_2=\deg g_2$ by Claim~\ref{claim:g2} (iii), 
and $\deg f_3>\deg g_3$ by (SU5), 
it suffices to verify that 
$\deg f_1+\deg f_3>\deg g_1+\deg g_3$ in Case $<$, 
where $(F,G)$ is of non-ps and non-pm7 type and $m_G^F=1$ 
(see Table 2). 
It follows that 
\begin{equation}\label{eq:pfSUreduction5}
\begin{aligned}
&\deg f_1+\deg f_3\ge \gamma _2
=\deg g_1-\deg g_2+\gamma _1 
&& (\text{by Remark~\ref{rem:independence} (iii), 
(\ref{eq:pfSUred5.5 p})}) \\
&\quad 
\ge 2\deg g_1-\deg g_2+\deg dg_1\wedge dg_2 
&& (\text{by Claim~\ref{claim:D} (i)}). 
\end{aligned}
\end{equation}
By (SU6), 
the right-hand side of (\ref{eq:pfSUreduction5}) 
is greater than $\deg g_1+\deg g_3$. 
\end{proof}

In Case $<$, 
we have $(\deg g_1,\deg g_2,\deg f_3)=(3\delta ,2\delta ,(3/2)\delta )$. 
Substituting these values into 
(\ref{eq:pfSUreduction5}) yields the following claim.

\begin{claim}\label{claim:P5}
In Case $<$, we have 
$\deg f_1\ge (5/2)\delta +\deg dg_1\wedge dg_2
>(5/3)\deg f_3$. 
\end{claim}

Case $<$ of Claim~\ref{claim:H} 
now follows from the inequalities 
$\deg f_3=(3/2)\delta 
<\deg f_2=2\delta <\deg f_1<\deg g_1=3\delta =\deg f_3^2<\deg f_2^2$, 
where Claim~\ref{claim:P5} gives $2\delta <\deg f_1$.

\begin{example}\label{ex:WSU sigma=id}
Let $F\in \Aut _k\kx $ 
and $\sigma \in \sym _3$, 
and set $\delta :=(1/2)\deg f_2$. 
If one of the following holds, 
then $(F_{\sigma },G)$ is not a weak SU pair 
for any $G\in \Aut _k\kx $:

\nd $(1)$ 
$(7/3)\delta <\deg f_1<3\delta $ 
and $\deg f_3=(4/3)\delta $;

\nd $(2)$ 
$\delta <\deg f_1<3\delta $, 
$\deg df_1\wedge df_2-\deg df_2\wedge df_3=\deg f_3=(3/2)\delta $, 
and $\sigma \ne \id $. 

\end{example}
\begin{proof}
Suppose the contrary, 
that $(F_{\sigma },G)$ is a weak SU pair for some $G\in \Aut _k\kx $. 
Then, 
one of the two equalities in Table 2 $\clubsuit $ 
must hold for $F_\sigma $.

(1) 
The assumptions imply that $\deg f_3<\deg f_2<\deg f_1$. 
Hence, 
Claim~\ref{claim:H} (i) applied to $(F_{\sigma },G)$ 
yields $f_{\sigma (1)}=f_1$. 
It follows that 
$(7/3)\delta <\deg f_{\sigma (1)}<3\delta $ and 
$\{ \deg f_{\sigma (2)},\deg f_{\sigma (3)}\} 
=\{ 2\delta ,(4/3)\delta \} $, 
contradicting $\clubsuit $.

(2) 
Since $\deg f_3<\deg f_2$ 
and $\sigma \ne \id $ by assumption, 
we see from Claim~\ref{claim:H} (i) applied to $(F_{\sigma },G)$ 
that one of the following three cases must hold:

\nd 
\underline{$F_{\sigma }=(f_1,f_3,f_2)$}: 
Since $3\deg f_3=(9/2)\delta \ne 8\delta =4\deg f_2$ by assumption, 
it follows from $\clubsuit $ that 
$2\deg f_1=s'\deg f_3=(3s'/2)\delta $ for some odd $s'\ge 3$. 
Since $\deg f_1<3\delta $, we obtain $s'=3$. 
Hence, $(F_\sigma ,G)$ is of non-ps type. 
Thus, 
by (\ref{eq:pfSUred5.5 p}), 
$\gamma _2^{F_\sigma }-\gamma _1^{F_\sigma }$ 
is equal to $(s'-2)\delta '=\delta '$, 
or less than $(3/2)\delta '$ or $(8/7)\delta '$, 
where $\delta ':=(1/2)\deg f_3$. 
This contradicts 
$\gamma _2^{F_\sigma }-\gamma _1^{F_\sigma }
=\deg df_1\wedge df_2-\deg df_2\wedge df_3=\deg f_3=2\delta '$.

\nd 
\underline{$F_{\sigma }=(f_2,f_3,f_1)$}: 
Since $2\deg f_2\ne s'\deg f_3$ for all $s'\ge 3$, 
we see from $\clubsuit $ that $(F_{\sigma },G)$ is in Case $<$. 
Then, 
the case $a\ne 0$ of (\ref{eq:P12-1}) 
combined with (\ref{eq:gamma order}) yields 
$\deg df_2\wedge df_3=\gamma _3^{F_\sigma }
>\gamma _2^{F_\sigma }=\deg df_1\wedge df_2
=\deg df_2\wedge df_3+(3/2)\delta $, 
a contradiction.

\nd 
\underline{$F_{\sigma }=(f_2,f_1,f_3)$}: 
Since $\deg f_2<(5/3)\deg f_3$, 
Claim~\ref{claim:P5} implies that 
$(F_{\sigma },G)$ is in Case $=$. 
Hence, by $\clubsuit $, 
$2\deg f_2=s'\deg f_1$ 
holds for some odd $s'\ge 3$. 
Since $\deg f_2=2\delta $ and $\deg f_1>\delta $, 
we get $s'=3$ and $\deg f_1=(2/s')\deg f_2=(4/3)\delta <\deg f_3$. 
This implies that $(F_{\sigma },G)$ is of non-ps type 
and corresponds to the 
``$2\delta <\deg f_3$" case in Table 2. 
Then, 
the case $a=b=0$ of (\ref{eq:P12-1}) 
combined with (\ref{eq:gamma order}) yields 
$\deg df_1\wedge df_2=\gamma _3^{F_\sigma }\le 
\gamma _1^{F_\sigma }
<\gamma _2^{F_\sigma }=\deg df_2\wedge df_3$, 
a contradiction as before. 
\end{proof}

\begin{rem}\label{rem:diff forms}\rm 
Let $(F,G)$ be a weak SU pair.

\nd (i) 
The implications 
``proper $\Rightarrow $ $(a,b,c)\ne (0,0,0)$ 
$\Rightarrow $ $\gamma _3\ge \gamma _1$" hold, 
with the first step following from Definition~\ref{def:wsup} (ii), 
and the second from 
(\ref{eq:P12-1}), 
(\ref{eq:gamma order}), 
and (\ref{eq:P12-1 ps}).

\nd (ii) 
If $(F,G)$ is of pm7 type, 
then we have $a=b=0$ by Table 2. 
Hence, 
the~implications 
``proper 
$\Rightarrow c\ne 0$ 
$\Rightarrow $ 
$\gamma _3=\gamma _1>(95/56)\delta >\delta $" 
hold by (\ref{eq:P12-1}) and Claim~\ref{claim:D}~(ii).

\nd (iii) 
Assume that $(F,G)$ is of non-ps and non-pm7 type.

\nd {\rm (1)} 
We have 
$\gamma _2>\gamma _1>\deg g_1\ge \deg f_1\ge \max \{ \deg f_2,\deg f_3\} $ 
by (\ref{eq:gamma order}), 
Claim~\ref{claim:D} (i), 
(SU$2'$), 
and Claim~\ref{claim:H} (i).

\nd {\rm (2)} 
If $\deg f_i\ge \gamma _j$ holds for some $i,j\in \{ 1,2,3\} $, 
then it follows from (1) that 
$j=3$ and $\gamma _2>\gamma _1>\gamma _3$. 
This implies that $(F,G)$ is not proper by~(i). 

\end{rem}

\begin{definition}\label{def:pm7.1}\rm 
Let $(F,G)$ be a weak SU pair of pm 7 type. 
We say that $(F,G)$ is of {\it pm7.1 type} if 
$c\ne 0$ 
and $\mathcal{P}(f_2,f_1,m)\ne \emptyset $ 
for some $m\ge 1$. 
\end{definition}

\begin{claim}\label{claim:pm7.1}
Assume that $(F,G)$ is of pm7.1 type. 

\nd{\rm (i)} 
We have $\deg f_3\ne \deg g_1$. 
Hence, $2\delta <\deg f_3<3\delta $ holds by Table $2$.

\nd{\rm (ii)} 
We have $(\ref{eq:pfSUreduction1.5})$. 
Hence, 
$\gamma _1\ge \deg g_1+\deg dg_1\wedge dg_2>s\delta $ 
holds by Claim~{\rm \ref{claim:D} (i)}. 

\end{claim}

\begin{proof}
Note that 
$\gamma _3>\delta $ (a) by Remark~\ref{rem:diff forms} (ii), 
since $c\ne 0$ by definition; 
$\deg f_1=\deg g_1$ (b) 
and $g:=g_1-\psi =f_1+cf_3\in f_1+k^*f_3$ (c), 
since $a=0$; 
$f_2=g_2-e$ and $\psi \in k[f_2]=k[g_2]$ (d), 
since $b=0$; 
and $\deg \psi <\deg g_1$ (e) by~Claim~\ref{claim:g1}~(i).

(i) 
Suppose that $\deg f_3=\deg g_1$. 
Then, together with (b) and $c\ne 0$, 
it follows from Claim~\ref{claim:g1}~(iv) that 
$f_1^{\w }$ and $g_1^{\w }$ are algebraically independent over $k$; 
hence, so are $f_1^{\w }$ and $g_2^{\w }=(f_2+e)^{\w }=f_2^{\w }$ 
by Claim~\ref{claim:A} (i) and (d). 
By Remark~\ref{rem:(1.3.7)}~(i), 
this contradicts the condition that 
$\mathcal{P}(f_2,f_1,m)\ne \emptyset $ for some $m\ge 1$.

(ii) 
By (\ref{eq:m(F,G)}) and by (c)--(e) and Remark~\ref{rem:PP} (ii), 
we have 
$\phi _3\in \mathcal{P}(g_2,g_1,7)=\mathcal{P}(f_2,g,7)$ 
with $g\in f_1+k^*f_3$. 
Hence, 
it follows from 
(i), (a), $\mathcal{P}(f_2,f_1,m)\ne \emptyset $, 
and Theorem~\ref{thm:modification}~(ii) that 
$\deg df_2\wedge d\phi _3>3\delta +\deg dg\wedge df_2$, 
since (i) implies $(\sharp )$, and hence $(\flat )$. 
This yields (\ref{eq:pfSUreduction1.5}), 
since $dg\wedge df_2=dg_1\wedge dg_2$ by (d). 
\end{proof}

\subsection{Proof of Lemma~{\rm \ref{prop:c1-c6}}}\label{subsect:pf of lem 3}
We remark the following:

\nd 
($1^\circ $) 
$\deg f_3=\deg f_1$ implies 
$\deg f_1=\deg g_1$ by Claim~\ref{claim:P5}, 
and $a=b=0$ by Table~2.

\nd 
($2^\circ $) 
By Claims~\ref{claim:H} (iii) and \ref{claim:g2} (iii), 
$f_3^{\w }\not\approx f_j^{\w }$ holds for $j=1,2$.

If $\w $ is independent, 
then ($2^\circ $) implies that $\deg f_3\ne \deg f_j$ for $j=1,2$ 
by (\ref{eq:approx}). 
In this case, 
Lemma~{\rm \ref{prop:c1-c6}} (ii) (2) is vacuously true.

\begin{claim}\label{claim:S3}
Assume that $(F,G)$ is proper. 
Let $\phi \in k[S_3]$ be such that $\deg \phi \le \deg f_3$. 
Then, we have $\deg ^{S_3}\phi \le \deg f_3$ $(\dag )$, 
which yields the following statements 
{\rm (i)}--{\rm (iii):}

\nd{\rm (i)} If $\deg f_3\ne \deg f_j$ for $j=1,2$, 
then $\phi \in k$, or 
$\phi \in k[f_2]$ and $g_2=f_2+e$.

\nd{\rm (ii)} If $\deg f_3=\deg f_1$, 
then $\phi \in kf_1+\rho $ 
for some $\rho \in k[f_2]$ with $\deg \rho <\deg g_1$.

\nd{\rm (iii)} If $\deg f_3=\deg f_2$, 
then $\phi \in kf_2+k$.

\end{claim}
\begin{proof}
Suppose that $\deg ^{S_3}\phi >\deg f_3$. 
Then, since $\deg \phi \le \deg f_3$ by assumption, 
we have $\deg \phi <\deg ^{S_3}\phi $, 
implying $\phi \in \mathcal{D}(f_1,f_2)$. 
Hence, 
Claim~\ref{claim:H}~(ii) and Corollary~\ref{cor:pSU0} (iii) 
yield $\deg \phi >\gamma _3$. 
For a contradiction, 
we prove $\gamma _3>\deg \phi $.

The assumption, Claim~\ref{claim:H}~(i), and (SU$2'$) yield 
$\deg \phi \le \deg f_3\le \deg f_1\le \deg g_1=s\delta $~($*$). 
Since $\phi \in \mathcal{D}(f_1,f_2)$, 
this gives $\phi \in \mathcal{P}(f_2,f_1,m)$ with $m\ge 1$. 
Hence, 
by Remark~\ref{rem:diff forms}~(ii) 
and the assumption that $(F,G)$ is proper, 
$(F,G)$ cannot be of pm7 and non-pm7.1 type. 
Thus, we get $\gamma _1>s\delta \ge \deg \phi $ 
by Claims~\ref{claim:D}~(i) and \ref{claim:pm7.1} (ii) and ($*$). 
By Remark~\ref{rem:diff forms}~(i), 
the properness of $(F,G)$ also implies that $\gamma _3\ge \gamma  _1$, 
yielding $\gamma _3>\deg \phi $.

(i) 
Since $\deg ^{S_3}\phi \le \deg f_3\le \deg f_1$ 
by ($\dag $) and Claim~\ref{claim:H}~(i), 
and since $\deg f_3\ne \deg f_1$ by assumption, 
we have $\deg ^{S_3}\phi <\deg f_1$. 
By Remark~\ref{rem:deg^fg}, 
this implies that $\phi \in k[f_2]$. 
Hence, if $\phi \not\in k$, 
then $\deg f_2\le \deg ^{S_3}\phi $ holds. 
This together with ($\dag $) and $\deg f_3\ne \deg f_2$ gives 
$\deg f_2<\deg f_3$, 
yielding $b=0$ via Claim~\ref{claim:g2}~(ii).

(ii) 
Since $\deg ^{S_3}\phi \le \deg f_3=\deg f_1$, 
we have $\phi \in kf_1+\rho $, 
where $\rho \in k[f_2]$ with 
$\deg \rho \le \deg f_1\le \deg g_1=s\delta $. 
Since $\deg f_2=2\delta $ and $s$ is odd, 
we get $\deg \rho <\deg g_1$.

(iii) 
Since $\deg ^{S_3}\phi \le \deg f_3=\deg f_2<\deg f_1$ 
by Claim~\ref{claim:H}~(i), 
this is clear. 
\end{proof}

\begin{claim}\label{claim:S2}
For all $\phi \in k[S_2]\sm k[f_3]$, 
we have $\deg \phi >\deg f_2$. 
\end{claim}
\begin{proof}
By assumption and Remark~\ref{rem:deg^fg}, 
we have $\deg ^{S_2}\phi \ge \deg f_1$. 
Now, 
suppose that $\deg \phi \le \deg f_2$. 
Then, 
since $\deg f_2<\deg f_1$ by Claim~\ref{claim:H} (i), 
we get $\deg \phi <\deg f_1\le \deg ^{S_2}\phi $. 
This implies that 
$\phi \in \mathcal{P}(f_3,f_1,m)$ for some $m\ge 1$ (a).

First, 
assume that $f_1^{\w }\not\in k[f_3^{\w }]$. 
Then, since $\deg f_3\le \deg f_1$ by Claim~\ref{claim:H} (i), 
it follows from Theorem~\ref{ex:ineq1} (i) and (iii) 
that $(f_1^{\w })^2\approx (f_3^{\w })^t$ for some odd $t\ge 3$ (b) 
and $\deg \phi >\gamma _2$ (c). 
Since $\deg f_1\le s\delta $, 
(b) yields 
$\deg f_3\le (2s/t)\delta $ (d). 
This implies that $(F,G)$ is of non-pm7 type, 
for otherwise $s=3$ and $\deg f_3>2\delta $ 
by Claim~\ref{claim:pm7} (iv), 
contradicting (d) and $t\ge 3$. 
From ($3^{\rm u}$) and (d), we also have 
\begin{equation}\label{eq:1/t+1/s}
s-2<\frac{2s}{t}\ \Rightarrow \ 
\frac{1}{2}<\frac{1}{s}+\frac{1}{t}
\ \Rightarrow \ 
(s,t)=(3,5),(5,3),(3,3). 
\end{equation}
Since $p\nmid s$, 
$(F,G)$ is of non-ps type. 
Thus, 
we obtain $\gamma _2>\deg f_2$ by Remark~\ref{rem:diff forms}~(iii) (1), 
and so $\deg \phi >\deg f_2$ by (c). 
This contradicts the supposition.

Next, assume that $f_1^{\w }\in k[f_3^{\w }]$. 
Then, Claim~\ref{claim:H} (iv) gives 
$s=3$ (e), 
$\deg f_3=(3/2)\delta $ (f), 
and $f_1^{\w }\approx (f_3^{\w })^2$ (g). 
Since $\deg f_3<2\delta $ by (f), 
we see from Claim~\ref{claim:pm7} (iii)--(v) that 
$(F,G)$ is of non-ps and non-pm7 type and $m_G^F=1$. 
Hence, from 
(\ref{eq:pfSUred5.5 p}), Claim~\ref{claim:D} (i), and (e), 
we get $\gamma _2=(s-2)\delta +\gamma _1>(2s-2)\delta =4\delta $. 
Thus, in view of (a) and (g), 
applying Corollary~\ref{cor:pSU0} (i) to $\phi $ 
yields $\deg \phi \ge \min \{ M_0,pM_1\} $, 
where $M_0=
-\deg f_3 +\gamma _2
>\frac{5}{2}\delta $ 
and $M_1=M_0-\frac{1}{p-1}((2+1)\deg f_3-\gamma _2)
>(\frac{5}{2}-\frac{1}{2(p-1)})\delta >\frac{2}{p}\delta $. 
This contradicts the supposition 
that $\deg \phi \le \deg f_2=2\delta $. 
\end{proof}

\begin{rem}\label{exc:c1-c6 0}\rm
Let $(F,G)\in (\Aut _k\kx )^2$ be an arbitrary pair.

\nd (i) 
If $(F,G)$ satisfies (SU$1'$), 
then so does $((f_1+q_1,f_2+q_2,f_3),G)$ 
for all $(q_1,q_2)\in k[f_2,f_3]\times k[f_3]$ 
(e.g., 
$(FE,G)$ with 
$E\in \E _1 \cup \{ E \in \E _2\mid E(x_2)\in x_2+k[x_3]\} $).

\nd (ii) 
Assume that $(F,G)$ satisfies 
(SU$3'$) and (SU4)--(SU6), 
and let $P\in \Aut _k\kx $ be such that 
[a] $\deg p_i\leq \deg g_i$ for $i=1,2$ and 
[b] $p_3^{\w }\in k^*f_3^{\w }+k[g_1^{\w },g_2^{\w }]$. 
Then, we can verify that $(P,GH)$ satisfies 
(SU$2'$), (SU$3'$), and (SU4)--(SU6) for all $H\in \D _3(k)$. 
In fact, 
since $f_3^{\w }\not\in k[g_1^{\w },g_2^{\w }]$ by assumption, 
[b] implies that $\deg p_3=\deg f_3$.

We remark that [a] (resp.\ [b]) holds if 
[a$'$] 
$(F,G)$ satisfies (SU$2'$) and $\deg p_i\le \deg f_i$ for $i=1,2$ 
(resp.\ [b$'$] $p_3=f_3$). 
\end{rem}

\begin{proof}[Proof of Lemma~{\rm \ref{prop:c1-c6}}]
(i) 
Applying Remark~\ref{exc:c1-c6 0} (ii) with $(F,G,P):=(F,G,FE)$ 
verifies that 
$(FE,G)$ satisfies 
(SU$2'$), (SU$3'$), and (SU4)--(SU6), 
since [a$'$] and [b$'$] hold 
by the assumption that $E\in \E _1\cup \E _2$ 
and $\deg FE\le \deg F$. 
For (SU$1'$), 
by Remark~\ref{exc:c1-c6 0}~(i), 
it suffices to show that 
$E\in \E _1\cup \{ E \in \E _2\mid E(x_2)\in x_2+k[x_3]\} $. 
If not, 
then we have $FE=(f_1,f_2+\phi ,f_3)$ 
for some $\phi \in k[S_2]\sm k[f_3]$. 
Since $\deg \phi >\deg f_2$ by Claim~\ref{claim:S2}, 
this contradicts $\deg FE\le \deg F$.

(ii) 
We can write $FE=(f_1,f_2,f_3')$, 
where $f_3':=f_3+\phi $, 
and $\phi \in k[S_3]$ with $\deg \phi \le \deg f_3$ (a). 
Since $(F,G)$ is proper by assumption, 
Claim~\ref{claim:S3} applies.

(1) 
Claim~\ref{claim:S3} (i) implies that 
$k[f_2,f_3]=k[f_2,f_3']$ (b), 
$k[f_3]=k[f_3']$ or $g_2-f_2\in k$ (c), and 
$\phi \in k[g_2]$ (d). 
Since $(F,G)$ satisfies (SU$1'$), 
(b)--(d) imply that so does $(FE,G)$. 
We verify the remaining conditions by applying 
Remark~\ref{exc:c1-c6 0}~(ii) with $(F,G,P):=(F,G,FE)$. 
Since [a$'$] holds, 
we only check [b]. 
Since $f_3^{\w }\not\in k[g_2^{\w }]$ 
by (SU4) applied to $(F,G)$, 
and $\phi ^{\w }\in k[g_2^{\w }]$ by (d), 
we have $f_3^{\w }+\phi ^{\w }\ne 0$. 
Hence, 
by (a), we obtain 
$(f_3')^{\w }=(f_3+\phi )^{\w }\in \{ f_3^{\w }+\phi ^{\w },f_3^{\w }\} 
\subset f_3^{\w }+k[g_2^{\w }]$, 
proving~[b].

(2) Take $j\in \{ 1,2\} $ with $\deg f_3=\deg f_j$. 
Then, 
by ($1^\circ $) and Claims~\ref{claim:g2} (i), 
\ref{claim:g1} (i), and \ref{claim:S3} (ii) and (iii), 
we can write $g_j=f_j+t_1f_3+\psi _1$ and $f_3'=f_3+t_2f_j+\psi _2$, 
where $t_1,t_2\in k$, 
and where 
$\psi _i\in k[f_2]=k[g_2]$ with $\deg \psi _i<\deg g_1=\deg f_1$ if $j=1$, 
and $\psi _i\in k$ if $j=2$, for $i=1,2$ (e). 
We define 
\begin{align*}
&(F',G'):=
\left\{ \!\!
\begin{array}{ll}
((f_1,f_2,f_3'),((1-t_1t_2)^{-1}g_1,g_2,(1-t_1t_2)g_3)) & 
\text{if $j=1$ and $t_1t_2\ne 1$} \\ 
((f_1,f_2,f_3'),(g_1,(1-t_1t_2)^{-1}g_2,(1-t_1t_2)g_3)) & 
\text{if $j=2$ and $t_1t_2\ne 1$} \\ 
((f_3',f_2,f_1),(t_1^{-1}g_1,g_2,-t_1g_3)) & 
\text{if $j=1$ and $t_1t_2=1$} \\ 
((f_1,f_3',f_2),(g_1,t_1^{-1}g_2,-t_1g_3)) & 
\text{if $j=2$ and $t_1t_2=1$.} 
\end{array}\!\! \right.
\end{align*}
Since $(F',G')$ has the form $((FE)_{\tau },GH)$ 
for some $\tau \in \sym _3$ and $H\in {\rm D}_3(k)$, 
it suffices to verify that $(F',G')$ is a weak SU pair.

By (SU$1'$) applied to $(F,G)$, 
we have 
$g_1-f_1\in k[f_2,f_3]$ (f) 
and $g_3-f_3\in k[g_1,g_2]$ (g). 
If $j=2$, then 
$g_1-f_1\in 
k[f_2,f_3']$ 
holds by (f), 
since $f_3'\in f_3+kf_2+k$. 
If $j=1$, then $g_2-f_2=e\in k[f_3']$ holds, 
since $b=0$ by ($1^\circ $). 
The following equalities, 
combined with (e) and (g), 
complete the proof that $(F',G')$ satisfies (SU$1'$): 
\begin{align*}
&\left\{ \!\!\begin{array}{l}
g_j-(1-t_1t_2)f_j
=t_1f_3'+\psi _1-t_1\psi _2\\
(1-t_1t_2)g_3-f_3'
=(1-t_1t_2)(g_3-f_3)-t_2g_j+t_2\psi _1-\psi _2 
\end{array}\right.\quad \text{if $t_1t_2\ne 1$},\\
&\left\{  \!\!\begin{array}{l}
t_1^{-1}g_j-f_3'
=t_1^{-1}\psi _1-\psi _2 \\
-t_1g_3-f_j
=-t_1(g_3-f_3)-g_j+\psi _1
\end{array}\right.\quad \text{if $t_1t_2=1$}. 
\end{align*}

For the remaining conditions, 
it suffices to verify [a] and [b] in Remark~\ref{exc:c1-c6 0}~(ii) 
with $(F,G,P):=(F,G,F')$. 
Since 
$\deg \psi _i<\deg f_j=\deg f_3$ for $i=1,2$ by (e), 
and $f_j^{\w }\not\approx f_3^{\w }$ by ($2^\circ $), 
we have 
$g_j^{\w }=f_j^{\w }+t_1f_3^{\w }$ and 
$(f_3')^{\w }
=f_3^{\w }+t_2f_j^{\w }
=(1-t_1t_2)f_3^{\w }+t_2g_j^{\w }$. 
This implies that 
$(f_3')^{\w }\in k^*f_3^{\w }+kg_j^{\w }$ 
if $t_1t_2\ne 1$, and 
$\deg f_3'=\deg g_j$ and $f_j^{\w }\in k^*f_3^{\w }+g_j^{\w }$ 
if $t_1t_2=1$. 
By (SU$2'$) applied to $(F,G)$, 
we also have $\deg f_l\le \deg g_l$ for $l=1,2$. 
These verify [a] and [b]. 
\end{proof}

\subsection{Proof of Lemma~\ref{prop:structure2}}
\label{sect:proof of key lemma}

Throughout this subsection, 
let $F$ be an arbitrary element of $\Aut _k\kx $ satisfying 
$\llparenthesis 1 \rrparenthesis$--$\llparenthesis 4 \rrparenthesis$. 
We remark that 
$\llparenthesis 1 \rrparenthesis$ and $\llparenthesis 2 \rrparenthesis$ 
imply the following: 
$\llparenthesis 5 \rrparenthesis$ 
$(\deg f_1,\deg f_2)=(s\delta ,2\delta )$; 
$\llparenthesis 6 \rrparenthesis$ 
$\delta \le (s-2)\delta <\deg f_3$; 
$\llparenthesis 7 \rrparenthesis$ 
$\deg f_1<\deg f_2f_3$ 
by $\llparenthesis 5 \rrparenthesis$ and $\llparenthesis 6 \rrparenthesis$; 
$\llparenthesis 8 \rrparenthesis$ $\deg df_1\wedge df_2<\deg f_3$; and 
$\llparenthesis 9 \rrparenthesis$ $\deg f_3<\deg f_1$.

\begin{lem}\label{note:prop:key2}
\nd{\rm (i)} The following 
$\llparenthesis 10 \rrparenthesis$--$\llparenthesis 12 \rrparenthesis$ 
hold$:$ 

$\llparenthesis 10 \rrparenthesis$ 
$f_2^{\w }\not\in k[f_1^{\w },f_3^{\w }];$ 
\quad 
$\llparenthesis 11 \rrparenthesis$ 
$f_1^{\w }\not\in k[f_2^{\w },f_3^{\w }]\sm k[f_3^{\w }];$ 

$\llparenthesis 12 \rrparenthesis$ 
If $f_1^{\w }\in k[f_3^{\w }]$, 
then $s=3$ and $\deg f_3=(3/2)\delta $.

\nd{\rm (ii)} 
If $2\deg f_1=t\deg f_3$ for some odd $t\ge 3$, 
then we have 
$t=5$ or $\deg f_2\le \deg f_3$.

\nd{\rm (iii)} 
If $(F_{\sigma },G)$ is a weak SU pair 
for some $\sigma \in \mathfrak{S}_3$ and $G\in \Aut _k\kx $, 
then we have $\sigma =\id $, 
and $(F,G)$ is not proper.

\end{lem}
\begin{proof}
(i) 
Since $\deg f_2<\deg f_1$ by $\llparenthesis 5 \rrparenthesis$, 
$f_2^{\w }\not\approx f_3^{\w }$ by $\llparenthesis 4 \rrparenthesis$, 
and 
$\deg f_2<\deg f_3^2$ 
by $\llparenthesis 5 \rrparenthesis$ and $\llparenthesis 6 \rrparenthesis$, 
we get $\llparenthesis 10 \rrparenthesis$. 
By $\llparenthesis 1 \rrparenthesis$, 
we have $f_1^{\w }\not\in k[f_2^{\w }]$. 
Hence, 
$\llparenthesis 7 \rrparenthesis$ 
yields $\llparenthesis 11 \rrparenthesis$. 
In view of $\llparenthesis 9 \rrparenthesis$ 
and $\llparenthesis 6 \rrparenthesis$, 
we can verify $\llparenthesis 12 \rrparenthesis$ 
similarly to Case $=$ of Claim~\ref{claim:H}~(iv).

(ii) 
Since $(s-2)\delta <\deg f_3=(2/t)\deg f_1=(2s/t)\delta $ by 
$\llparenthesis 6 \rrparenthesis$ and $\llparenthesis 5 \rrparenthesis$, 
(\ref{eq:1/t+1/s}) holds. 
If $t\ne 5$, 
then $(s,t)\in \{ (5,3),(3,3)\} $, 
so $\deg f_3=(2s/t)\delta \ge 2\delta =\deg f_2$ 
by~$\llparenthesis 5 \rrparenthesis$.

(iii) 
Since $\deg f_1>\deg f_i$ for $i=2,3$ by 
$\llparenthesis 5 \rrparenthesis$ and $\llparenthesis 9 \rrparenthesis$, 
Claim~\ref{claim:H}~(i) gives $\sigma (1)=1$.

First, assume that 
$(F_\sigma ,G)$ is of non-ps and non-pm7 type. 
Then, by $\llparenthesis 8 \rrparenthesis$, 
applying Remark~\ref{rem:diff forms}~(iii) (2) to $(F_\sigma ,G)$ 
shows that 
$\gamma _2^{F_\sigma }>\gamma _1^{F_\sigma }>\gamma _3^{F_\sigma }
=\deg df_1\wedge df_2$ 
and $(F_{\sigma },G)$ is not proper. 
The former implies that $\sigma (3)=3$, and hence $\sigma =\id $.

Next, assume that $(F_\sigma ,G)$ is of ps or pm7 type. 
Suppose that 
$\sigma \ne \id $, i.e.,~$F_\sigma =(f_1,f_3,f_2)$. 
Then, 
Case $=$ of $\clubsuit $ defines $t$ as in (ii). 
If $t=5$, 
then $(F_\sigma ,G)$ is of non-ps and non-pm7 type. 
The same holds if $\deg f_2\le \deg f_3$, 
as it corresponds to the ``$\deg f_3\le 2\delta $" case 
in Table 2. 
This is a contradiction, proving $\sigma =\id $.

If $(F,G)$ is of ps or pm7.1 type, 
then Claim~\ref{claim:D} (i) or Claim~\ref{claim:pm7.1}~(ii) 
yields $\gamma _1>\deg g_1$. 
Since $\deg g_1\ge \deg f_3>\deg df_1\wedge df_2=\gamma _3$ 
by (SU4) and $\llparenthesis 8 \rrparenthesis$, 
it follows that $\gamma _1>\gamma _3$. 
If $(F,G)$ is of pm7 and non-pm7.1 type, 
then we must have $c=0$ due to 
$\llparenthesis 3 \rrparenthesis$. 
In either case, 
$(F,G)$ is not proper 
by parts (i) and (ii) of Remark~\ref{rem:diff forms}. 
\end{proof}

We now prove Lemma~\ref{prop:structure2}. 
Lemma~\ref{note:prop:key2} (iii) 
and Remark~\ref{exc:used for Claim B}~(i) directly imply~(i). 
Next, to prove (ii) and (iii), 
note that $\llparenthesis 2 \rrparenthesis$ and 
$\llparenthesis 5 \rrparenthesis$ imply 
the assumptions of Theorem~\ref{prop:Exc27,28}. 
For (ii), 
by Lemma~\ref{def:initial algebra}, 
we can choose $\phi \in k[S_2]$ with $\phi ^{\w }=f_2^{\w }$. 
Then, $\llparenthesis 10 \rrparenthesis$ gives 
$\phi ^{\w }\not\in k[f_1^{\w },f_3^{\w }]$, 
and so $\phi \in \mathcal{D}(f_1,f_3)$ 
by Remark~\ref{rem:(1.3.7)}~(ii). 
Hence, Theorem~\ref{prop:Exc27,28} (ii) 
yields $f_1^{\w }\approx (f_3^{\w })^2$.

Finally, for (iii), 
take $\phi \in k[S_1]$ with $\phi ^{\w }=f_1^{\w }$ (a). 
Then, 
$\llparenthesis 11 \rrparenthesis$ 
and parts (ii) and (i) of Remark~\ref{rem:(1.3.7)} yield 
the chain of implications 
\begin{equation}\label{eq:implication}
f_1^{\w }\not\in k[f_3^{\w }] 
\ \Rightarrow \ 
\phi ^{\w }=f_1^{\w }\not\in k[f_2^{\w },f_3^{\w }]
\ \Rightarrow \ 
\phi \in \mathcal{D}(f_1,f_3) 
\ \Rightarrow \ 
\text{(1$^\star $)}.
\end{equation}
By $\llparenthesis 1 \rrparenthesis$, 
$f_1^{\w }\in k[f_3^{\w }]$ also implies (1$^\star $), 
proving (1$^\star $). 
Next, set $f_1':=f_1-\phi $. 
Then, (a) gives $\deg f_1'<\deg f_1=s\delta $ (b). 
In view of Remark~\ref{rem:ch2} (iv), 
we may also assume that $(f_1')^{\w }\not\in k[S_1]^{\w }$ (c) 
by replacing $\phi $ if necessary. 
The following claim now completes the proof of (iii).

\begin{klaim}\label{claim:step1}\it 
{\rm (I)} If $\phi \in \mathcal{D}(f_1,f_3)$, 
then {\rm (2$^\star $)} holds for $(f_1',f_2,f_3)$, 
and $((f_1',f_2,f_3)_{\sigma },G)$ is not a weak SU pair 
for any $\sigma \in \sym _3$ and $G\in \Aut _k\kx $.

\nd {\rm (II)} 
If $\phi \not\in \mathcal{D}(f_1,f_3)$, 
or equivalently, $\deg ^{S_1}\phi =\deg \phi $, 
then there exists $f_1''\in f_1'+kf_2$ 
such that 
{\rm (2$^\star $)} and {\rm (3$^\star $)} 
hold for $(f_1'',f_2,f_3)$. 
\end{klaim}

\begin{proof}
(I) 
By $\llparenthesis 10 \rrparenthesis$ and 
$\llparenthesis 4 \rrparenthesis$, we have 
$f_2^{\w }\not\in k[f_3^{\w }]$ and $f_3^{\w }\not\in k[f_2^{\w }]$ (d). 
Hence, 
it follows from (a), $\phi \in \mathcal{D}(f_1,f_3)$, 
and Theorem~\ref{prop:Exc27,28} (i) that 
$(7/3)\delta <\deg f_1'<3\delta $ and $\deg f_3=(4/3)\delta $ (e). 
Therefore, the last part follows from 
Example~\ref{ex:WSU sigma=id}~(1).

For (2$^\star $), by (c), 
it suffices to verify that $f_u^{\w }\not\in k[f_1',f_t]^{\w }$ 
for $(t,u)=(2,3),(3,2)$ 
(see Remark~\ref{rem:sim}). 
Note that $\deg f_3<\deg f_2<\deg f_1'$ by (e), 
$(f_1')^{\w }\not\in k[f_t^{\w }]$ by (c), 
and $f_u^{\w }\not\in k[f_t^{\w }]$ by (d). 
Hence, 
the condition $\mathfrak{f}(1,t,u)$ 
defined before Lemma~\ref{exc:ineq1} 
holds for $(f_1',f_2,f_3)$. 
Thus, supposing $f_u^{\w }\in k[f_1',f_t]^{\w }$, 
Lemma~\ref{exc:ineq1} (i) (1) yields an odd $b\ge 3$ 
such that 
$(2/b)\deg f_1'=\deg f_t
\in \{ \deg f_2,\deg f_3\} =\{ 2\delta ,(4/3)\delta \} $. 
This contradicts $(7/3)\delta <\deg f_1'<3\delta $.

(II) 
We define $f_1''$ 
by applying Theorem~\ref{ex:total} (iv) to $f_1'$. 
To this end, 
we first note that 
$f_1^{\w }\in k[f_3^{\w }]$ by (\ref{eq:implication}), 
since $\phi \not\in \mathcal{D}(f_1,f_3)$ by assumption. 
Hence, 
from $\llparenthesis 12 \rrparenthesis$, 
we get 
$s=3$ and $\deg f_3=(3/2)\delta $ (f). 
Thus, 
$\llparenthesis 2 \rrparenthesis$ yields 
$\deg df_1\wedge df_2\le \deg f_3-(s-2)\delta =(1/2)\delta $. 
The assumption, 
(a), $\llparenthesis 5 \rrparenthesis$, and (e) give 
$\deg ^{S_1}\phi =\deg \phi =\deg f_1=s\delta =3\delta $. 
Since $(\deg f_2,\deg f_3)=(2\delta ,(3/2)\delta )$, 
this implies that 
$\phi \in af_3^2+kf_3+kf_2+k$ for some $a\in k^*$, 
yielding $f_1^{\w }=\phi ^{\w }=a(f_3^{\w })^2$ 
and $f_1'=f_1-\phi \in f_1-af_3^2+kf_3+kf_2+k$. 
By (c), 
$(f_1')^{\w }\not\approx f_i$ holds for $i=2,3$.

Now, take $f_1''\in f_1'+kf_2$ as in Theorem~\ref{ex:total} (iv). 
Then, 
(2$^\star $) holds for $F'':=(f_1'',f_2,f_3)$, 
since 
$f_2^{\w }\not\in k[f_1'',f_3]^{\w }$, 
$f_3^{\w }\not\in k[f_1'',f_2]^{\w }$, 
and $(f_1'')^{\w }=(f_1')^{\w }\not\in k[S_1]^{\w }$ by (c). 
Also note that Theorem~\ref{ex:total}~(i) 
still holds with $f_1'$ replaced by $f_1''$ (g).

For (3$^\star $), 
assume that $(F''_\sigma ,G)$ is a weak SU pair 
for some $\sigma \in \sym _3$ and $G\in \Aut _k\kx $. 
Then, 
(g), (f), and Example~\ref{ex:WSU sigma=id}~(2) 
imply that $\sigma =\id $, 
i.e., 
$(F'',G)$ is a weak SU pair. 
Then, 
since $F=(f_1''+(f_1-f_1''),f_2,f_3)$ 
and $f_1-f_1''\in k[f_2,f_3]$, 
applying Remark~\ref{exc:c1-c6 0}~(i) with $(F,G):=(F'',G)$ 
shows that $(F,G)$ satisfies (SU$1'$). 
We verify the remaining conditions by applying 
Remark~\ref{exc:c1-c6 0}~(ii) with $(F,G,P):=(F'',G,F)$. 
Since [b$'$] holds, 
we only check [a]. 
By (b), 
$2\deg f_1''=2\deg f_1'<6\delta \le t\deg f_2$ 
holds for all $t\ge 3$. 
Hence, we see from $\clubsuit $ that $(F'',G)$ is in Case $<$. 
Thus, 
Claim~\ref{claim:g1}~(iii) applied to $(F'',G)$ 
yields $\deg g_1=2\deg f_3$. 
This is equal to $\deg f_1$ 
by (f) and $\llparenthesis 5 \rrparenthesis$. 
(SU$2'$) applied to $(F'',G)$ also gives $\deg f_2\le \deg g_2$, proving~[a]. 
\end{proof}

This completes the proof of 
Lemma~\ref{prop:structure2} 
and thereby of Theorem~\ref{thm:mainSU}.

\subsection{Shestakov-Umirbaev pair of ps type}\setcounter{equation}{0}

Various kinds of weak SU pairs were discussed in this section, 
but they were only theoretically considered 
and it is not clear whether they actually exist. 
When $p=0$, 
the first example of an SU pair was given by 
Shestakov-Umirbaev~\cite[Example 1]{SU} for $s=3$, 
where $\w =\mathbb{1}$. 
Later, similar examples were constructed for $s=3,5,7$ 
by van den Essen--Makar-Limanov--Willems~\cite{EMW}, 
and for all odd $s\ge 3$ by us \cite{typeI}. 
These examples involve $G\in \Aut _k\kx $ 
such that $\kappa :=(g_1^{\w })^2/(g_2^{\w })^s\in k^*$ 
and $\deg (g_1^2-\kappa g_2^s)<\deg g_1$. 
When $p\mid s$, 
however, 
no such $G$ exists, 
because 
$\deg (g_1^2-\kappa g_2^s)\ge \deg d(g_1^2-\kappa g_2^s)
=\deg 2g_1dg_1>\deg g_1$. 
Yet, there does exist an SU pair of ps type 
as constructed below.

\begin{lem}\label{lem:typeI}
Let 
$G\in \Aut _k\kx $ and $\phi \in k[g_1,g_2]$, 
and set 
$\delta :=(1/2)\deg g_2$. 
If $(g_1^{\w })^2\approx (g_2^{\w })^s$, 
$\deg g_3\le (s-2)\delta $, 
and $(s-1)\delta <\deg \phi <s\delta $ 
for some odd $s\ge 3$, 
then $(F,G)$ is an SU pair, 
where $F:=(g_1-g_3-\phi ,g_2,g_3+\phi )$. 
\end{lem}

\begin{proof}
Since $\deg g_3<\deg \phi $, 
we have 
$\deg f_3=\deg \phi <s\delta =\deg g_1$ 
and $f_3^{\w }=\phi ^{\w }$. 
Since $\deg \phi \not\in \Z \delta $ and 
$\deg g_1,\deg g_2\in \Z \delta $, 
we also have $\phi ^{\w }\not\in k[g_1^{\w },g_2^{\w }]$, 
proving (SU4). 
The other conditions are verified directly. 
\end{proof}

Now, 
let $\w =\mathbb{1}$, 
and take $d,e\in \N $ with $d\ge 3e+1$ and $e\ge 2$. 
Set $g_1:=x_3^{pd}+q_1$, 
$g_2:=x_3^{2d}+q_2$, 
and $\phi :=g_1^2-g_2^p-2g_2^{(p+3)/2}$, 
where $q_2:=x_3^{d-e}+x_1x_3^e+x_2$ and 
$$
q_1:=x_3^{3d}+\dfrac{3}{2}x_3^dq_2+\dfrac{3}{8}(x_3^{d-2e}+2x_1)
\in \dfrac{3}{4}x_1+k[q_2,x_3]=\dfrac{3}{4}x_1+k[g_2,x_3].
$$
Then, we have 
$(g_1^{\w },g_2^{\w })=(x_3^{pd},x_3^{2d})$ 
and 
$G:=(g_1,g_2,x_3)\in \T _3(k)$, since 
$G\sim (\frac{3}{4}x_1,g_2,x_3)\sim (\frac{3}{4}x_1,x_2,x_3)$. 
Now, observe that 
$\phi =q_1^2-q_2^p+2x_3^{pd}q_1-2g_2^{(p+3)/2}
=q_1^2-q_2^p+2(\frac{3}{8}h-q)$, 
where 
\begin{gather*}
h:=x_3^{pd}(x_3^{d-2e}+2x_1)-\underline{x_3^{(p-1)d}q_2^2}
=-x_3^{(p-1)d}(x_1^2x_3^{2e}+x_2^2+2x_1x_2x_3^e
+2x_2x_3^{d-e}), \\
q:=g_2^{(p+3)/2}
-x_3^{(p+3)d}-\dfrac{3}{2}x_3^{(p+1)d}q_2
-\dfrac{3}{8}\underline{x_3^{(p-1)d}q_2^2}
=\sum _{i=3}^{(p+3)/2}
\binom{\frac{p+3}{2}}{i}x_3^{(p+3-2i)d}q_2^i. 
\end{gather*}
This shows that $\deg \phi =pd-e+1$, 
since $p\ge 7$, 
$\deg q_1=3d$, $\deg q_2=d-e$, 
$\deg h=pd-e+1$, 
and 
$\deg q=\deg x_3^{(p-3)d}q_2^3=pd-3e$. 
Hence, by Lemma~\ref{lem:typeI}, 
$G$ and $\phi $ construct an SU pair with $s=p$. 
Since $\deg (q_1^2-g_2^3)<5d-e$, 
replacing $\phi $ with $\phi -g_2^3$ 
yields a similar example for $p\ge 5$.

\end{document}